\documentclass[a4paper,10pt,twoside]{article}

\usepackage[margin=1.5in, marginparwidth=1.2in]{geometry}
\usepackage[numbers]{natbib}
\usepackage[utf8]{inputenc}

\usepackage{amsmath}  
\usepackage{mathrsfs}
\usepackage{amsthm} 
\usepackage{amsfonts}
\usepackage{amssymb, latexsym}
\usepackage{mhequ}
\usepackage{verbatim}

\usepackage{kerkis}

\usepackage{graphicx}
\usepackage{color}
\usepackage{xcolor}
\usepackage{tikz}
\usetikzlibrary{calc,intersections,through,backgrounds}
\usetikzlibrary{decorations.pathreplacing}
\usetikzlibrary{shapes}

\usepackage{float} 
\usepackage{enumerate}

\usepackage[pdfauthor={}, pdftitle={}, pdftex]{hyperref}
\hypersetup{
  colorlinks=false,
  pdfborder={0 0 0}
  }

\usepackage{tocloft}
\usepackage{fancyhdr}

\usepackage[]{appendix}

\usepackage{hyperref}

\hypersetup{
    colorlinks=true,       
        linkcolor=blue,          
    citecolor=red,        
    filecolor=blue,     
    urlcolor=cyan,
}

\usepackage[colorinlistoftodos,prependcaption,color=yellow,textsize=tiny]{todonotes}

\definecolor{darkred}{rgb}{0.5,0.0,0.0}

\theoremstyle{plain}
\newtheorem{theorem}{Theorem}[section]
\newtheorem*{theorem*}{Theorem}

\newtheorem{corollary}[theorem]{Corollary}
\newtheorem{lemma}[theorem]{Lemma}

\theoremstyle{definition}
\newtheorem{definition}[theorem]{Definition}

\theoremstyle{remark}
\newtheorem{remark}[theorem]{Remark}

\theoremstyle{example}
\newtheorem{example}[theorem]{Example}

\numberwithin{equation}{section}

\title{Higher Order Fluctuation Expansions for Nonlinear Stochastic Heat Equations in Singular Limits}
\author{Benjamin Gess, Zhengyan Wu, Rangrang Zhang}
\date{}

\fancypagestyle{plain}{

  \fancyhf{}
  \fancyfoot[L]{\footnotesize \today}
  \fancyfoot[C]{\thepage}
}

\renewenvironment{abstract}
{
\begin{minipage}{.9\textwidth}\small\textbf{Abstract}.\noindent
}
{
\end{minipage}
}

\newenvironment{keywords}
{
\begin{minipage}{.9\textwidth}\small\textbf{Keywords}:\noindent
}
{
\end{minipage}
}

\newenvironment{msc}
{
\begin{minipage}{.9\textwidth}\small\textbf{MSC 2010}:\noindent
}
{
\end{minipage}
}

\begin{document}

\maketitle

\begin{center}
\begin{abstract}
Higher order fluctuation expansions for stochastic heat equations (SHE) with nonlinear, non-conservative and conservative noise are obtained. These Edgeworth-type expansions describe the asymptotic behavior of solutions in suitable joint scaling regimes  of small noise intensity ($\varepsilon\rightarrow0$) and diverging singularity ($\delta\rightarrow0$). The results include both the case of the SHE with regular and irregular diffusion coefficients.  In particular, this includes the correlated Dawson-Watanabe and Dean-Kawasaki SPDEs, as well as SPDEs corresponding to the Fleming-Viot and symmetric simple exclusion processes.
\end{abstract}	
\end{center}

\bigskip


\smallskip

\begin{msc} 60H15 \end{msc}

\tableofcontents

\section{Introduction}
Nonlinear stochastic heat equations (SHE) with irregular diffusion coefficients appear as fluctuating continuum models for the density profiles of stochastic interacting particle systems. For example, motivated from the theory of fluctuating hydrodynamics (see \cite{LL87}, \cite{S12}) and the fluctuation-dissipation relation (see \cite[Appendix A.3]{GGLS}, \cite{Ottinger}), in \cite{GG}, the following SPDE with conservative noise for the symmetric simple exclusion process (SSEP) has been introduced
\begin{equation}\label{SHE0.0}
\partial_tu^{\varepsilon}=\Delta u^{\varepsilon}+\varepsilon^{\frac{1}{2}}\nabla\cdot(\sqrt{u^{\varepsilon}(1-u^{\varepsilon})}\xi^{\varepsilon}),
\end{equation}
where $\xi^{\varepsilon}$ is a noise that is white in time and correlated in space with correlation length $\varepsilon>0$. Therefore, we are typically facing simultaneous scaling limits of small noise intensity ($\varepsilon^{\frac{1}{2}} \to 0$) and noise $\xi^{\varepsilon}$ converging to space time white noise $\xi$.

In \cite{DFG}, it has been shown that the central limit fluctuations and large deviations of solutions to 
\begin{equation}\label{SHE0.0-1}
\partial_tu^{\varepsilon,\delta}=\Delta u^{\varepsilon,\delta}+\varepsilon^{\frac{1}{2}}\nabla\cdot(\sqrt{u^{\varepsilon,\delta}(1-u^{\varepsilon,\delta})}\circ\xi^{\delta}),
\end{equation}
in appropriate joint scaling limits $(\varepsilon,\delta(\varepsilon))\to (0,0)$ are identical to those of the SSEP. Precisely, in \cite{DFG} the following first order expansion in the noise intensity $\varepsilon^{\frac{1}{2}}$ is shown, 
\begin{equation}\label{fluctuaionansatz-spde-intro}
   u^{\varepsilon,\delta(\varepsilon)}=\bar{u} + \varepsilon^{\frac{1}{2}} \bar{u}^{1,\delta(\varepsilon)} + o(\varepsilon^{\frac{1}{2}}),
\end{equation}
where $\bar{u}$ solves the (deterministic) heat equation and $\bar{u}^{1,\delta(\varepsilon)} \to \bar{u}^1$ for a Gaussian $\bar u^1$. 

Analogously, the central limit theorem for the empirical density field $\pi^\varepsilon$ of the SSEP derived in \cite{RK} for $d\geq2$ can be interpreted as the first order expansion
\begin{equation}\label{fluctuaionansatz-intro}
\pi^\varepsilon=\bar{u}+\varepsilon^{\frac{1}{2}}\bar{u}^{1,\delta(\varepsilon)}+o(\varepsilon^{\frac{1}{2}}), 
\end{equation}
where $\varepsilon = \frac{1}{N}$ with $N$ being the particle number, and $\bar{u}^{1,\varepsilon} \to \bar u^1$ with the same Gaussian limit $\bar u^1$ as in \eqref{fluctuaionansatz-spde-intro}. As argued in \cite{DFG}, this implies that $u^{\varepsilon,\delta(\varepsilon)}$ offers an improved order of approximation of the SSEP $\pi^\varepsilon$ in the sense that 
\begin{equation}\label{fluctuaionansatz-2}
 d(\pi^\varepsilon,u^{\varepsilon,\delta(\varepsilon)}) = o(\varepsilon^{\frac{1}{2}}),
\end{equation}
compared to the hydrodynamic limit $d(\pi^\varepsilon,\bar{u}) \approx \varepsilon^{\frac{1}{2}}.$ For further details see Section \ref{ssec:applications} below. 

Going beyond the order of approximation $o(\varepsilon^{\frac{1}{2}})$ in \eqref{fluctuaionansatz-2} requires the derivation of higher order small noise expansions than \eqref{fluctuaionansatz-spde-intro} and \eqref{fluctuaionansatz-intro}. Motivated from this, in the present paper, we prove higher order small noise expansions for nonlinear SHEs, both with non-conservative noise
\begin{equation}\label{SHE02}
\partial_tu^{\varepsilon,\delta}=\Delta u^{\varepsilon,\delta}+\varepsilon^{\frac{1}{2}}G(u^{\varepsilon,\delta})\xi^{\delta},\ \ u^{\varepsilon,\delta}(0)=u_0, 
\end{equation}
and conservative noise\footnote{With a small abuse of notation,
 $\xi^{\delta}$ in equation (\ref{SHE02}) refers to a scalar Gaussian field, while $\xi^{\delta}$ in equation (\ref{SHE01}) represents $d$-dimensional vector-valued Gaussian field.}
\begin{equation}\label{SHE01}
\partial_tu^{\varepsilon,\delta}=\Delta u^{\varepsilon,\delta}+\varepsilon^{\frac{1}{2}}\nabla\cdot(G(u^{\varepsilon,\delta})\xi^{\delta}),\ \ u^{\varepsilon,\delta}(0)=u_0,
\end{equation}
where $\xi^{\delta}=\xi\ast\eta_{\delta}$ is white in time and correlated in space, with $\eta_{\delta}$ specified in (\ref{etadelta-1}) below.

 Notably, with an eye on SPDEs arising from fluctuations in conservative and non-conservative particle systems, we include the case of irregular diffusion coefficients $G$, see, for example \eqref{SHE0.0-1}. Since fluctuation expansions rely on expansions of the nonlinearities in \eqref{SHE02}, resp.~\eqref{SHE01}, this causes additional challenges in the proof.

In a suitable scaling regime $(\varepsilon,\delta(\varepsilon))$ with $\delta(\varepsilon)\rightarrow0$ as $\varepsilon\rightarrow0$,
the higher order fluctuation (expansion) for (\ref{SHE02}) and (\ref{SHE01}) can be written in the unified form
\begin{equation}\label{expansion-2}
u^{\varepsilon,\delta(\varepsilon)}=\bar{u}^{0,\delta(\varepsilon)}+\varepsilon^{\frac{1}{2}}\bar{u}^{1,\delta(\varepsilon)}+\varepsilon^\frac{2}{2}\bar{u}^{2,\delta(\varepsilon)}
+\dots+\varepsilon^\frac{n}{2}\bar{u}^{n,\delta(\varepsilon)}+o(\varepsilon^\frac{n}{2}),
\end{equation}
where $\bar{u}^{0,\delta}$ is a solution to the (deterministic) heat equation \footnote{While $\bar{u}^{0,\delta}$ is independent of $\delta$ we choose this notation for the sake of consistency.}
\begin{equation}\label{HE}
  \partial_t\bar{u}^{0,\delta}=\Delta\bar{u}^{0,\delta},\ \ \bar{u}^{0,\delta}(0)=u_0,  
\end{equation}
and the coefficients $\bar{u}^{k,\delta}$, $k\geq 1$ satisfy separate equations in the case of non-conservative and conservative noise: For the case of non-conservative noise, $\bar{u}^{k,\delta}$ is iteratively defined as the mild solution to 
\begin{equation}\label{coe}
\partial_t\bar{u}^{k,\delta}=\Delta\bar{u}^{k,\delta}+\Big[\sum_{l=0}^{k-1}\frac{1}{l!}G^{(l)}(\bar{u}^{0,\delta})\mathcal{J}^{\delta}(k,l)\Big]\xi^{\delta},\quad
\bar{u}^{k,\delta}(0)=0,
\end{equation}
where $k\geq1$, $G^{(l)}(\cdot),l\in\mathbb{N}$ is the $l-$th derivative of $G$, $G^{(0)}(\cdot)=G$, and $\mathcal{J}^{\delta}(k,l)$ is given by
\begin{equation}\label{Jkl}
\mathcal{J}^{\delta}(k,l)=\sum_{(q_1,\dots,q_{k-l})\in\Lambda(k,l)}(\frac{l!}{q_1!\dots q_{k-l}!})\prod_{1\leq i\leq k-l}(\bar{u}^{i,\delta})^{q_i}.
\end{equation}
For $d=1$, we use the convention $u^{\varepsilon}=u^{\varepsilon,0}$, $\bar{u}^k=\bar{u}^{k,0}$.

Here, $\Lambda(k,l), k>l$ is the set of all integer solutions $(q_1, \dots, q_{k-l})$, $q_i\geq0$, $i=1,\dots,k-l$ satisfying
\begin{eqnarray}
\left\{
  \begin{array}{ll}\label{inteq}
   &q_1+q_2+\dots+q_{k-l}=l,\\
   &q_1+2q_2+\dots+(k-l)q_{k-l}=k-1.
  \end{array}
\right.
\end{eqnarray}
Moreover, if $\Lambda(k,l)=\emptyset$, then we set $\mathcal{J}^{\delta}(k,l)=0$. For example, $\mathcal{J}^{\delta}(k,0)=0$, for $k\geq1$. 

For notational efficiency we define 
\begin{eqnarray}\label{Kdelta}
K_i(\delta,d)=\left\{
  \begin{array}{ll}
   CI_{\{d=1\}} + \log(1/\delta)I_{\{d=2\}}+\delta^{-d+2}I_{\{d\geq3\}} & {\rm{if}}\ i=1,\ d\geq1,\\
   \delta^{-d} &{\rm{if}}\ i=2,\ d\geq1,
  \end{array}
\right.
\end{eqnarray}
corresponding to the speed of blow up due to the singularity of \eqref{SHE02}, \eqref{SHE01} in the limit $\delta\to0$. Here, $i=1$ corresponds to  the non-conservative case, and $i=2$ to the conservative case, and $C>0$ is a constant.  

\begin{theorem}[Non-conservative noise]\label{main-1}
Let $n\in\mathbb{N}=\{0,1,2,3,...\}$, $\varepsilon,\delta>0$. Assume that $G$ is regular with bounded $n$-th order derivatives, that is, $G^{(l)}(\cdot)\in C_b(\mathbb{R})$, for all $0\leq l\leq n$. Let $u^{\varepsilon,\delta}$ be the mild solution of (\ref{SHE02}) with initial data $u_0\in L^{\infty}(\mathbb{T}^d)$, $\bar{u}^{k,\delta}$ be the mild solution of (\ref{coe}), $k=1,...,n$. Assume that $\lim_{\varepsilon\rightarrow0}\varepsilon K_1(\delta(\varepsilon),d)^{(n+1)}=0.$ Then, for every $p\in[1,+\infty)$, there is a constant $C=C(G,p,n)$, such that, for every $\varepsilon>0$,  
\begin{align}\label{eq-22b}
\sup_{t\in[0,T],x\in\mathbb{T}^d}\mathbb{E}\Big|\Big(u^{\varepsilon,\delta(\varepsilon)}-\sum_{j=0}^{n}\varepsilon^\frac{j}{2}\bar{u}^{j,\delta(\varepsilon)}\Big)(t,x)\Big|^p\leq C\varepsilon^{\frac{np}{2}}\Big(\varepsilon K_1(\delta(\varepsilon),d)^{(n+1)}\Big)^{\frac{p}{2}}.
\end{align}

\end{theorem} 

Informally, this result is optimal in the sense that the exponents in \eqref{eq-22b} are consistent with the optimal regularity of the stochastic heat equation: For example, for $n=0$, and  $w^{\varepsilon,\delta(\varepsilon)}_0:=u^{\varepsilon,\delta(\varepsilon)}-\bar{u}^{0,\delta(\varepsilon)}$ we have 
\begin{align*}
  dw^{\varepsilon,\delta(\varepsilon)}_0 
  &= \Delta w^{\varepsilon,\delta(\varepsilon)}_0+\varepsilon^\frac{1}{2} G(u^{\varepsilon,\delta(\varepsilon)})\xi^{\delta(\varepsilon)} 
  \approx \Delta w^{\varepsilon,\delta(\varepsilon)}_0+\varepsilon^\frac{1}{2} G(\bar u^{0})\xi^{\delta(\varepsilon)}.
\end{align*}
Hence, $w^{\varepsilon,\delta(\varepsilon)}_0$ is at best as good as the stochastic heat equation, that is, see Section \ref{sec-3},
 $$\mathbb{E}|w^{\varepsilon,\delta(\varepsilon)}_0(t,x)|^2 \approx \varepsilon K_1(\delta(\varepsilon),d(\varepsilon)).$$
Analogously, for $n=1$, and  $w^{\varepsilon,\delta(\varepsilon)}_1:=\varepsilon^{-\frac{1}{2}}(u^{\varepsilon,\delta(\varepsilon)}-\bar{u}^{0,\delta(\varepsilon)})-\bar{u}^{1,\delta}$ we have that 
\begin{align*}
\partial_tw^{\varepsilon,\delta(\varepsilon)}_1=&\Delta w^{\varepsilon,\delta(\varepsilon)}_1+(G(u^{\varepsilon,\delta(\varepsilon)})-G(\bar{u}^{0,\delta(\varepsilon)}))\xi^{\delta(\varepsilon)}
\approx&\Delta w^{\varepsilon,\delta(\varepsilon)}_1+G'(\bar{u}^{0,\delta(\varepsilon)})w^{\varepsilon,\delta(\varepsilon)}_0\xi^{\delta(\varepsilon)},
\end{align*}
and $w^{\varepsilon,\delta(\varepsilon)}_1$ is of the size of the solutions to the stochastic heat equation with noise coefficient of order $\mathbb{E}|w^{\varepsilon,\delta(\varepsilon)}_0(t,x)|^2 = \varepsilon K_1(\delta(\varepsilon),d(\varepsilon))$. This yields
$\mathbb{E}|w^{\varepsilon,\delta(\varepsilon)}_1(t,x)|^2 \approx \varepsilon K_1(\delta(\varepsilon),d(\varepsilon))^{2}$, which is consistent with the exponents in (\ref{eq-22b}).

  A key challenge in the proof is the singularity of the stochastic PDEs considered here:  For $\delta=0$, the non-conservative SHE with $d\geq2$ and the conservative SHE with $d\geq1$ are super-critical in the sense of singular SPDEs. Therefore, treating (\ref{expansion-2}) with $\delta(\varepsilon)=0$ is currently out of reach for two main reasons: (1) Neither methods from classical It\^o calculus nor the theory of regularity structures \cite{H14}, para-controlled distributions \cite{GIP15} are applicable to super-critical, singular SPDE. More precisely, due to the low regularity of space time white noise it is unclear how to give meaning to the terms `$G(u^{\varepsilon})\xi$' and `$\nabla\cdot(G(u^{\varepsilon})\xi)$'. (2) For $\delta=0$ it is unclear how to give meaning to the coefficients $\{\bar{u}^{k,\delta}\}_{k\geq2}$ on the righthand side of (\ref{expansion-2}), even in a renormalized sense.  As a consequence, at the current state of theory, the  spatial regularization $\xi_\delta$ in the joint scaling regimes $(\varepsilon,\delta(\varepsilon))$ encountered above appear to be necessary.
 
 This shifts the focus to the relative scaling regimes with respect to $(\varepsilon,\delta)$ that can be treated. In the case of nonconservative noise, $L^p$ norms of the mollified space-time white noise $\xi_\delta$ diverge like $\delta^{-d}$, leading to simplest relative scaling regime $\varepsilon = o(\delta^{-d})$. However, it should be possible to exploit the regularizing effect of the viscosity present in \eqref{SHE02} to relax this assumption. Achieving this requires to use simultaneously the regularization of the noise by mollification, and the regularization of the noise by the heat semigroup in setting of multiplicative noise, like
 \begin{equation}\label{sto-convolution-intro}
    du=\Delta udt+\varepsilon^{\frac{1}{2}}g(t,x)\xi_\delta.
 \end{equation}
 While the well-established theory of maximal $L^p$-regularity for stochastic convolutions can be employed to obtain optimal estimates for \eqref{sto-convolution-intro} with respect to \textit{either} the regularization by the heat semigroup or  the regularization by mollification, it does not seem to be applicable in order to simutaneously exploit both. This is resolved in the present work by moving to less standard pointwise estimates, as in \eqref{eq-22b}, and by deriving heat kernel estimates manually.
 
 A second challenge addressed in the proof results from the Taylor expansions of the diffusion coefficients $G$. These result in combinatorial challenges with respect to the various occurring indices, as they already becomes apparent in \eqref{Jkl}. 

For the case of conservative noise\footnote{With a slight abuse of notation, we use the same notation $u^{\varepsilon,\delta}, \bar{u}^{k,\delta}$  for the cases of non-conservative and  conservative noise throughout the paper.} and $k\geq1$, let $\bar{u}^{k,\delta}$ be the mild solution to
\begin{equation}\label{ccoe}
\partial_t\bar{u}^{k,\delta}=\Delta\bar{u}^{k,\delta}+\nabla\cdot\Big(\Big[\sum_{l=0}^{k-1}\frac{1}{l!}G^{(l)}(\bar{u}^{0,\delta})\mathcal{J}^{\delta}(k,l)\Big]\xi^{\delta}\Big),
\quad \bar{u}^{k,\delta}(0)=0,
\end{equation}
where $\mathcal{J}^{\delta}(k,l)$ is defined by (\ref{Jkl}). For $k=0$, $\bar{u}^{0,\delta}$ is defined by (\ref{HE}).

\begin{theorem}[Conservative noise]\label{main-2}
Let $d\geq1$, and $n\in\mathbb{N}$. Assume that $G$ is regular with bounded $n$-order derivatives, that is, $G^{(l)}(\cdot)\in C_b(\mathbb{R})$, for all $0\leq l\leq n$. Let $\delta>0$, $\varepsilon_0=2\delta^d\|G^{(1)}(\cdot)\|_{L^{\infty}(\mathbb{R})}^{-2}>0$. For every $\varepsilon\in(0,\varepsilon_0)$, let $u^{\varepsilon,\delta}$ be the mild solution of (\ref{SHE01}) with initial data $u_0\in L^{\infty}(\mathbb{T}^d)$, $\bar{u}^{k,\delta}$ be the mild solution of (\ref{ccoe}), $k=1,...,n$.  Assume that $(\varepsilon,\delta(\varepsilon))$ satisfies $\lim_{\varepsilon\rightarrow0}\varepsilon K_2(\delta(\varepsilon),d)^{(n+1)}=0$. Then, for every $p\in[1,+\infty)$, there is a constant $C=C(G,p,n)$, such that, for every $\varepsilon>0$, 
\begin{align}\label{eq-22c}
\mathbb{E}\Big\|u^{\varepsilon,\delta(\varepsilon)}-\sum_{j=0}^{n}\varepsilon^\frac{j}{2}\bar{u}^{j,\delta(\varepsilon)}\Big\|_{L^p([0,T]\times\mathbb{T}^d)}^p\leq C\varepsilon^{\frac{np}{2}}\Big(\varepsilon K_2(\delta(\varepsilon),d)^{(n+1)}\Big)^{\frac{p}{2}}.
\end{align}

\end{theorem}

Compared to the case of nonconservative noise, we establish the higher order fluctuations expansion for conservative SHE in the $L^p([0,T]\times\mathbb{T}^d)$-norm, as opposed to taking the supremum over $t$ and $x$ outside of the expectation as in \eqref{eq-22b}. The essential difference to the case of nonconservative noise is that the regularization offered by the heat semigroup is entirely used to compensate the gradient in front of the noise in \eqref{SHE01}, so that no interaction between the heat semigroup and the singularity in $\delta$ appears. For this reason, in the conservative case, stochastic maximal $L^p$-regularity (see \cite{krylovlp, JML}) can be used.

We next address the case of irregular coefficients $G$. In this case, we face two additional problems: (1) The well-posedness of (\ref{SHE02}) and (\ref{SHE01}) are challenging problems. In the case of non-conservative noise, this has been addressed in  \cite{LEA, ML, CLE}. In the case of conservative Stratonovich noise, a well-posedness theory has been developed in \cite{FG21}. In contrast, the case of conservative It\^o noise with irregular coefficients $G$, as considered in this work, remained an open problem. Due to the resulting It\^o correction terms, this case lacks a stochastic coercivity condition, leading to problems to even construct solutions. (2) higher order singular expansions rely on taking derivatives of coefficients. For the case of irregular diffusion coefficients,  this is impossible at the points of their irregularity.

To address these problems, we develop a new local in time well-posedness approach to \eqref{SHE01}: Notably, if the initial data of (\ref{SHE01}) is bounded away from the singularities of $G$, one expects that locally in time the solution stays away from the singularities as well. As long as this is the case, $G$ behaves as a Lipschitz continuous function, which ensures the local in time well-posedness of (\ref{SHE01}). A key part in making this idea rigorous is to show that the resulting time interval is nontrivial. This is achieved in the present work by deriving new regularity estimates on the solutions. More precisely, a novel $L^{\infty}([0,T]\times\mathbb{T}^d)$-estimate is obtained by Moser iteration. As a consequence, we obtain the local-in-time existence and uniqueness of a solution $(u^{\varepsilon,\delta},\tau^{\varepsilon,\delta}_{\gamma})$ to (\ref{SHE02}), resp.~(\ref{SHE01}).

For convenience, we define the extension 
\begin{equation}\label{extension-main}
u^{\varepsilon,\delta}(t)=u^{\varepsilon,\delta}(\tau^{\varepsilon,\delta}_{\gamma}),\ \  \text{if}\  t\in[\tau^{\varepsilon,\delta}_{\gamma},T],
\end{equation}
after the stopping time
\begin{equation}\begin{split}\label{stop-intro}
\tau^{\varepsilon,\delta}_{\gamma}:=\inf\Big\{t\in[0,T];~&{\rm{ess\sup}}_{x\in\mathbb{T}^d}u^{\varepsilon,\delta}(t,x)>{\rm{ess\sup}u_0}+\gamma,\\
&\text{or }{\rm{ess\inf}}_{x\in\mathbb{T}^d}u^{\varepsilon,\delta}(t,x)<{\rm{ess\inf}} u_0-\gamma\Big\}.
\end{split}\end{equation}

The following result provides the high order fluctuation expansion for the local in time solution $u^{\varepsilon,\delta}$.   

\begin{theorem}\label{sbm2}
Assume that Hypothesis H3 holds for $u_0$, $G$, for some $\gamma>0$. Let $n\in\mathbb{N}$, and $\varepsilon,\delta>0$. Consider the local in time solution $(u^{\varepsilon,\delta},\tau^{\varepsilon,\delta}_{\gamma})$ of  (\ref{SHE02}) (resp. (\ref{SHE01})), where $\tau^{\varepsilon,\delta}_{\gamma}$ is defined by (\ref{stop-intro}).  
Then we have the following high order fluctuation expansion. 
\begin{description}
\item[(i)] ({\rm{Non-conservative\ noise}})\quad Assume that 
\begin{equation}\label{localscale-3}
\varepsilon(\delta(\varepsilon)^{-d}+K_1(\delta(\varepsilon),d)^{(n+1)})\rightarrow0 \quad {\rm{as\ \varepsilon\rightarrow0}},
\end{equation}

then, for almost every $(t,x)\in[0,T)\times\mathbb{T}^d$,
\begin{align*}
\varepsilon^{-\frac{n}{2}}\Big(u^{\varepsilon,\delta(\varepsilon)}-\sum_{i=0}^{n}\varepsilon^{\frac{i}{2}}\bar{u}^{i,\delta(\varepsilon)}\Big)(t,x)\rightarrow0,
\end{align*}
in probability, as $\varepsilon\rightarrow0$.

  \item[(ii)] ({\rm{Conservative\ noise}})\quad Assume that  
\begin{equation}\label{localscale-4}
\varepsilon(\delta(\varepsilon)^{-d-2}+K_2(\delta(\varepsilon),d)^{(n+1)})\rightarrow0\quad {\rm{as\ \varepsilon\rightarrow0}},
\end{equation}
then, for every $p\in[1,+\infty)$,  
\begin{align*}
\Big\|\varepsilon^{-\frac{n}{2}}\Big(u^{\varepsilon,\delta(\varepsilon)}-\sum_{i=0}^{n}\varepsilon^{\frac{i}{2}}\bar{u}^{i,\delta(\varepsilon)}\Big)\Big\|_{L^p([0,T]\times\mathbb{T}^d)}\rightarrow0,
\end{align*}
in probability, as $\varepsilon\rightarrow0$.
\end{description}
\end{theorem}

\subsection{Applications}\label{ssec:applications}

In this section, we further elaborate on the relation of the results of this work to several conservative and non-conservative interacting particle systems.

\subsubsection{Non-conservative models}

The following stepping stone model from population genetics is introduced in \cite{BD}. We track two alleles, $a$ and $A$, the population subject to these two alleles is divided into discrete demes, indexed by $i\in\mathbb{Z}^d$. Let $u_i(t)$ be the proportion of allele $a$ in the $i$th deme, which is governed by the following system of SDEs
\begin{align}\label{ctssm0}
du_i(t)=&D\sum_{j\in\mathbb{Z}^d}m_{ij}(u_j(t)-u_i(t))dt-\mu_1 u_i(t)dt\notag\\
+&\mu_2(1-u_i(t))dt+\sqrt{\gamma u_i(t)(1-u_i(t))}dB_i(t).
\end{align}
Here, $\mu_1$ is the mutation rate from type $a$ to $A$ and $\mu_2$ is the mutation rate from type $A$ to $a$, $\{B_i\}_{i\in\mathbb{Z}^d}$ are independent Brownian motions and $\{m_{ij}\}_{i,j\in\mathbb{Z}^d}$ are migration rates from deme $j$ to deme $i$. The constants $D$ and $\gamma$ will be addressed below. See \cite{BD} for more details on this model.

In the following, we introduce the structured coalescent process, which will play a role in making some calculations for the stepping stone model. The structured coalescent takes values in ${(n_i)_{i\in\mathbb{Z}^d}: n_i\in\mathbb{N}}$, and its dynamics are stated as follows. Migration from site $i$ to $j$ decreases the number of individuals at site $i$ by one, $n_i\rightarrow n_i-1$, and increases the number of individuals at site $j$ by one, $n_j\rightarrow n_j+1$, at rate $Dn_im_{ij}$. Death at site $i$ decreases the number of individuals at site $i$ by one, $n_i\rightarrow n_i-1$, at rate $\mu_2n_i$. Coalescence at site $i$ decreases the number of individuals at site $i$ by one, $n_i\rightarrow n_i-1$, at rate $\frac{1}{2}\gamma n_i(n_i-1)$. The duality between the stepping stone model and the structured coalescent process is stated as below, for any $t>0$, 
\begin{align*}
\mathbb{E}(u_i(t)^{n_i(0)})=\mathbb{E}\Big[u_i(0)^{n_i(t)}\exp\Big\{-\int_0^t\mu_1\Big(\sum_{i\in\mathbb{Z}^d}n_i(s)\Big)ds\Big\}\Big],	\ \ i\in\mathbb{Z}^d, 
\end{align*}
where $(u_i)_{i\in\mathbb{Z}^d}$ is the solution of (\ref{ctssm0}), $n$ is the structured coalescent process. See \cite[Lemma 2.3]{BD} for more details and see \cite{Shiga} for the proof.

We track two individual samples started from deme $i$ and deme $j$, which are two random walks on $\mathbb{Z}^d$ denoted by $X(t)$ and $Y(t)$, respectively, with initial positions $X(0)=i$ and $Y(0)=j$, and with transition rates $\{m_{ij}\}_{i,j\in\mathbb{Z}^d}$ and coalescence rate $\gamma$ when they occupy the same position. According to the dynamics of the structured coalescent process above, the coalescence rate of $X$ and $Y$ depends on where they meet as well, denoted by $\phi(i,j)$, and the values of $n_{\phi(i,j)}$. 

Let $\tau_{\gamma}(i,j)$ be the stopping time when $X$ and $Y$ coalesce, more precisely,
\begin{equation*}
\tau_{\gamma}(i,j)=\inf\{t>0: X(t)=Y(t) \text{ and } X(t), Y(t) \text{ coalesce}\}.
\end{equation*}
One is then interested in approximating the law of $\tau_{\gamma}$ by approximating its Laplace transform 
\begin{equation}\label{laplace-time}
\mathbb{E}[e^{-2(\mu_1+\mu_2)\tau_{\gamma}(i,j)}]. 	
\end{equation}

In order to estimate (\ref{laplace-time}), stochastic PDE models are employed: The Fleming-Viot process
\begin{equation}\label{e-6}
\partial_t u=\frac{\sigma^2}{2}\Delta u-\mu_1u+\mu_2(1-u)+\sqrt{\gamma u(1-u)}\xi,
\end{equation}
is introduced as a fluctuating continuum model for the stepping stone model (\ref{ctssm0}), where $\xi$ is the space time white noise, see \cite[Lemma 2.7]{BD}. As indicated above, giving rigorous meaning to \eqref{e-6} is an open problem in $(d\geq2)$-spatial dimension, due to the irregularity of space time white noise. Therefore, in \cite{BD} a linear expansion of (\ref{e-6}) is derived, which leads to the (affine-)linear SPDE
\begin{equation*}
\partial_t\bar{u}^1=\Delta\bar{u}^1-\mu_1\bar{u}^1+\mu_2(1-\bar{u}^1)+\sqrt{\gamma\bar{u}(1-\bar{u})}\xi.
\end{equation*}
Here $\bar{u}$ is the constant solving
\begin{equation*}
\partial_t\bar{u}=\frac{\sigma^2}{2}\Delta^{(1)} \bar{u}-\mu_1\bar{u}+\mu_2(1-\bar{u}),
\end{equation*}
with constant initial data $\bar{u}(0)=\frac{\mu_2}{\mu_1+\mu_2}$, and $\Delta^{(1)}$ is the infinitesimal generator of the random walk with transition rates $\{m_{ij}\}$.

As pointed out in \cite{BD}, the covariance $\mathbb{E}(\bar{u}^1(t,x)\bar{u}^1(t,y))$ can be used to approximate (\ref{laplace-time}) when the population is in equilibrium. More precisely, this covariance is identified as the first term in the series expansion with respect to $\gamma$ of (\ref{laplace-time}). 

Motivated by the aim to derive a higher order approximation in $\gamma$ of (\ref{laplace-time}), we devote to deriving a higher order approximation of  (\ref{e-6}) when $\gamma\rightarrow0$. The latter is given by the expansion formula for non-conservative SHEs with irregular diffusion coefficients, which is derived in the present work.

\subsubsection{Conservative models}

Let $\eta^N$ be the SSEP with initial measure $\mu^N$ given by a smooth local equilibrium profile $u_0$, which means that $\mathbb{E}_{\mu^N}(\eta^N_0(x))=\mu^N(\eta^N_0(x)=1)=u_0(x/N)$, where $\mathbb{E}_{\mu^N}$ denotes the expectation with respect to $\mu^N$. For details see Ravishankar \cite{RK} and Kipnis and Landim \cite[Chapter 2]{KL99}. The corresponding empirical density $\pi^N$ is defined by
\begin{align*}
\pi^N(t,x)=\frac{1}{N^d}\sum_{x\in(\mathbb{Z}^d/N\mathbb{Z}^d)}\delta_{x/N}\eta_{N^2t}(x).
\end{align*}
Let $\bar{u}$ be the hydrodynamic limits of SSEP, that is, the solution to \eqref{HE}. In Ravishankar \cite{RK}, Kipnis, Varadhan \cite{KV2}, Rezakhanlou \cite{RF} the central limit fluctuations for non-equilibrium of SSEP has been proven. We refer to \cite{RK}, the result of the fluctuations therein shows that the law of the fluctuation density fields converges to a Gaussian
\begin{equation*}
\mathcal{L}\Big(N^{\frac{1}{2}}(\pi^N-\bar{u})\Big)\rightharpoonup\mathcal{L}(\bar{u}^1), \ \ N\rightarrow\infty,
\end{equation*}
where $\bar{u}^1$ solves
\begin{equation}\label{CLTssep}
\partial_t\bar{u}^1=\Delta\bar{u}^1+\nabla\cdot(\sqrt{\bar{u}(1-\bar{u})}\xi),\quad \bar{u}^1(0)=0,
\end{equation}
with $\xi$ being space time white noise. In other words, it is shown that
\begin{align*}
\pi^N=\bar{u}+N^{-\frac{1}{2}}\bar{u}^1+o(N^{-\frac{1}{2}}).
\end{align*}

The dynamical large deviation principle for the empirical density $\pi^N$ has been shown by Quastel, Rezakhanlou and Varadhan \cite{QRV} and Kipnis and Landim \cite{KL99}, with rate function given by 
\begin{equation}\label{ratefunc}
I_{u_0}(\pi):=\frac{1}{2}\inf\Big\{\int_0^T\int_{\mathbb{T}^d}|g|^2dxdt:\partial_tu=\Delta u+\nabla\cdot(\sqrt{u(1-u)}g), u(0)=u_0\Big\},
\end{equation}
for $\pi(dx)=udx$ and $I_{u_0}(\pi)<\infty$, where $u_0$ is the initial density profile of SSEP. 

In \cite{DFG}, the fluctuating continuum model 
\begin{equation}\label{GessSSEP}
\partial_t u^{\varepsilon}=\Delta u^{\varepsilon}+\varepsilon^{\frac{1}{2}}\nabla\cdot(\sqrt{u^{\varepsilon}(1-u^{\varepsilon})}\circ\xi^{\delta(\varepsilon)}) 
\end{equation} 
for SSEP has been proposed, which is a modification of  (\ref{SHE0.0}). Here, $\circ$ denotes the Stratonovich integral, and $\xi^{\delta}$ is a regularized space time white noise. In \cite{DFG} the large deviations for $u^{\varepsilon}$ are proved with the same rate function (\ref{ratefunc}) as for SSEP. In addition, \cite{DFG} proves the central limit theorem of $u^{\varepsilon}$, that is, $\varepsilon^{-\frac{1}{2}}(u^{\varepsilon}-\bar{u})\rightarrow\bar{u}^1$, where $\bar{u}^1$ is the same Gaussian as in \eqref{CLTssep}. This shows that fluctuations and rare events of SSEP can be predicted by the SPDE (\ref{GessSSEP}). The main results of the present paper go beyond this by establishing higher order fluctuation expansions.

\subsection{Comments on the literature}
Early works on small noise asymptotic expansions for SPDEs comprise  Albeverio, Di Persio, and Mastrogiacomo \cite{ADM}, and Albeverio, Mastrogiacomo, and Smii \cite{AMS} for stochastic reaction-diffusion equations. For related results, see also Albeverio and Smii \cite{AS} and Klosek, Matkowsky, and Schuss \cite{KMS}, where the authors consider small noise expansions for SDEs. Furthermore, Friz and Klose proved small noise expansions for the subcritical parabolic Anderson model in \cite{FK3}.

Non-conservative SHEs with irregular diffusion coefficients contain the correlated Fleming-Viot process and the Dawson-Watanabe equation. The Fleming-Viot process was introduced by Fleming and Viot in \cite{FV} in population genetics. See also Biswas, Etheridge, and Klimek \cite{BEK}, Champagnat and Villemonais \cite{CV}, Kouritzin and L\^{e} \cite{KL}, and Altomare, Cappelletti, and Leonessa \cite{ACL}, among others. For various studies of the Dawson-Watanabe equation, we refer to the works of Mandler and Overbeck \cite{MO}, L\^{e} \cite{LK}, Kallenberg \cite{KO}, Chen, Ren, and Wang \cite{CRW}, and the additional references provided in these sources. For the central limit fluctuations of the Fleming-Viot process and the Dawson-Watanabe equation, see, for example, C\'{e}rou, Delyon, Guyader, and Rousset \cite{CDGR}.

Conservative-type SHEs with irregular coefficients appear as fluctuating continuum models in a variety of settings, including the SSEP and the Dean-Kawasaki equation. The Dean-Kawasaki equation was proposed by Dean \cite{DD} and Kawasaki \cite{KK}. The existence of martingale solutions of the Dean-Kawasaki equation has been obtained by Konarovskyi, Lehmann and von Renesse \cite{KLV}. Regarding the well-posedness of the correlated Dean-Kawasaki equation, see Fehrman and Gess \cite{FG21}. Later on, a large deviation principle is obtained by \cite{FG22} and a central limits theorem is obtained by \cite{CF23}.  For the weak error estimates between the Dean-Kawasaki equation and interacting mean-field systems, see Cornalba and Fischer \cite{CF}, Cornalba, Fischer, Ingmanns, and Raithel \cite{CFIR}, and Djurdjevac, Kremp, and Perkowski \cite{DKP}. Further literature on the Dean-Kawasaki equation, we refer readers to Andres and von Renesse \cite{AV}, Konarovskyi and von Renesse \cite{KV}, and Konarovskyi, Lehmann, and von Renesse \cite{KLV2}. Regarding to the fluctuations of the symmetric simple exclusion process, we refer to Ravishankar \cite{RK}, Kipnis and Varadhan \cite{KV2}, Rezakhanlou \cite{RF}, and the references therein.

\subsection{Structure of the paper}
The structure of this paper is organized as follows. In Section \ref{sec-2}, we introduce notations and the precise framework of this work. In Section \ref{sec-3}, Two technical lemmas on estimates of the heat semigroup are presented. In Section \ref{sec-4}, the global in time well-posedness of SHE with smooth diffusion coefficients and the local in time well-posedness for the case of non-smooth diffusion coefficients  will be provided.  In Section \ref{sec-5}, we analyze the speed of divergence of the coefficients $\bar{u}^{k,\delta}$. In Section \ref{sec-6}, we prove the main results for the case of smooth diffusion coefficients. The small noise expansion for irregular coefficients is proven in Section \ref{sec-7}. In Section \ref{sec-8}, we apply the main results to the Fleming-Viot process, to the SSEP, the Dawson-Watanabe equation, and the Dean-Kawasaki equation.

\section{Preliminaries}\label{sec-2}

\subsection{Notations}

We fix $T>0$ in the whole paper, and let $(\Omega,\mathcal{F},\mathbb{P},\{\mathcal{F}_t\}_{t\in
[0,T]})$ be a standard filtered probability space, with  $\mathbb{E}$ being the corresponding expectation. Let $\mathbb{N}_+$ be the set of all positive natural numbers. For any $d\in\mathbb{N}_+$, let $\mathbb{T}^d$ denote the $d-$dimensional torus, with convention  $\mathbb{T}^d=[-1/2,1/2]^d$. The Laplace operator has the following eigenvalue decomposition 
\begin{equation*}
	\Delta e_{k,\theta}=-\alpha_ke_{k,\theta},\ k\geq0,\ \theta=1,2,
\end{equation*}
for some $0=\alpha_0<...<\alpha_k<...\rightarrow+\infty$ and  $\{e_{k,\theta}\}_{k\geq1,\theta=1,2}$ take the form $e_{k,1}(x)=\sqrt{2}\sin(2\pi\sqrt{\alpha_k}x)$, $e_{k,2}(x)=\sqrt{2}\cos(2\pi\sqrt{\alpha_k}x)$, $k\geq0$. By \cite[(2.2)]{CD17}, there exists $c>0$ such that  
\begin{equation}\label{eq:alpha_scaling}
c^{-1}k^{\frac{2}{d}}\leq\alpha_k\leq ck^{\frac{2}{d}},\ \ k\geq1. 	
\end{equation}
Let $W$ be a cylindrical Wiener process defined on the real-valued Lebesgue space $L^2(\mathbb{T}^d)$, given by $W(t)=\sum_{k\geq0,\theta=1,2}\beta_{k,\theta}(t) e_{k,\theta}$, $t\in[0,T]$, where $\{\beta_{k,\theta}(t)\}_{t\in[0,T]},k\geq0,\theta=1,2$ are independent $\{\mathcal{F}_t\}_{t\in [0,T]}-$Brownian motions. As mentioned in the introduction, in the case of non-conservative SHEs $\{\beta_{k,\theta}(t)\}_{t\in[0,T]},k\geq0,\theta=1,2$ are one-dimensional Brownian motions, while in the case of conservative SHEs $\{\beta_{k,\theta}(t)\}_{t\in[0,T]},k\geq0,\theta=1,2$ are $d$-dimensional Brownian motions. For $a,b\in\mathbb{R}$, the notation $a\lesssim b$ means that there exists a constant $C>0$, such that $a\leq Cb$.

\subsection{Assumptions}
For the initial data of the stochastic heat equation (\ref{SHE02}) and (\ref{SHE01}), we assume the following conditions. 
\begin{description}
  \item[Hypothesis H1] The initial data $u_0\in L^{\infty}(\mathbb{T}^d)$.
\end{description}
\begin{remark}
  The requirement of Hypothesis H1 is motivated by the regularity of the diffusion coefficient. For instance, the diffusion term in Fleming-Viot process is given by ``$\sqrt{u(1-u)}dW(t)$", which requires that the solution $u(t,x)\in[0,1]$, for every $(t,x)\in[0,T]\times\mathbb{T}^d$. Thus, we need to impose the condition $u_0(x)\in[0,1]$ for every $x\in\mathbb{T}^d$.
\end{remark}

In addition, we also need conditions on the derivatives of $G$ when deriving higher order small noise expansions for nonlinear stochastic heat equations.
\begin{description}
  \item[Hypothesis H2] The derivatives of $G$ satisfy
  \begin{eqnarray}\label{e-2}
    G^{(l)}(\cdot)\in C_b(\mathbb{R}),\ \forall l\geq 0.
  \end{eqnarray}
\end{description}
From Hypothesis H2, we deduce that for any $n\geq 1$, there exists a constant $K=K(n,G)$ such that
\begin{equation*}
\Big|G(a)-G(b)-\sum_{l=1}^{n-1}\frac{1}{l!}G^{(l)}(b)(a-b)^{l}\Big|\leq K|a-b|^n.
\end{equation*}

We note that neither the coefficient $G(u)=\sqrt{u}$ in the Dawson-Watanabe equation nor the coefficient $G(u)=\sqrt{u(1-u)}$ in the  Fleming-Viot process satisfies Hypothesis H2. In order to apply the results on higher order small noise expansions in these cases, we relax Hypothesis H2 to the following Hypothesis H3.
\begin{description}
  \item[Hypothesis H3]   The initial data $u_0\in L^{\infty}(\mathbb{T}^d)$, and there exists a constant $\gamma>0$, such that for every $l\geq 0$,
  \begin{eqnarray*}
    G^{(l)}\in C_b([{\rm{ess\inf}} u_0-\gamma,{\rm{ess\sup}} u_0+\gamma]).
  \end{eqnarray*}
\end{description}
In Section \ref{sec-7}, we apply the main results to several particle models with diffusion coefficient $G$ and initial data $u_0$ satisfying Hypothesis H3.

\subsection{The approximation of space time white noise}

In the following we fix once and for all an integer $n>(\frac{d-2}{4})\vee0$. For every $\delta>0$, let $\eta_{\delta}$ be the fundamental solution of 
\begin{align}\label{etadelta-1}
(I-\delta^2\Delta)^n\eta_{\delta}=\delta_0. 	
\end{align}
We denote the heat semigroup by $\{S(t)\}_{t\geq0}$. Thanks to the fact that the resolvent of the Laplace operator can be represented as the Laplace transform of the heat semi-group \cite[Chapter 1, Remark 5.4]{Pazy83}, it follows that 
\begin{align*}
\eta_{\delta}=(I-\delta^2\Delta)^{-n}\delta_0=\int_{\mathbb{R}_+^n}\frac{1}{\delta^{2n}}e^{-\frac{t_1+...+t_n}{\delta^2}}S(t_1)...S(t_n)\delta_0dt_1...dt_n\geq0,	
\end{align*}
This shows that the convolution kernel $\eta_{\delta}$ is nonnegative. Furthermore, we are able to see that 
\begin{align}\label{propertyofeta}
\langle\eta_{\delta},e_{j,\theta}\rangle\lesssim \frac{1}{1+(\delta^2\alpha_j)^n}, 	\ \ j\geq0,\ \theta=1,2. 
\end{align}
Indeed, by using the representation of the resolvent again, we find that 
\begin{align*}
\langle\eta_{\delta},e_{j,\theta}\rangle=&((I-\delta^2\Delta)^{-n}e_{j,\theta})(0)\\
=&\Big(\int_{\mathbb{R}_+^n}\frac{1}{\delta^{2n}}e^{-\frac{t_1+...+t_n}{\delta^2}}S(t_1+...+t_n)e_{j,\theta}dt_1...dt_n\Big)(0)\\
\leq&\Big(\int_{\mathbb{R}_+^n}\frac{1}{\delta^{2n}}e^{-\frac{t_1+...+t_n}{\delta^2}}e^{-(t_1+...+t_n)\alpha_j}dt_1...dt_n\Big)\\
=&(1-\delta^2\alpha_j)^n\lesssim\frac{1}{1+(\delta^2\alpha_j)^n}. 
\end{align*}
By rescaling we have $\eta_{\delta}(\cdot)=\frac{1}{\delta^d}\eta_1(\frac{\cdot}{\delta})$.

We next introduce a smooth spatial approximation of the infinite dimensional Brownian motion. We say $W$ is an $L^2(\mathbb{T}^d)$-cylindrical Wiener process, if it is a sequence of  bounded linear operator $W(t): L^2(\mathbb{T}^d)\rightarrow L^2(\Omega)$, $t\geq0$ such that 

\begin{description}
	\item {(i)} for all $t\geq0$, $h\in L^2(\mathbb{T}^d)$, $W(t)h$ is centred Gaussian and $\mathcal{F}(t)$-measurable;  
	\item {(ii)} we have that $\mathbb{E}(W(t_1)h_1\cdot W(t_2)h_2)=t_1\wedge t_2\langle h_1,h_2\rangle$, for all $t_1,t_2\geq0$, $h_1,h_2\in L^2(\mathbb{T}^d)$;
	\item {(iii)} for every $0\leq t_1\leq t_2$, $h\in L^2(\mathbb{T}^d)$, the random variable $W(t_2)h-W(t_1)h$ is independent of $\mathcal{F}(t_1)$.  
\end{description}
Furthermore, an $L^2(\mathbb{T}^d)$-cylindrical Winer process can be represented by 
\begin{align*}
	W(t)=\sum_{k\geq0,\theta=1,2}\beta_{k,\theta}(t)\langle e_{k,\theta},\cdot\rangle, 
\end{align*}
where $\beta_{k,\theta}(t)=W(t)e_{k,\theta}$. A spatially smooth Brownian motion $W^{\delta}$ is defined as follows. For all $x\in\mathbb{T}^d$, 
\begin{equation}\label{smooth-1}
W^{\delta}(t,x):=W(t)\eta_{\delta}(x-\cdot)=\sum_{k\geq0,\theta=1,2}\beta_{k,\theta}(t)(\eta_{\delta}\ast e_{k,\theta})(x),
\end{equation}
where $\langle\cdot,\cdot\rangle$ denotes the $L^2(\mathbb{T}^d)$-inner product. $W^{\delta}(t)$ is correlated in space with correlation length $\delta$, in the sense that for $t,s>0$, $x,y\in\mathbb{T}^d$, we have
\begin{align*}
\mathbb{E}(W^{\delta}(t,x)W^{\delta}(s,y))=(t\wedge s)R_{\delta}(x-y),	
\end{align*}
where $R_{\delta}(x-y)=0$, if $|x-y|>\delta$. Furthermore, we introduce a bilinear form induced by the covariance structure: For $f_1,f_2\in L^2(\mathbb{T}^d)$, set
\begin{equation}\label{innerproduct}
\langle f_1,f_2\rangle_{\delta}=\int_{\mathbb{T}^{2d}}f_1(x)f_2(y)R_{\delta}(x-y)dxdy,
\end{equation}
where $R_{\delta}$ can be calculated directly: 
\begin{eqnarray}\label{e-7}
R_{\delta}(x-y)=\int_{\mathbb{T}^d}\eta_{\delta}(x-y+z)\eta_{\delta}(z)dz.
\end{eqnarray}
By the definition of $R_{\delta}$, for any $f\in L^2(\mathbb{T}^d)$, we have
\begin{align*}
\langle f,f\rangle_{\delta}=\int_{\mathbb{T}^d}|(f\ast\eta_{\delta})(x)|^2dx\geq 0.
\end{align*}
As a consequence, 
\begin{equation}\label{norm}
\|f\|_{\delta}:=\langle f,f\rangle_{\delta}^{\frac{1}{2}}
\end{equation}
is a semi-norm. For every $h,g \in L^2(\mathbb{T}^d)$, $t,x\geq0$, we have the covariance structure  
\begin{align}\label{cov}
\mathbb{E}[\langle W^{\delta}(t),h\rangle\langle W^{\delta}(s),g\rangle]=&(t\wedge s)\langle h,g\rangle_{\delta}\notag\\
=&(t\wedge s)\int_{\mathbb{T}^{2d}}h(x)g(y)R_{\delta}(x-y)dxdy.
\end{align}
When $\delta=0$, $W^{\delta}(t)$ turns out to be the cylindrical Brownian motion.

Since $R_{\delta}$ is nonnegative,  we are able to  define the It\^o stochastic integration against $W^{\delta}(t)$ as in \cite{LR}.

Let $p$ represent the heat kernel on $\mathbb{T}^d$. Referring to \cite{Daprato}, mild solutions of (\ref{SHE02}) and (\ref{SHE01}) can be written as  
\begin{equation}\label{mild-0}
u^{\varepsilon,\delta}(t,x)=(p(t)\ast u_0)(x)+\varepsilon^{\frac{1}{2}}\int_0^t\langle p(t-s,x-\cdot),G(u^{\varepsilon,\delta}(s,\cdot))dW^{\delta}(s)\rangle,
\end{equation}
and
\begin{equation}\label{mild-0-cons}
u^{\varepsilon,\delta}(t,x)=(p(t)\ast u_0)(x)+\varepsilon^{\frac{1}{2}}\int_0^t\langle \nabla_x p(t-s,x-\cdot),G(u^{\varepsilon,\delta}(s,\cdot))dW^{\delta}(s)\rangle,
\end{equation}
respectively.  

For $d=1$, (\ref{SHE02}) is also well-posed when driving by space time white noise.  To ease notation, when considering the nonconservative case, $d=1$, we implicity understand $\delta=0$. 

\section{Maximal $L^p$-regularity and scaling regimes}\label{sec-3}

In this section, we provide two forms of $L^p$-estimates for stochastic convolutions: Let $g(t)$ be an $L^2(\mathbb{T}^d)$-progressively measurable process with 
\begin{equation}\label{Lpforg}
\mathbb{E}\|g\|_{L^p([0,T]\times\mathbb{T}^d)}^p<\infty,\ \ \text{for any }p\in[1,+\infty). 	
\end{equation}
For any $t\in[0,T]$, let $g_{\delta}(t):L^2(\mathbb{T}^d)\rightarrow L^2(\mathbb{T}^d)$  be the operator defined by $f\rightarrow g(t)(\eta_{\delta}\ast f)$. We consider the following two types of stochastic heat equations, 
\begin{equation}\label{sto-convolution-1}
du=\Delta udt+\varepsilon^{\frac{1}{2}}g_{\delta}(t)dW(t),\ \ u(0)=0,	
\end{equation}
and 
\begin{equation}\label{sto-convolution-2}
	du=\Delta udt+\varepsilon^{\frac{1}{2}}\nabla\cdot(g_{\delta}(t)dW(t)),\ \ u(0)=0, 
\end{equation}
corresponding to (\ref{SHE02}) and (\ref{SHE01}), respectively. Here, with  abuse of notation, $W$ in (\ref{sto-convolution-1}) denotes a scalar $L^2(\mathbb{T}^d)$-cylindrical process, while in (\ref{sto-convolution-2}), $W$ denotes a vector-valued $L^2(\mathbb{T}^d)$-cylindrical process. The mild solution to (\ref{sto-convolution-1}) and (\ref{sto-convolution-2}) are given by stochastic convolutions 
\begin{equation}\label{sto-convolution-3}
	u(t)=\varepsilon^{\frac{1}{2}}\int^t_0S(t-s)g_{\delta}(s)dW(s), 
\end{equation}
and 
\begin{equation}\label{sto-convolution-4}
	u(t)=\varepsilon^{\frac{1}{2}}\int^t_0\nabla S(t-s)g_{\delta}(s)dW(s),  
\end{equation}
respectively. The main aim of this section is the derivation of regularity estimates for these stochastic convolutions. As indicated in the introduction to this work, the cases of conservative and nonconservative noise are of distinct difficulty, and, thus, are treated separately.

\subsection{The case of nonconservative noise}\label{subsec-4-1}

The derivation of the expansion formula \eqref{expansion-2} requires estimates on the speed of divergence of the expansion coefficients $\bar{u}^{k,\delta}$ in \eqref{coe} and \eqref{ccoe} as $\delta\rightarrow0$. These will be derived in the present section.

For $x\in\mathbb{T}^d$, $t\in  [0,T]$, set
\begin{equation}\label{K2i}
K_{\delta}(t):=\int_0^t\int_{\mathbb{T}^{2d}}p(t-s,x-y_1)p(t-s,x-y_2)R_{\delta}(y_1-y_2)dy_1dy_2ds,
\end{equation}
where $p$ is the heat kernel. We note that the righthand side of (\ref{K2i}) does not depend on $x$. For $d=1$, it turns to the case of $\delta=0$ automatically, then $K_{\delta}(t)=\|p(t-\cdot,\cdot)\|_{L^2([0,t]\times\mathbb{T}^d)}^2<\infty$, which means the integration of (\ref{K2i}) does not blow up. Thus we only estimate the divergence speed of (\ref{K2i}) as the parameter $\delta\rightarrow0$ when $d\geq2$. 

\begin{lemma}\label{K2}
Let $d\geq2$. For every $\delta\in(0,1/2)$, let $K_{\delta}$, $K_1(\delta,d)$ be defined by (\ref{K2i}), (\ref{Kdelta}), respectively. Then 
  \begin{equation}\label{divkernel}
\sup_{t\in[0,T]}K_{\delta}(t)\lesssim K_1(\delta,d).
\end{equation}

\end{lemma}
\begin{proof}
Note that $K_{\delta}(t)$ can be written as
\begin{align*}
K_{\delta}(t)&=\int_0^t\int_{(\mathbb{T}^d)^2}p(t-s,x-y_1)\int_{\mathbb{T}^d}\eta_{\delta}(y_1+z)\eta_{\delta}(z+y_2)dz p(t-s,x-y_2)dy_1dy_2ds\\
&=\int_0^t\int_{\mathbb{T}^d}\Big(\int_{\mathbb{T}^d}p(t-s,x+z-y_1)\eta_{\delta}(y_1)dy_1\Big)^2dzds\\
&=\int^t_0\|P_{t-s}\eta_{\delta}\|_{L^2(\mathbb{T}^d)}^2ds.
\end{align*}

By Parseval's identity we have that
\begin{align*}
\int^t_0\|P_{t-s}\eta_{\delta}\|_{L^2(\mathbb{T}^d)}	^2ds=&\int^t_0\sum_{j\geq0,\theta=1,2}\langle P_{t-s}\eta_{\delta},e_{j,\theta}\rangle^2ds=\int^t_0\sum_{j\geq0,\theta=1,2}\langle\eta_{\delta},P_{t-s}e_{j,\theta}\rangle^2ds\\
=&\int^t_0\sum_{j\geq0,\theta=1,2}e^{-2\alpha_j(t-s)}\langle\eta_{\delta},e_{j,\theta}\rangle^2ds\\
=&\sum_{j\geq1,\theta=1,2}\frac{1}{2\alpha_j}(1-e^{-2\alpha_jt})\langle\eta_{\delta},e_{j,\theta}\rangle^2+t\\
\leq&\sum_{j\geq1,\theta=1,2}\frac{1}{2\alpha_j}\langle\eta_{\delta},e_{j,\theta}\rangle^2+t. 
\end{align*}
Due to \eqref{eq:alpha_scaling} and  (\ref{propertyofeta}), by changing variables, we obtain that
\begin{align*}
	\int^t_0\|P_{t-s}\eta_{\delta}\|_{L^2(\mathbb{T}^d)}	^2ds=&\sum_{j\geq1,\theta=1,2}\frac{1}{2\alpha_j}\langle\eta_{\delta},e_{j,\theta}\rangle^2+t\lesssim\sum_{j\geq1,\theta=1,2}\frac{1}{\alpha_j(1+(\delta^2\alpha_j)^n)^2}+t\\
	\lesssim&\sum_{j\geq1}\frac{1}{j^{\frac{2}{d}}(1+(\delta^2 j^{\frac{2}{d}})^n)^2}+t\lesssim\int_1^{\infty}\frac{1}{x^{\frac{2}{d}}(1+(\delta^2 x^{\frac{2}{d}})^n)^2}dx+t\\
	\lesssim&\delta^{2-d}\int_{\delta^d}^{\infty}\frac{1}{x^{\frac{2}{d}}(1+x^{\frac{2n}{d}})^2}dx+t. 
\end{align*}
Thanks to the choice of $n>(\frac{d-2}{4})\vee0$, the above integrals are finite. When $d=2$, this implies that 
\begin{align*}
	\delta^{-d+2}\int_{\delta^d}^{\infty}\frac{1}{x^{\frac{2}{d}}(1+x^{\frac{2n}{d}})^2}dx=\int_{\delta^2}^{\infty}\frac{1}{x(1+x^{\frac{2n}{2}})^2}dx\lesssim\log(1/\delta).
\end{align*}
When $d\geq3$, we arrive at 
\begin{align*}
	\delta^{-d+2}\int_{\delta^d}^{\infty}\frac{1}{x^{\frac{2}{d}}(1+x^{\frac{2n}{d}})^2}dx\lesssim\delta^{-d+2}.
\end{align*}
In conclusion,   
\begin{align*}
	\sup_{t\in[0,T]}\int^t_0\|P_{t-s}\eta_{\delta}\|_{L^2(\mathbb{T}^d)}	^2ds\lesssim \log(1/\delta)I_{\{d=2\}}+\delta^{-d+2}I_{\{d\geq3\}}. 
\end{align*}
\end{proof}

\begin{lemma}\label{maximalLp-noncons}
 Assume that $g$ satisfies (\ref{Lpforg}). Then there is a constant $C=C(p)$, such that, 
	\begin{equation*}
		\sup_{t\in[0,T],x\in\mathbb{T}^d}\mathbb{E}|u(t,x)|^p\leq C\cdot(\varepsilon K_1(\delta,d))^{\frac{p}{2}}\sup_{t\in[0,T],x\in\mathbb{T}^d}\mathbb{E}|g(t,x)|^p.
	\end{equation*}

\end{lemma}
\begin{proof}
It is sufficient to consider the case $d\geq2$, since $d=1$ can treated analogously. For $p\geq2$, thanks to the $L^p$-isometry of the stochastic integral (see \cite[Corollary 3.11]{NW}), and applying Minkowski's inequality, we observe that
\begin{align*}
&\mathbb{E}|u(t,x)|^p=\varepsilon^{\frac{p}{2}}\mathbb{E}\Big|\int^t_0\langle p(t-s,x-\cdot),g(s,\cdot)dW^{\delta}(s)\rangle\Big|^p\\
\leq&C(p)\varepsilon^{\frac{p}{2}}\mathbb{E}\Big|\int_0^t\int_{\mathbb{T}^{2d}}p(t-s,x-y_1)p(t-s,x-y_2)g(s,y_1)g(s,y_2)R_{\delta}(y_1-y_2)dy_1dy_2ds\Big|^{\frac{p}{2}}\\
\leq&C(p)\varepsilon^{\frac{p}{2}}\Big|\int_0^{t}\int_{\mathbb{T}^{2d}}|p(t-s,x-y_1)p(t-s,x-y_2)|(\mathbb{E}g(s,y_1)g(s,y_2)|^{\frac{p}{2}})^{\frac{2}{p}}R_{\delta}(y_1-y_2)dy_1dy_2ds\Big|^{\frac{p}{2}}.
\end{align*}
Using H\"older's inequality, we find that  
\begin{align}\label{holder-coefficient}
\Big(\mathbb{E}|g(s,y_1)g(s,y_2)|^{\frac{p}{2}}\Big)^{\frac{2}{p}}\leq\Big[(\mathbb{E}|g(s,y_1)|^p)^{\frac{1}{2}}(\mathbb{E}|g(s,y_2)|^p)^{\frac{1}{2}}\Big]^{\frac{2}{p}}\lesssim\Big(\sup_{s\in[0,T],y\in\mathbb{T}^d}\mathbb{E}|g(s,y)|^p\Big)^{\frac{2}{p}}.
\end{align}
Thanks to Lemma \ref{K2}, it follows that
\begin{align*}
\sup_{t\in[0,T],x\in\mathbb{T}^d}\mathbb{E}|u(t,x)|^p\leq&\varepsilon^{\frac{p}{2}}C(p)\Big(\sup_{t\in[0,T],x\in\mathbb{T}^d}\mathbb{E}|g(t,x)|^p\Big)\Big(\sup_{t\in[0,T]}K_{\delta}(t)^{\frac{p}{2}}\Big)\\
\leq&C(p)(\varepsilon K_1(\delta,d))^{\frac{p}{2}}\sup_{t\in[0,T],x\in\mathbb{T}^d}\mathbb{E}|g(t,x)|^p. 
\end{align*}
Then we complete the proof by using H\"older's inequality to see the same estimate holds for $p\in[1,2)$.  

\end{proof}

\subsection{The case of conservative noise}\label{subsec-4-2}
In the study of the conservative SHE (\ref{SHE01}), we will employ the framework introduced in \cite{krylovlp, JML, NW}, with a key role played by a generalization of the Littlewood-Paley inequality. We state the result below.
\begin{lemma}\label{gen-littlewoodpaley}\cite{krylovlp,JML}
	Let $\mathcal{H}$ be a Hilbert space, and $p\in(2,+\infty)$. Let $\{S(t)\}_{t\geq0}$ be the heat semi-group. For every $f\in L^p(\Omega\times[0,T]\times\mathbb{T}^d;\mathcal{H})$, there is a constant $C=C(p)$, such that, 
	\begin{align}\label{littlewood-paley}
	\mathbb{E}\Big(\int^T_0\int_{\mathbb{T}^d}\Big(\int^t_0\|(\nabla S(t-s)f(s))(x)\|_{\mathcal{H}}^2ds\Big)^{\frac{p}{2}}dx dt\Big)\leq C\mathbb{E}\|f\|_{L^p([0,T]\times\mathbb{T}^d;\mathcal{H})}^p.
\end{align}

\end{lemma}

Let $u$ be the stochastic convolution defined by (\ref{sto-convolution-4}). Using the $L^p$-isometry of the stochastic integral \cite[Corollary 3.11]{NW} to see that  
\begin{align}\label{Lpisometry}
\mathbb{E}\|u\|_{L^p([0,T]\times\mathbb{T}^d)}^p\leq&\varepsilon^{\frac{p}{2}}\mathbb{E}\Big\|\int^{\cdot}_0\nabla S(t-s)g_{\delta}(s)dW(s)\Big\|_{L^p([0,T]\times\mathbb{T}^d)}^p\notag\\
\leq&\varepsilon^{\frac{p}{2}}C(p)\mathbb{E}\Big(\int^T_0\int_{\mathbb{T}^d}\Big(\int^t_0\int_{\mathbb{T}^d}|(\nabla S(t-s)g_{\delta}(s,z))(x)|^2dzds\Big)^{\frac{p}{2}}dx dt\Big)\notag\\
\leq&\varepsilon^{\frac{p}{2}}C(p)\mathbb{E}\Big(\int^T_0\int_{\mathbb{T}^d}\Big(\int^t_0\int_{\mathbb{T}^d}|(\nabla S(t-s)g(s)\eta_{\delta}(\cdot-z))(x)|^2dzds\Big)^{\frac{p}{2}}dx dt\Big).	
\end{align}
Therefore, taking $f(s,x)=g(s,x)\eta_{\delta}(x-\cdot)$ in Lemma \ref{gen-littlewoodpaley}, we have the following estimate.

\begin{lemma}\label{maximalLp-cons}
For every $p\in[1,\infty)$, there is a constant $C=C(p)$, such that,  
	\begin{equation*}
		\mathbb{E}\|u\|_{L^p([0,T]\times\mathbb{T}^d)}^p\leq C\cdot(\varepsilon K_2(\delta,d))^{\frac{p}{2}}\mathbb{E}\|g\|_{L^p([0,T]\times\mathbb{T}^d)}^p.
	\end{equation*}
\end{lemma}
\begin{proof}
In combination of (\ref{Lpisometry}) and (\ref{littlewood-paley}), we deduce that 
	\begin{align}\label{computation-1}
\mathbb{E}\|u\|_{L^p([0,T]\times\mathbb{T}^d)}^p\leq C(p)\varepsilon^{\frac{p}{2}}\mathbb{E}\int^T_0\int_{\mathbb{T}^d}\Big(\int_{\mathbb{T}^d}g(s,y)^2\eta_{\delta}(y-z)^2dz\Big)^{\frac{p}{2}}dyds.
\end{align}
By the definition of $\eta_{\delta}$ and by changing variables, 
\begin{align}\label{computation-2}
\mathbb{E}\int^T_0\int_{\mathbb{T}^d}\Big(\int_{\mathbb{T}^d}g(s,y)^2\eta_{\delta}(y-z)^2dz\Big)^{\frac{p}{2}}dyds=&\varepsilon^{\frac{p}{2}}\mathbb{E}\int^T_0\int_{\mathbb{T}^d}|g(s,y)|^pdy\Big(\int_{\mathbb{T}^d}\eta_{\delta}(z)^2dz\Big)^{\frac{p}{2}}ds\notag\\
=&\varepsilon^{\frac{p}{2}}\mathbb{E}\int^T_0\int_{\mathbb{T}^d}|g(s,y)|^pdy\Big(\int_{\mathbb{R}^d}\delta^{-2d}\tilde{\eta}(z/\delta)^2dz\Big)^{\frac{p}{2}}ds\notag\\
=&\varepsilon^{\frac{p}{2}}\mathbb{E}\int^T_0\int_{\mathbb{T}^d}|g(s,y)|^pdy\Big(\int_{\mathbb{R}^d}\delta^{-d}\tilde{\eta}(z)^2dz\Big)^{\frac{p}{2}}ds\notag\\
\lesssim&(\varepsilon\delta^{-d})^{\frac{p}{2}}\mathbb{E}\int^T_0\int_{\mathbb{T}^d}|g(s,y)|^pdyds. 
\end{align}
This completes the proof. 
\end{proof}

\section{Well-posedness for SHE with smooth and non-smooth coefficients}\label{sec-4}
In this section, we will prove the well-posedness of (\ref{SHE02}) and (\ref{SHE01}) with smooth and non-smooth diffusion coefficients. In the case of smooth coefficients, global in time well-posedness results are provided in the Subsection \ref{subsec-smooth}. In the case of non-smooth coefficients, we will show results for local in time well-posedness in the Subsection \ref{subsec-nonsmooth}. 
\subsection{Global in time well-posedness for SHE with smooth diffusion coefficients}\label{subsec-smooth}
We first introduce the definition of mild solution for (\ref{SHE02}) (resp. (\ref{SHE01})). 
\begin{definition}[Mild solution]
	Let $\varepsilon,\delta>0$.  
		\begin{description}
		\item [(i)] An $L^2(\mathbb{T}^d)$-valued $\{\mathcal{F}(t)\}_{t\in[0,T]}$-adapted process $u^{\varepsilon,\delta}$ is called a mild solution of  (\ref{SHE02}) with initial data $u_0\in L^{\infty}(\mathbb{T}^d)$, if almost surely, $u^{\varepsilon,\delta}\in C([0,T];L^2(\mathbb{T}^d))$,  $G(u^{\varepsilon,\delta}) \in L^2([0,T]\times\Omega\times \mathbb{T}^d)$
and, $\mathbb{P}$-almost surely,
	\begin{equation}\label{mild-1}
u^{\varepsilon,\delta}(t,x)=(p(t)\ast u_0)(x)+\varepsilon^{\frac{1}{2}}\int_0^t\langle p(t-s,x-\cdot),G(u^{\varepsilon,\delta}(s,\cdot))dW^{\delta}(s)\rangle,
\end{equation}
for every $t\in[0,T]$.

\item [(ii)] An $L^2(\mathbb{T}^d)$-valued $\{\mathcal{F}(t)\}_{t\in[0,T]}$-adapted process $u^{\varepsilon,\delta}$ is called a mild solution of  (\ref{SHE01}) with initial data $u_0\in L^{\infty}(\mathbb{T}^d)$, if almost surely, $u^{\varepsilon,\delta}\in C([0,T];H^{-1}(\mathbb{T}^d))$, and we have that for almost every $t\in[0,T]$, $G(u^{\varepsilon,\delta}) \in L^2([0,T]\times\Omega\times \mathbb{T}^d)$

and, $\mathbb{P}$-almost surely,
	\begin{equation}\label{mild-1-cons}
u^{\varepsilon,\delta}(t,x)=(p(t)\ast u_0)(x)+\varepsilon^{\frac{1}{2}}\int_0^t\langle \nabla_x p(t-s,x-\cdot),G(u^{\varepsilon,\delta}(s,\cdot))dW^{\delta}(s)\rangle,
\end{equation}
for every $t\in[0,T]$.
	\end{description}
	
	\end{definition}

In the case of nonconservative noise, $d=1$, we take $\delta=0$ and let $u^{\varepsilon}=u^{\varepsilon,0}$, $W=W_{0}$. 

\begin{remark}
We emphasize that by using Lemma \ref{gen-littlewoodpaley} and Lemma \ref{maximalLp-cons}, the square integrability conditions on $G(u^{\varepsilon,\delta})$ are sufficient to guarantee that the stochastic integrals in (\ref{mild-1}) and (\ref{mild-1-cons}) are well-defined. 
\end{remark}

The existence and uniqueness of the mild solution to  (\ref{SHE02}) is well known, see, for example, \cite{Daprato}, \cite{walsh}. 

In the following, for this sake, we introduce an $H^{-1}$-variational framework. Recall that $\{e_{k,\theta}\}_{k\geq0,\theta=1,2}$ is the orthonormal basis of $L^2(\mathbb{T}^d)$.  For any $m\geq0$, let  
\begin{align}\label{sobolev-norm}
  	H^{m}(\mathbb{T}^d)=&\Big\{f\in L^2(\mathbb{T}^d):\sum_{k\geq0,\theta=1,2}\alpha_k^m|\langle f,e_{k,\theta}\rangle|^2<\infty\Big\},\notag\\
  	\|f\|_{H^{m}(\mathbb{T}^d)}^2=&\sum_{k\geq0,\theta=1,2}\alpha_k^m|\langle f,e_{k,\theta}\rangle|^2.
  \end{align}
 
For any $m>0$, let $H^{-m}(\mathbb{T}^d)=(H^{m}(\mathbb{T}^d))^*$. Extending the $L^2(\mathbb{T}^d)$-inner product $\langle\cdot,\cdot\rangle$ to the $H^{-m}(\mathbb{T}^d)-H^m(\mathbb{T}^d)$ duality denoted by $(\cdot,\cdot)$, and the norm of an element $f\in H^{-m}(\mathbb{T}^d)$ is given by 
\begin{align*}
\|f\|_{H^{-m}(\mathbb{T}^d)}^2=&\sum_{k\geq1,\theta=1,2}\alpha_k^{-m}|(f,e_{k,\theta})|^2+|(f,e_{0,\theta})|^2.	
\end{align*}

In order to apply the variational approach to (\ref{SHE01}), we introduce the following space. For any $m\geq0$, let 
\begin{align}\label{H-1norm}
  	H^{m}_0(\mathbb{T}^d)=&\Big\{f\in L^2(\mathbb{T}^d):\sum_{k\geq0,\theta=1,2}\alpha_k^m|\langle f,e_{k,\theta}\rangle|^2<\infty,\ \langle f,e_{0,\theta}\rangle=0,\ \theta=1,2\Big\},\notag\\
  	\|f\|_{H^{m}_0(\mathbb{T}^d)}^2=&\sum_{k\geq1,\theta=1,2}\alpha_k^m|\langle f,e_{k,\theta}\rangle|^2.
  \end{align}
  For any $m>0$, let $H^{-m}_0(\mathbb{T}^d)=(H^{m}_0(\mathbb{T}^d))^*$. We denote the extension of the $L^2(\mathbb{T}^d)$-inner product $\langle\cdot,\cdot\rangle$ to the $H^{-m}_0(\mathbb{T}^d)-H^m_0(\mathbb{T}^d)$ duality by $(\cdot,\cdot)$, and the norm of an element $f\in H^{-m}_0(\mathbb{T}^d)$ is given by 
\begin{align*}
\|f\|_{H^{-m}_0(\mathbb{T}^d)}^2=&\sum_{k\geq1,\theta=1,2}\alpha_k^{-m}|(f,e_{k,\theta})|^2.	
\end{align*}

Since for the moment the correlation length $\delta>0$ is fixed, for simplicity, we let $u^{\varepsilon}:=u^{\varepsilon,\delta}$ for $\varepsilon>0$.

\begin{definition}[Global in time $H^{-1}$-variational solution]\label{variation}
Let $\varepsilon>0$, and $u_0\in H^{-1}(\mathbb{T}^d)$. A continuous $H^{-1}(\mathbb{T}^d)$-valued $\{\mathcal{F}(t)\}_{t\in[0,T]}$-adapted process $(u^{\varepsilon}(t))_{t\in[0,T]}$ is called a global in time $H^{-1}(\mathbb{T}^d)$-variational solution of  (\ref{SHE01}) on $[0,T]$, if the following conditions hold.
\begin{description}
\item[(i)]For the $dt\otimes\mathbb{P}$-equivalence class of $u^{\varepsilon}$, denoted by $\hat{u}^{\varepsilon}$, we have  
\begin{equation*}
\hat{u}^{\varepsilon}\in L^2([0,T];L^2(\Omega;L^2(\mathbb{T}^d))).
\end{equation*}
\item[(ii)] $\mathbb{P}$-almost surely,
\begin{align}\label{variational-solu}
u^{\varepsilon}(t)=&u_0+\int^t_0\Delta \bar{u}^{\varepsilon}(s)ds+\varepsilon^{\frac{1}{2}}\int^t_0\nabla\cdot(G(\bar{u}^{\varepsilon}(s))dW^{\delta}(s))
\end{align}
holds in $(L^2(\mathbb{T}^d))^*$ for every $t\in[0,T]$, where $\bar{u}^{\varepsilon}$ is any $L^2(\mathbb{T}^d)$-valued progressively measurable $dt\otimes\mathbb{P}$-version of $\hat{u}^{\varepsilon}$.
\end{description}

\end{definition}
Due to the fact that (\ref{SHE01}) takes the form of a conservation law, the preservation of the mean value for the solution of (\ref{SHE01}) holds, by which we mean that for any $\varepsilon, \delta>0$, let $u^{\varepsilon,\delta}$ be a $H^{-1}$-variational solution of (\ref{SHE01}), then for every $t\in[0,T]$, it follows that almost surely, 
\begin{align*}
(u^{\varepsilon,\delta}(t),e_{0,\theta})=\int_{\mathbb{T}^d}u^{\varepsilon,\delta}(t,x)dx=\int_{\mathbb{T}^d}u_0(x)dx=(u_0,e_{0,\theta}). 
\end{align*}
Inspired by this property, the mean-zero Sobolev spaces $H^{m}_0(\mathbb{T}^d)$, $m\in\mathbb{R}$, are employed to study the well-posedness of (\ref{SHE01}). 
\begin{lemma}\label{variationwellposed}
Let $G$ and $u_0$ satisfy Hypothesis H1 and H2. Let $\delta>0$, $\varepsilon_0=2\delta^d\|G^{(1)}(\cdot)\|_{L^{\infty}(\mathbb{R})}^{-2}>0$. Then for every $\varepsilon\in(0,\varepsilon_0)$,  (\ref{SHE01}) admits a unique $H^{-1}(\mathbb{T}^d)$-variational solution in the sense of Definition \ref{variation}. Moreover, the variational solution $u^{\varepsilon}$ satisfies the mild form (\ref{mild-1-cons}). 
\end{lemma}
\begin{proof}
Consider the following equation
\begin{align}\label{zeromeaneq}
d\tilde{u}=\Delta\tilde{u}dt+\varepsilon^{\frac{1}{2}}\nabla\cdot(G(\tilde{u}+\int_{\mathbb{T}^d}u_0(x)dx)dW^{\delta}(t)), 	
\end{align}
with initial data $\tilde{u}(0)=u_0-\int_{\mathbb{T}^d}u_0(x)dx$. If $\tilde{u}$ is an $H^{-1}(\mathbb{T}^d)$-variational solution of (\ref{zeromeaneq}), then 
\begin{align*}
u=\tilde{u}+\int_{\mathbb{T}^d}u_0(x)dx	
\end{align*}
is an $H^{-1}(\mathbb{T}^d)$-variational solution of (\ref{SHE01}). Set $G_{u_0}(\zeta)=G(\zeta+\int_{\mathbb{T}^d}u_0(x)dx)$, $\zeta\in\mathbb{R}$, then $G_{u_0}(\cdot)$ satisfies Hypothesis H2 as well. Let $V=H^0_0(\mathbb{T}^d)$, $H=H^{-1}_0(\mathbb{T}^d)$, then $V\subset H\equiv H^*\subset V^*$ is a Gelfand triple, where we have used the Riesz isomorphism $(-\Delta)^{-1}:H\rightarrow H^*$ to identify $H$ with its dual $H^*$. The Laplacian operator $\Delta$ can be extended to a continuous map $\Delta: V\rightarrow V^*$, see \cite[Lemma 4.1.13]{LR}. For any $u\in V$, set $B_{\varepsilon}(u)\cdot=\varepsilon^{\frac{1}{2}}\nabla\cdot(G_{u_0}(u)\eta_{\delta}\ast\cdot)$. Let $\mathcal{L}_2(V,H)$ be the space of Hilbert-Schmidt operators from $V$ to $H$. For any $u^1,u^2\in V$, by \cite[Lemma 3.4 and (5.1)]{KGG} and the definition of the $H^{-1}(\mathbb{T}^d)$-norm, 
\begin{align}\label{HSoperator}
\|B_{\varepsilon}(u^1)-B_{\varepsilon}(u^2)\|_{\mathcal{L}_2(V,H)}^2\lesssim&\varepsilon\sum_{k\geq0,\theta=1,2}\|(G_{u_0}(u^1)-G_{u_0}(u^2))(\eta_{\delta}\ast e_{k,\theta})\|_{V}^2\notag\\
\leq&\frac{\varepsilon\|G_{u_0}^{(1)}(\cdot)\|_{L^{\infty}(\mathbb{R})}^2}{2}\sum_{k\geq0,\theta=1,2}\Big(\int_{\mathbb{T}^d}|u^1-u^2|^2(\eta_{\delta}\ast e_{k,\theta})^2dx\Big).  
\end{align}
By a direct calculation,
\begin{equation*}
\sum_{k\geq0,\theta=1,2}(\eta_{\delta}\ast e_{k,\theta})^2=\sum_{k\geq0,\theta=1,2}\langle\eta_{\delta}(x-\cdot),e_{k,\theta}\rangle^2=\|\eta_{\delta}(x-\cdot)\|_{V}^2\leq\delta^{-d}.
\end{equation*}
Therefore, it follows that   
\begin{align}\label{absorb}
\|B_{\varepsilon}(u^1)-B_{\varepsilon}(u^2)\|_{\mathcal{L}_2(V,H)}^2\lesssim& \frac{\varepsilon\|G_{u_0}^{(1)}(\cdot)\|_{L^{\infty}(\mathbb{R})}^2\delta^{-d}}{2}\|u^1-u^2\|_{V}^2.
\end{align}
The Laplacian term can be treated as a special case of the porous medium operator in \cite{LR}, thus the hemi-continuity condition  holds, that is, $\langle u^1,\Delta(u^1+\lambda u^2)\rangle$ is continuous in $\lambda\in\mathbb{R}$, and the growth condition holds, that is, $\|\Delta u^1\|_{V^*}\lesssim\|u^1\|_V$. See \cite[page 87, (H1), page 88, (H4)]{LR} for more details. Furthermore, we have that 
\begin{align}\label{laplacemonotone}
{}_{V^*}\langle\Delta(u^1-u^2),(u^1-u^2)\rangle_V=-\|u^1-u^2\|_{V}^2.
\end{align}
Choosing $\varepsilon$ such that $\frac{\varepsilon\|G_{u_0}^{(1)}(\cdot)\|_{L^{\infty}(\mathbb{R})}^2\delta^{-d}}{2}<1$, together with (\ref{absorb}), it follows that 
\begin{align}\label{H-1unique}
&{}_{V^*}\langle\Delta(u^1-u^2),(u^1-u^2)\rangle_V+\|B_{\varepsilon}(u^1)-B_{\varepsilon}(u^2)\|_{\mathcal{L}_2(V,H)}^2\notag\\
&\lesssim -\Big(1-\frac{\varepsilon\|G_{u_0}^{(1)}(\cdot)\|_{L^{\infty}(\mathbb{R})}^2\delta^{-d}}{2}\Big)\|u^1-u^2\|_V^2,
\end{align}
which implies the weak monotonicity condition (see \cite[page 70, (H2)]{LR}).  Similar to (\ref{HSoperator}) and (\ref{laplacemonotone}), for $u\in V$, choosing $\varepsilon$ such that $\frac{\varepsilon\|G_{u_0}^{(1)}(\cdot)\|_{L^{\infty}(\mathbb{R})}^2\delta^{-d}}{2}<1$, it follows that
\begin{align*}
{}_{V^*}\langle\Delta u,u\rangle_V+\|B_{\varepsilon}(u)\|_{\mathcal{L}_2(V,H)}^2\leq -\Big(1-\frac{\varepsilon\|G_{u_0}^{(1)}(\cdot)\|_{L^{\infty}(\mathbb{R})}^2\delta^{-d}}{2}\Big)\|u\|_V^2,
\end{align*}
which implies the coercivity condition in \cite[page 70, (H3)]{LR}. As a consequence,  (\ref{zeromeaneq}) fulfills the conditions required in the well-posedness framework of \cite{LR} (also see \cite{Pardoux} \cite{Krylovoriginal}). More precisely, due to \cite[Definition 5.1.2, Theorem 5.1.3]{LR}, (\ref{zeromeaneq})  has a unique $H$-variational solution $\tilde{u}$ in the sense of \cite[Definition 4.2.1]{LR} with $\tilde{u}\in L^2([0,T];L^2(\Omega;V))$ and $\mathbb{P}$-almost surely $\tilde{u}\in C([0,T];H)$. Therefore (\ref{SHE01}) has a unique $H^{-1}(\mathbb{T}^d)$-variational solution $u$ in the sense of Definition \ref{variation}. 
 
Moreover, a direct verification shows that the variational solution $u$ satisfies a weak form of (\ref{SHE01}) given by \cite[Definition 3.1]{GM}. By the assumption of $G$ and \cite[Theorem 3.2]{GM}, it follows that  $u$ satisfies the mild form (\ref{mild-1}). 
\end{proof}

\subsection{Local in time well-posedness for SHE with non-smooth  coefficients}\label{subsec-nonsmooth}
In this section, we prove the local in time well-posedness of  (\ref{SHE02}) and  (\ref{SHE01}) with non-Lipschitz diffusion coefficients in $d$-dimension.

We first introduce the definitions of local in time $H^{-1}(\mathbb{T}^d)$-variational solution of  (\ref{SHE01}), local in time mild solutions of  (\ref{SHE02}),  (\ref{SHE01}), and local in time uniquness.

\begin{definition}[Local in time $H^{-1}$-variational solution]\label{variation-local}
	Let $\varepsilon,\delta>0$, a couple $(u^{\varepsilon,\delta},\tau^{\varepsilon,\delta})$ is called a local in time $H^{-1}$-variational solution of  (\ref{SHE01}) with initial data $u_0\in H^{-1}(\mathbb{T}^d)$, if 
	\begin{description}
\item[(i)]$\tau^{\varepsilon,\delta}$ is an $\{\mathcal{F}(t)\}_{t\in[0,T]}$-stopping time with $\tau^{\varepsilon,\delta}\in(0,T]$, $\mathbb{P}$-almost surely, and $(u^{\varepsilon,\delta}(t))_{t\in[0,\tau^{\varepsilon,\delta}]}$ is an $H^{-1}(\mathbb{T}^d)$-valued $\{\mathcal{F}(t)\}_{t\in[0,T]}$-adapted stochastic process in the sense that 
\begin{eqnarray*}
\tilde{u}^{\varepsilon,\delta}(t)=\left\{
  \begin{array}{ll}
   u^{\varepsilon,\delta}(t)
  &  {\rm{if}}\ t\in[0,\tau^{\varepsilon,\delta}),\\
      u^{\varepsilon,\delta}(\tau^{\varepsilon,\delta}) &{\rm{if}}\ t\in[\tau^{\varepsilon,\delta},T],
  \end{array}
\right.
\end{eqnarray*}
is $H^{-1}(\mathbb{T}^d)$-valued $\{\mathcal{F}(t)\}_{t\in[0,T]}$-adapted.

\item[(ii)]We have $\mathbb{P}$-almost surely,  $u^{\varepsilon,\delta}\in C([0,\tau^{\varepsilon,\delta}];H^{-1}(\mathbb{T}^d))\cap L^2([0,\tau^{\varepsilon,\delta}];L^2(\mathbb{T}^d)),
$
and
\begin{align*}
u^{\varepsilon,\delta}(t\wedge\tau^{\varepsilon,\delta})=&u_0+\int^{t\wedge\tau^{\varepsilon,\delta}}_0\Delta \bar{u}^{\varepsilon,\delta}(s)ds+\varepsilon^{\frac{1}{2}}\int^{t\wedge\tau^{\varepsilon,\delta}}_0\nabla\cdot(G(\bar{u}^{\varepsilon,\delta}(s))dW^{\delta}(s))
\end{align*}
holds in $(L^2(\mathbb{T}^d))^*$ for every $t\in[0,T]$, where $\bar{u}^{\varepsilon,\delta}$ is any $L^2(\mathbb{T}^d)$-valued progressively measurable $dt\otimes\mathbb{P}$-version of $\hat{u}^{\varepsilon,\delta}$. 
\end{description}
\end{definition}

\begin{definition}[Local in time mild solution]\label{mild-local}
	Let $\varepsilon,\delta>0$. A tuple $(u^{\varepsilon,\delta},\tau^{\varepsilon,\delta})$ is called a local in time mild solution of  (\ref{SHE02}) (resp. (\ref{SHE01})) with initial data $u_0\in L^{\infty}(\mathbb{T}^d)$, if 
	\begin{description}
\item[(i)]$\tau^{\varepsilon,\delta}$ is an $\{\mathcal{F}(t)\}_{t\in[0,T]}$-stopping time with $\tau^{\varepsilon,\delta}\in(0,T]$, $\mathbb{P}$-almost surely, and $u^{\varepsilon,\delta}\in C([0,\tau^{\varepsilon,\delta}];L^2(\mathbb{T}^d))$ (resp.\ $u^{\varepsilon,\delta}\in C([0,\tau^{\varepsilon,\delta}];H^{-1}(\mathbb{T}^d))\cap L^2([0,\tau^{\varepsilon,\delta}];L^2(\mathbb{T}^d))$) almost surely, and $(u^{\varepsilon,\delta}(t))_{t\in[0,\tau^{\varepsilon,\delta}]}$ is an $L^2(\mathbb{T}^d)$-valued (resp.\ $H^{-1}(\mathbb{T}^d)$-valued) $\{\mathcal{F}(t)\}_{t\in[0,T]}$-adapted stochastic process.

\item[(ii)] We have that $1_{ [0,\tau^{\varepsilon,\delta}]}G(u^{\varepsilon,\delta}(s,\cdot)) \in L^2([0,T]\times\Omega\times\mathbb{T}^d)$
and, $\mathbb{P}$-almost surely, 
		\begin{equation*}
u^{\varepsilon,\delta}(t\wedge\tau^{\varepsilon,\delta},x)=(p(t\wedge\tau^{\varepsilon,\delta})\ast u_0)(x)+\varepsilon^{\frac{1}{2}}\int_0^{t\wedge\tau^{\varepsilon,\delta}}\langle p(t\wedge\tau^{\varepsilon,\delta}-s,x-\cdot),G(u^{\varepsilon,\delta}(s,\cdot))dW^{\delta}(s)\rangle,
\end{equation*}
holds for every $t\in[0,T]$ and almost every $x\in\mathbb{T}^d$, resp.\ for (\ref{SHE01}),
		\begin{equation*}
u^{\varepsilon,\delta}(t\wedge\tau^{\varepsilon,\delta},x)=(p(t\wedge\tau^{\varepsilon,\delta})\ast u_0)(x)+\varepsilon^{\frac{1}{2}}\int_0^{t\wedge\tau^{\varepsilon,\delta}}\langle \nabla_x p(t\wedge\tau^{\varepsilon,\delta}-s,x-\cdot),G(u^{\varepsilon,\delta}(s,\cdot))dW^{\delta}(s)\rangle.
\end{equation*}
\end{description}

\end{definition}

\begin{definition}[Uniqueness up to time $\tau^{\varepsilon,\delta}$]\label{localuniqueness}
	Let $\varepsilon,\delta>0$, and let $\tau^{\varepsilon,\delta}$ be an $\{\mathcal{F}(t)\}_{t\in[0,T]}$-stopping time such that $\tau^{\varepsilon,\delta}\in(0,T]$, $\mathbb{P}$-almost surely. We say that the local in time solution of (\ref{SHE02}) (resp. (\ref{SHE01})) with initial data $u_0\in L^{\infty}(\mathbb{T}^d)$ is unique up to time $\tau^{\varepsilon,\delta}$, if for any two local in time solutions $(u_1,\tau_1)$, $(u_2,\tau_2)$, we have $\mathbb{P}$-almost surely, $u_1(t,x)=u_2(t,x)$ for almost every $x\in\mathbb{T}^d$ and every $t\in[0,\tau^1\wedge\tau^2\wedge\tau^{\varepsilon,\delta}]$. 
\end{definition}

In order to study the local in time well-posedness of  (\ref{SHE01}), we first introduce an approximation. Assume that Hypothesis H3 holds for the initial data $u_0$ and the diffusion coefficient $G$ , let $\gamma$ be a fixed suitable constant that appears in Hypothesis H3. Let $G_0\in C^{\infty}(\mathbb{R})$ such that
\begin{equation}\label{G0}
G_0(\zeta)=G(\zeta),\ \text{for } \zeta\in\Big[{\rm{ess\inf}}u_0-\gamma,{\rm{ess\sup}}u_0+\gamma\Big],
\end{equation}
and $G_0=0$ on $({\rm{ess\inf}}u_0-\gamma-\gamma',{\rm{ess\sup}}u_0+\gamma+\gamma')^c$, for some $\gamma'>0$.

We then consider the following approximation of (\ref{SHE01})
\begin{equation}\label{approxeq}
d\rho^{\varepsilon,\delta}=\Delta \rho^{\varepsilon,\delta}dt+\varepsilon^{\frac{1}{2}}\nabla\cdot(G_0(\rho^{\varepsilon,\delta})dW^{\delta}(t)),
\quad \rho^{\varepsilon,\delta}(0)=u_0.
\end{equation}
By the $H^{-1}(\mathbb{T}^d)$-variational approach, we obtain the global in time well-posedness for (\ref{approxeq}), see Lemma \ref{variationwellposed}. A similar procedure can be applied to the case of non-conservative noise (\ref{SHE02}). In this case, $\rho^{\varepsilon,\delta}$ is the mild solution of  
\begin{equation}\label{approxeq-2}
d\rho^{\varepsilon,\delta}=\Delta\rho^{\varepsilon,\delta}dt+\varepsilon^{\frac{1}{2}}G_0(\rho^{\varepsilon,\delta})dW^{\delta}(t),\quad \rho^{\varepsilon,\delta}(0)=u_0,
\end{equation}
 where $G_0$ is defined by (\ref{G0}).
 Referring to \cite{walsh}, \cite{Daprato}, we have the well-posedness of (\ref{approxeq-2}). More precisely, the following lemma holds. 

\begin{lemma}\label{variationwellposed-2}
Assume that Hypothesis H3 on the initial data $u_0$ and the diffusion coefficient $G$ holds for some $\gamma>0$. Let $G_0$ be defined by (\ref{G0}). Let $\delta>0$, $\varepsilon_0=2\delta^d\|G_0^{(1)}(\cdot)\|_{L^{\infty}(\mathbb{R})}^{-2}>0$. Then for every $\varepsilon\in(0,\varepsilon_0)$,  there exists a unique $H^{-1}(\mathbb{T}^d)$-variational solution for (\ref{approxeq}) in the sense of Definition \ref{variation}. Moreover, the variational solution is a mild solution of (\ref{approxeq}) as well. Furthermore, for every $\varepsilon>0$, there exists a unique mild solution for (\ref{approxeq-2}).  
\end{lemma}

In the following, we provide a regularity estimate for (\ref{approxeq}). 
 \begin{lemma}\label{lem-moser-1}
Under the same hypotheses as Lemma \ref{variationwellposed-2}. Let $\varepsilon,\delta>0$. In the conservative case, we further assume that $\varepsilon<\varepsilon_0$, where $\varepsilon_0$ is the constant that appears in Lemma \ref{variationwellposed-2}. Let $\rho^{\varepsilon,\delta}$ be the variational solution (resp.  mild solution) of  (\ref{approxeq}) (resp.  (\ref{approxeq-2})) with initial data $u_0$. Then we have the following results. 

\begin{description}
\item[(i)]({\rm{Non-conservative\ noise}}) There exists $\tilde{\gamma}=\tilde{\gamma}(d)>0$, $C=C(T)>0$ with $\lim_{T\rightarrow0}C(T)<\infty$, and $\gamma'>0$ independent of $\varepsilon,\delta,T,G_0$, such that 
\begin{align}\label{Linfinity-2}
\mathbb{E}\Big(\|(\rho^{\varepsilon,\delta}-K)^+\|_{L^{\infty}(\mathbb{T}^d\times[0,T])}\Big)+\mathbb{E}\Big(\|(\rho^{\varepsilon,\delta}-K')^-\|_{L^{\infty}(\mathbb{T}^d\times[0,T])}\Big)\leq C(T)\cdot(\varepsilon\delta^{-d}T^{\tilde{\gamma}})^{\gamma'}.
\end{align}
\item[(ii)] ({\rm{Conservative\ noise}}) There exists $\tilde{\gamma}=\tilde{\gamma}(d)>0$, $C=C(T)>0$ with $\lim_{T\rightarrow0}C(T)<\infty$, and $\gamma'>0$ independent of $\varepsilon_0,\delta,T,G_0$, such that 
\begin{align}\label{Linfinity}
\mathbb{E}\Big(\|(\rho^{\varepsilon,\delta}-K)^+\|_{L^{\infty}(\mathbb{T}^d\times[0,T])}\Big)+\mathbb{E}\Big(\|(\rho^{\varepsilon,\delta}-K')^-\|_{L^{\infty}(\mathbb{T}^d\times[0,T])}\Big)\leq C(T)\cdot(\varepsilon \delta^{-d-2}T^{\tilde{\gamma}})^{\gamma'}.
\end{align}
\end{description}

\end{lemma}
\begin{proof}
In this proof, the case of conservative noise produces extra terms compared to the non-conservative case. Therefore, we focus on the conservative case; the proof of the non-conservative case is analogous. For simplicity, we denote the solution $\rho^{\varepsilon,\delta}$ of (\ref{approxeq}) by $\rho$. The proof is divided into two steps. First, we will obtain the $H^1(\mathbb{T}^d)$-regularity of the variational solution $\rho$. With this regularity, the Moser iteration technique can be employed to obtain the $L^{\infty}(\mathbb{T}^d)$-estimate.

{\bf{Step 1. $L^2(\mathbb{T}^d)$-estimate of the variational solution.}}  Recall that $\{e_{k,\theta}\}_{k\geq0,\theta=1,2}$ are the eigenvectors of the Laplacian operator. For every $k\geq0,\ \theta=1,2$, set $\bar{e}_{k,\theta}=\Delta e_{k,\theta}=-\alpha_ke_{k,\theta}$. For every $m\in\mathbb{N}_+$, set $H_m=\text{span}\{(\bar{e}_{k,\theta})_{k\leq m,\theta=1,2}\}$. Recall that $V=L^2(\mathbb{T}^d)$. Let the projection operator $P_m: V^*\rightarrow H_m$ be defined by 
\begin{equation*}
P_mf=\sum_{k\leq m,\theta=1,2}\ _{V^*}\langle f,\bar{e}_{k,\theta}\rangle_V \bar{e}_{k,\theta}, 	\ \ f\in V^*.
\end{equation*}
For every $K\in\mathbb{N}_+$, consider the project equation on  $H_m\subset V^*$,  
\begin{align}\label{m-proj}
	d\rho_{m}=P_m\Delta\rho_{m}dt+\varepsilon^{\frac{1}{2}}P_m\nabla\cdot(G_0(\rho_{m})dW^{\delta,m}),\ \ \rho_{m}(0)=P_mu_0,  
\end{align}
where $W^{\delta,m}=\sum_{k\leq m,\theta=1,2}\beta_{k,\theta}(\eta_{\delta}\ast e_{k,\theta})$. Then for every $k\leq m$, we have that 
\begin{align*}
_{V^*}\langle\rho_m(t),\bar{e}_{k,\theta}\rangle_V=&_{V^*}\langle P_mu_0,\bar{e}_{k,\theta}\rangle_V+\int^t_0\ _{V^*}\langle P_m\Delta \rho_m(s),\bar{e}_{k,\theta}\rangle_Vds\\
&+\varepsilon^{\frac{1}{2}}\ _{V^*}\langle\int^t_0 P_m\nabla\cdot(G_0(\rho_m(s))dW^{\delta^m}(s)),\bar{e}_{k,\theta}\rangle_V.
\end{align*}
This implies that for every $k\leq m,\ \theta=1,2$, 
\begin{align}\label{SDEsystem}
\langle\rho_m(t),e_{k,\theta}\rangle=&\langle P_mu_0,e_{k,\theta}\rangle+\int^t_0\langle \rho_m(s),\Delta e_{k,\theta}\rangle ds\notag\\
&+\varepsilon^{\frac{1}{2}}\langle\int^t_0 \nabla\cdot(G_0(\rho_m(s))dW^{\delta,m}(s)),e_{k,\theta}\rangle, 
\end{align}
where $\langle\cdot,\cdot\rangle$ denotes the $L^2(\mathbb{T}^d)$-inner product. Notice that (\ref{SDEsystem}) can be rewritten into a finite-dimensional stochastic differential equation of $(\rho_{m,k,\theta}=\langle\rho_m,e_{k,\theta}\rangle)_{k\leq m,\theta=1,2}$. Solving (\ref{m-proj}) is equivalent to solving (\ref{SDEsystem}). By the classical theory of stochastic differential equations, see for example, \cite{KS0}, there exists a unique solution $(\rho_{m,k})_{k\leq m}$ of (\ref{SDEsystem}), and thus (\ref{m-proj}) is well-posed in $H_m$. Applying It\^o's formula to $\rho_{m}$, it follows that for every $t\in[0,T]$, 
\begin{align*}
&\frac{1}{2}\|\rho_{m}(t)\|_{L^2(\mathbb{T}^d)}^2+\int_0^{t}\|\nabla\rho_{m}(s)\|_{L^2(\mathbb{T}^d)}^2ds	\\
=&\frac{1}{2}\|P_mu_0\|_{L^2(\mathbb{T}^d)}^2+\varepsilon^{\frac{1}{2}}\int_0^{t}\int_{\mathbb{T}^d}\rho_{m}P_m\nabla\cdot(G_0(\rho_{m})dW^{\delta,m}(s))dx\\
&+\frac{\varepsilon}{2}\int_0^{t}\int_{\mathbb{T}^d}G_0^{(1)}(\rho_{m})^2|\nabla\rho_{m}|^2F_1^m(\delta)dxds+\frac{\varepsilon}{2}\int_0^{t}\int_{\mathbb{T}^d}G_0(\rho_{m})^2F_3^m(\delta)dxds,
\end{align*}
where
\begin{align*}
&F_1^m(\delta)=\sum_{k\leq m,\theta=1,2}(\eta_{\delta}\ast e_{k,\theta})^2\lesssim\delta^{-d},\quad F_3^m(\delta)=\sum_{k\leq m,\theta=1,2}|\nabla \eta_{\delta}\ast e_{k,\theta}|^2\lesssim\delta^{-d-2}.
\end{align*}
Thanks to the boundedness of $G_0$ and $G^{(1)}_0(\cdot)$, taking expectations, choosing $\varepsilon_0=\|G^{(1)}_0\|_{L^{\infty}(\mathbb{R})}^{-1}\delta^{d}$, we find that for every $\varepsilon\in(0,\varepsilon_0)$,  
\begin{align}\label{H1-reg-N}
	\frac{1}{2}\mathbb{E}\|\rho_{m}(t)\|_{L^2(\mathbb{T}^d)}^2+\mathbb{E}\int_0^{t}\|\nabla\rho_{m}(s)\|_{L^2(\mathbb{T}^d)}^2ds\lesssim\frac{1}{2}\|u_0\|_{L^2(\mathbb{T}^d)}^2+\frac{\varepsilon}{2}\|G_0\|_{L^{\infty}(\mathbb{R})}^2\delta^{-d-2}T.
\end{align}
With the help of the proof of \cite[Theorem 5.1.3]{LR}, a compactness argument and the uniqueness of (\ref{approxeq}) show that there exists a subsequence (still denotes by $(\rho_m)_{m\in\mathbb{N}_+}$), such that 
\begin{align*}
\rho_m\rightharpoonup\rho,	
\end{align*}
weakly in $L^2(\Omega;L^2([0,T];L^2(\mathbb{T}^d)))$, as $m\rightarrow\infty$. With the help of (\ref{H1-reg-N}), there exists a subsequence (still denotes by $(\rho_m)_{m\in\mathbb{N}_+}$) and $f\in L^2(\Omega;L^2([0,T];L^2(\mathbb{T}^d)))$ such that
\begin{align*}
\nabla\rho_{m}\rightharpoonup f,	
\end{align*}
weakly in $L^2(\Omega;L^2([0,T];L^2(\mathbb{T}^d)))$, as $m\rightarrow\infty$. As a consequence, for every $\varphi\in C^{\infty}(\mathbb{T}^d)$, $A\in\mathcal{F}$, the integration by parts formula implies that  
\begin{align*}
	&\mathbb{E}\Big(I_A\int_0^T\langle f(s),\varphi\rangle ds\Big)=\lim_{m\rightarrow+\infty}\mathbb{E}\Big(I_A\int_0^T\langle \nabla\rho_m(s),\varphi\rangle ds\Big)\\
	=&-\lim_{m\rightarrow+\infty}\mathbb{E}\Big(I_A\int_0^T\langle \rho_m,\nabla\varphi\rangle ds\Big)=-\mathbb{E}\Big(I_A\int_0^T\langle \rho(s),\nabla\varphi\rangle ds\Big).
\end{align*}
This shows that 
\begin{align*}
\nabla\rho_{m}\rightharpoonup\nabla\rho,	 
\end{align*}
weakly in $L^2(\Omega;L^2([0,T];L^2(\mathbb{T}^d)))$, as $m\rightarrow\infty$. Then the lower semi-continuity of the $L^2(\Omega;L^2([0,T];L^2(\mathbb{T}^d)))$-norm implies that  
\begin{align}\label{H1-reg}
\nabla\rho\in L^2(\Omega;L^2([0,T];L^2(\mathbb{T}^d))). 	
\end{align}

{\bf{Step 2. $L^{\infty}(\mathbb{T}^d)$-estimate of the variational solution.}} Let $\psi(\zeta)=(\zeta-K)^+$ and $\psi_n=\psi\ast\eta_n$, where
 $\eta_n$ is a standard convolution kernel for $n\geq1$.  For every $\alpha\in[1,\infty)$, applying It\^o's formula (see \cite[Theorem 3.1]{Kryito}) to $\psi_n(\rho)^{\alpha+1}$, combining with  (\ref{H1-reg}), by a similar procedure in \cite[Theorem 3.9]{DFG} and passing to the limit  $n\rightarrow\infty$, we obtain that
\begin{align}\label{moser-ito}
d\int_{\mathbb{T}^d}\psi^{\alpha+1}(\rho)dx=&-\int_{\mathbb{T}^d}|\nabla\psi^{\frac{\alpha+1}{2}}(\rho)|^2 dxdt+\varepsilon^{\frac{1}{2}}\int_{\mathbb{T}^d}(\alpha+1)\psi^{\alpha}(\rho)\nabla\cdot(G_0(\rho)dW^{\delta}(t))dx\notag\\
&+\frac{\varepsilon}{2}\int_{\mathbb{T}^d}\alpha(\alpha+1)\psi^{\alpha-1}(\rho)G_0^{(1)}(\rho)^2|\nabla\rho|^2F_1(\delta)dxdt\notag\\
&+\frac{\varepsilon}{2}\int_{\mathbb{T}^d}\alpha(\alpha+1)\psi^{\alpha-1}(\rho)G_0(\rho)^2F_3(\delta)dxdt,
\end{align}
where by using the definition of $\eta_{\delta}$, we have that 
\begin{align*}
&F_1(\delta)=\sum_{k\geq0,\theta=1,2}(\eta_{\delta}\ast e_{k,\theta})^2\lesssim\delta^{-d},\quad F_3(\delta)=\sum_{k\geq0,\theta=1,2}|\nabla \eta_{\delta}\ast e_{k,\theta}|^2\lesssim\delta^{-d-2}.
\end{align*}
For the martingale term, the chain rule implies that  
\begin{align*}
	&\varepsilon^{\frac{1}{2}}\int_0^{t}\int_{\mathbb{T}^d}(\alpha+1)\psi^{\alpha}(\rho)\nabla\cdot(G_0(\rho)dW^{\delta}(t))dx=I_1+I_2,
\end{align*}
where
\begin{align*}
	I_1=&\varepsilon^{\frac{1}{2}}\int_0^{t}\int_{\mathbb{T}^d}(\alpha+1)\psi^{\alpha}(\rho)G_0^{(1)}(\rho)\nabla\rho\cdot dW^{\delta}(t)dx,\\
	I_2=&\varepsilon^{\frac{1}{2}}\int_0^{t}\int_{\mathbb{T}^d}(\alpha+1)\psi^{\alpha}(\rho)G_0(\rho)\nabla\cdot dW^{\delta}(t)dx.
\end{align*}
We denote the conditional expectation by $\mathbb{E}_{\mathcal{F}_0}(\cdot)=\mathbb{E}(\cdot|\mathcal{F}_0)$. For any $\{\mathcal{F}(t)\}_{t\in[0,T]}$-stopping time $\tau$ with $\tau\in[0,T]$, $\mathbb{P}$-almost surely, applying Burkholder-Davis-Gundy's inequality to $I_1$, by H\"older's inequality and the boundedness of $G_0^{(1)}$, 
\begin{align*}
&\mathbb{E}_{\mathcal{F}_0}\Big[\sup_{t\in[0,\tau]}I_1\Big]\\
=&\mathbb{E}_{\mathcal{F}_0}\Big[\sup_{t\in[0,\tau]}\Big|\varepsilon^{\frac{1}{2}}\int_0^{t}\int_{\mathbb{T}^d}(\alpha+1)\psi^{\alpha}(\rho)G_0^{(1)}(\rho)\nabla\rho\cdot dW^{\delta}(s)dx\Big|\Big]\\
=&\mathbb{E}_{\mathcal{F}_0}\Big[\sup_{t\in[0,\tau]}\Big|2\varepsilon^{\frac{1}{2}}\int_0^t\int_{\mathbb{T}^d} G_0^{(1)}(\rho)\psi^{\frac{\alpha+1}{2}}(\rho)\nabla\psi^{\frac{\alpha+1}{2}}(\rho)\cdot dW^{\delta}(s)dx\Big|\Big]\\
\leq&2\|G_0^{(1)}\|_{L^{\infty}(\mathbb{R})}^{\frac{1}{2}}\varepsilon^{\frac{1}{2}}F_1(\delta)^{\frac{1}{2}}\mathbb{E}_{\mathcal{F}_0}\Big[\Big(\int_0^{\tau}\Big(\int_{\mathbb{T}^d}\psi^{\alpha+1}(\rho)dx\Big)\Big(\int_{\mathbb{T}^d}|\nabla\psi^{\frac{\alpha+1}{2}}(\rho)|^2dx\Big)ds\Big)^{\frac{1}{2}}\Big]\\
\leq&2\|G_0^{(1)}\|_{L^{\infty}(\mathbb{R})}^{\frac{1}{2}}\varepsilon^{\frac{1}{2}}F_1(\delta)^{\frac{1}{2}}\mathbb{E}_{\mathcal{F}_0}\Big[\Big(\sup_{t\in[0,\tau]}\int_{\mathbb{T}^d}\psi^{\alpha+1}(\rho)dx\Big)^{\frac{1}{2}}\Big(\int_0^{\tau}\int_{\mathbb{T}^d}|\nabla\psi^{\frac{\alpha+1}{2}}(\rho)|^2dxds\Big)^{\frac{1}{2}}\Big].
\end{align*}
By Young's inequality, it follows that 
\begin{align*}
	&\mathbb{E}_{\mathcal{F}_0}\Big[\sup_{t\in[0,\tau]}\Big|\varepsilon^{\frac{1}{2}}\int_0^t\int_{\mathbb{T}^d}(\alpha+1)\psi^{\alpha}(\rho)G_0^{(1)}(\rho)\nabla\rho\cdot dW^{\delta}(s)dx\Big|\Big]\\
	\leq&\frac{1}{2}\mathbb{E}_{\mathcal{F}_0}\Big[\sup_{t\in[0,\tau]}\int_{\mathbb{T}^d}\psi^{\alpha+1}(\rho)dx\Big]+\varepsilon\|G_0^{(1)}\|_{L^{\infty}(\mathbb{R})}F_1(\delta)\mathbb{E}_{\mathcal{F}_0}\Big[\int_0^{\tau}\int_{\mathbb{T}^d}|\nabla\psi^{\frac{\alpha+1}{2}}(\rho)|^2dxds\Big]. 
\end{align*}
Similarly, applying Burkholder-Davis-Gundy's inequality to $I_2$, by H\"older's inequality and the boundedness of $G_0^{(1)}$, 
\begin{align*}
\mathbb{E}_{\mathcal{F}_0}\Big[\sup_{t\in[0,\tau]}|I_2|\Big]\leq&\frac{1}{4}\mathbb{E}_{\mathcal{F}_0}\Big[\sup_{t\in[0,\tau]}\int_{\mathbb{T}^d}\psi^{\alpha+1}(\rho)dx\Big]+\varepsilon\|G_0\|_{L^{\infty}(\mathbb{R})}F_3(\delta)\mathbb{E}_{\mathcal{F}_0}\Big[\int_0^{\tau}\int_{\mathbb{T}^d}\psi^{\alpha-1}(\rho)dxdt\Big].
\end{align*}

Combining with (\ref{moser-ito}), by the boundedness of $G_0$ and $G^{(1)}_0(\cdot)$, choosing $\varepsilon_0=\|G^{(1)}_0\|_{L^{\infty}(\mathbb{R})}^{-1}\delta^{d}$, then there exists a constant $c=c(G_0)$ independent of $\varepsilon$ such that for every $\varepsilon\in(0,\varepsilon_0)$,
\begin{align*}
\mathbb{E}_{\mathcal{F}_0}\Big[\sup_{t\in[0,\tau]}\int_{\mathbb{T}^d}\psi^{\alpha+1}(\rho)dx+\int^{\tau}_0\int_{\mathbb{T}^d}|\nabla\psi^{\frac{\alpha+1}{2}}(\rho)|^2dxdt\Big]\leq c\alpha^2\varepsilon\delta^{-d-2}\mathbb{E}_{\mathcal{F}_0}\int_0^{\tau}\int_{\mathbb{T}^d}\psi^{\alpha-1}(\rho)dxdt.
\end{align*}

For every $\alpha\in[1,+\infty)$, set $n_{\alpha}=(\alpha+1)^{-1}$. By \cite[Chapter 4, Proposition 4.7, Exercise 4.3]{RY99}, there exists a constant $c=c(G_0)$ such that  
\begin{align}\label{moser-1}
&\mathbb{E}\Big[\Big(\sup_{t\in[0,T]}\int_{\mathbb{T}^d}\psi^{\alpha+1}(\rho)dx+\int^T_0\int_{\mathbb{T}^d}|\nabla\psi^{\frac{\alpha+1}{2}}(\rho)|^2dxdt\Big)^{\frac{1}{\alpha+1}}\Big]\notag\\
\leq&\frac{n_{\alpha}^{-n_{\alpha}}}{1-n_{\alpha}}(c\alpha^2\varepsilon \delta^{-d-2})^{n_{\alpha}}\mathbb{E}\Big[\|\psi(\rho)\|_{L^{\alpha-1}(\mathbb{T}^d\times[0,T])}\Big]^{\frac{\alpha-1}{\alpha+1}}.
\end{align}
For $d>2$, set $e\in(0,2)$,  
\begin{align}\label{thetaq}
\theta=\frac{d}{2+d},\ \ q=\frac{(e+d)(\alpha+1)}{d}.	
\end{align}
By H\"older's inequality and the $L^p$-interpolation inequality, we find that 
\begin{align}\label{Lq-moser}
\|\psi(\rho)\|_{L^q(\mathbb{T}^d\times[0,T])}=&\Big(\int_0^T\int_{\mathbb{T}^d}\psi(\rho)^{\frac{(e+d)(\alpha+1)}{d}}dx dt\Big)^{\frac{d}{(d+e)(\alpha+1)}}\notag\\
\leq&\Big(\int_0^T\|\psi(\rho)\|_{L^{\frac{(e+d)(\alpha+1)}{d}}(\mathbb{T}^d)}^{\frac{(e+d)(\alpha+1)}{d}}dt\Big)^{\frac{d}{(d+e)(\alpha+1)}}\notag\\
\leq&\Big(\int_0^T\|\psi(\rho)\|_{L^{\frac{(2+d)(\alpha+1)}{d}}(\mathbb{T}^d)}^{\frac{(2+d)(\alpha+1)}{d}}dt\Big)^{\frac{d}{(d+2)(\alpha+1)}}\Big(\int_0^T1dt\Big)^{\frac{d(2-e)}{(d+e)(\alpha+1)(2+d)}}\notag\\
\leq&\Big(T^{\frac{d(2-e)}{(d+e)(2+d)}}\Big)^{\frac{1}{\alpha+1}}\Big(\int_0^T\|\psi(\rho)\|_{L^{\alpha+1}(\mathbb{T}^d)}^{\frac{(2+d)(\alpha+1)}{d}\theta'}\|\psi(\rho)\|_{L^{\frac{d(\alpha+1)}{d-2}}(\mathbb{T}^d)}^{\frac{(2+d)(\alpha+1)(1-\theta')}{d}}dt\Big)^{\frac{d}{(d+2)(\alpha+1)}},	
\end{align}

where $\theta'$ satisfies that 
\begin{align*}
\frac{d}{(2+d)(\alpha+1)}=\frac{\theta'}{\alpha+1}+\frac{1-\theta'}{(\frac{d}{d-2})(\alpha+1)}. 
\end{align*}
By a direct calculation, $\theta'=\frac{2}{2+d}$, $1-\theta'=\frac{d}{2+d}$, and therefore 
\begin{align*}
	\|\psi(\rho)\|_{L^q(\mathbb{T}^d\times[0,T])}\leq\Big(T^{\frac{d(2-e)}{(d+e)(2+d)}}\Big)^{\frac{1}{\alpha+1}}\|\psi(\rho)\|_{L^{\infty}([0,T];L^{\alpha+1}(\mathbb{T}^d))}^{1-\theta}\|\psi(\rho)\|_{L^{\alpha+1}([0,T];L^{\frac{d(\alpha+1)}{d-2}}(\mathbb{T}^d))}^{\theta}.
\end{align*}
Combining with the Sobolev's embedding theory, let $c=\sqrt{T}+1$, it follows that 
\begin{align*}
	\|\psi(\rho)\|_{L^q(\mathbb{T}^d\times[0,T])}\leq&\Big(T^{\frac{d(2-e)}{(d+e)(2+d)}}\Big)^{\frac{1}{\alpha+1}}\|\psi(\rho)\|_{L^{\infty}([0,T];L^{\alpha+1}(\mathbb{T}^d))}^{1-\theta}\|\psi(\rho)^{\frac{\alpha+1}{2}}\|_{L^2([0,T];L^{\frac{2d}{d-2}}(\mathbb{T}^d))}^{\frac{2\theta}{\alpha+1}}\\
	\leq&\Big(T^{\frac{d(2-e)}{(d+e)(2+d)}}\Big)^{\frac{1}{\alpha+1}}\|\psi(\rho)\|_{L^{\infty}([0,T];L^{\alpha+1}(\mathbb{T}^d))}^{1-\theta}\\
	&\cdot\Big(c\Big(\|\psi(\rho)\|_{L^{\infty}([0,T];L^{\alpha+1}(\mathbb{T}^d))}^{\frac{\alpha+1}{2}}+\|\nabla\psi^{\frac{\alpha+1}{2}}(\rho)\|_{L^2([0,T];L^2(\mathbb{T}^d))}\Big)\Big)^{\frac{2\theta}{\alpha+1}}. 
\end{align*}
By H\"older's inequality, the inequality $(x+y)^2\leq2x^2+2y^2$, for $x,y\geq0$, and by $\theta\in(0,1)$, it follows from (\ref{moser-1}) that 
\begin{align*}
	\mathbb{E}\|\psi(\rho)\|_{L^{\frac{(e+d)(\alpha+1)}{d}}(\mathbb{T}^d\times[0,T])}\lesssim\Big(\frac{n_{\alpha}^{-n_{\alpha}}}{1-n_{\alpha}}\Big)\Big(c\alpha^2\varepsilon \delta^{-d-2}T^{\frac{d(2-e)}{(d+e)(2+d)}}\Big)^{n_{\alpha}}\mathbb{E}\Big[\|\psi(\rho)\|_{L^{\alpha-1}(\mathbb{T}^d\times[0,T])}\Big]^{\frac{\alpha-1}{\alpha+1}}. 
\end{align*}
Now we are ready to employ a standard Moser iteration argument. Set 
\begin{align*}
\alpha_0=0,\ \ \alpha_k=\frac{e+d}{d}(\alpha_{k-1}+2),\ \ \beta_k=\alpha_{k-1}+1,\ \text{for }k\in\mathbb{N}/\{0\},  	
\end{align*}
a recursive computation implies that for $k\in\mathbb{N}/\{0\}$, there exists a constant $c>0$ such that 
\begin{align*}
&\mathbb{E}\|\psi(\rho)\|_{L^{\alpha_k}(\mathbb{T}^d\times[0,T])}\\
\leq&\prod_{r=1}^k\Big(\frac{n_{\beta_r}^{-n_{\beta_r}}}{1-n_{\beta_r}}(c\beta_r)^{2n_{\beta_r}}\Big)^{\prod_{s=r+1}^k\frac{\beta_s-1}{\beta_s+1}}\Big(\varepsilon \delta^{-d-2}T^{\frac{d(2-e)}{(d+e)(2+d)}}\Big)^{\sum_{r=1}^kn_{\beta_r}\Pi^k_{s=r+1}\frac{\beta_s-1}{\beta_s+1}}.
\end{align*}
For more details of the computation, and an analysis of the convergence of series, see the proof of \cite[Theorem 3.9]{DFG}. By using the same argument therein, taking $k\rightarrow\infty$, we obtain that there exists a constant $C=C(T)$ depending on $T$ with $1\lesssim\lim_{T\rightarrow0}C(T)<\infty$, and a constant $\gamma'>0$, such that 
\begin{align*}
\mathbb{E}\|\psi(\rho)\|_{L^{\infty}(\mathbb{T}^d\times[0,T])}\leq C	\Big(\varepsilon \delta^{-d-2}T^{\frac{d(2-e)}{(d+e)(2+d)}}\Big)^{\gamma'}.
\end{align*}
For the cases of $d=1$ and $d=2$, we choose $b\in(0,d)$, let $\theta=\frac{d}{b+d}$ and set $q=\frac{(e+d)(\alpha+1)}{d}$ with some $e\in(0,b)$ in (\ref{thetaq}). Returning to (\ref{Lq-moser}), we have
\begin{align}\label{Lq-moser=2}
\|\psi(\rho)\|_{L^q(\mathbb{T}^d\times[0,T])}\leq&\Big(\int_0^T\|\psi(\rho)\|_{L^{\frac{(b+d)(\alpha+1)}{d}}(\mathbb{T}^d)}^{\frac{(b+d)(\alpha+1)}{d}}dt\Big)^{\frac{d}{(d+b)(\alpha+1)}}\Big(\int_0^T1dt\Big)^{\frac{d(b-e)}{(d+e)(\alpha+1)(b+d)}}\notag\\
\leq&\Big(T^{\frac{d(b-e)}{(d+e)(b+d)}}\Big)^{\frac{1}{\alpha+1}}\Big(\int_0^T\|\psi(\rho)\|_{L^{\alpha+1}(\mathbb{T}^d)}^{\frac{(b+d)(\alpha+1)}{d}\theta'}\|\psi(\rho)\|_{L^{\frac{d(\alpha+1)}{d-b}}(\mathbb{T}^d)}^{\frac{(b+d)(\alpha+1)(1-\theta')}{d}}dt\Big)^{\frac{d}{(d+b)(\alpha+1)}},	
\end{align}
where $\theta'$ satisfies that 
\begin{align*}
\frac{d}{(b+d)(\alpha+1)}=\frac{\theta'}{\alpha+1}+\frac{1-\theta'}{(\frac{d}{d-b})(\alpha+1)}. 
\end{align*}
By direct calculation, $\theta'=\frac{b}{b+d}$, $1-\theta'=\frac{d}{b+d}$ and therefore 
\begin{align*}
	\|\psi(\rho)\|_{L^q(\mathbb{T}^d\times[0,T])}\leq\Big(T^{\frac{d(b-e)}{(d+e)(2+d)}}\Big)^{\frac{1}{\alpha+1}}\|\psi(\rho)\|_{L^{\infty}([0,T];L^{\alpha+1}(\mathbb{T}^d))}^{1-\theta}\|\psi(\rho)\|_{L^{\alpha+1}([0,T];L^{\frac{d(\alpha+1)}{d-b}}(\mathbb{T}^d))}^{\theta}.
\end{align*}
Combining with Sobolev's embedding, with $c=\sqrt{T}+1$, it follows that 
\begin{align*}
	\|\psi(\rho)\|_{L^q(\mathbb{T}^d\times[0,T])}\leq&\Big(T^{\frac{d(b-e)}{(d+e)(b+d)}}\Big)^{\frac{1}{\alpha+1}}\|\psi(\rho)\|_{L^{\infty}([0,T];L^{\alpha+1}(\mathbb{T}^d))}^{1-\theta}\|\psi(\rho)^{\frac{\alpha+1}{2}}\|_{L^2([0,T];L^{\frac{2d}{d-b}}(\mathbb{T}^d))}^{\frac{2\theta}{\alpha+1}}\\
	\leq&\Big(T^{\frac{d(b-e)}{(d+e)(b+d)}}\Big)^{\frac{1}{\alpha+1}}\|\psi(\rho)\|_{L^{\infty}([0,T];L^{\alpha+1}(\mathbb{T}^d))}^{1-\theta}\\
	&\cdot\Big(c\Big(\|\psi(\rho)\|_{L^{\infty}([0,T];L^{\alpha+1}(\mathbb{T}^d))}^{\frac{\alpha+1}{2}}+\|\nabla\psi^{\frac{\alpha+1}{2}}(\rho)\|_{L^2([0,T];L^2(\mathbb{T}^d))}\Big)\Big)^{\frac{2\theta}{\alpha+1}}. 
\end{align*}
Then the same approach as for the Moser iteration argument can be applied.

Analogous arguments can be employed for the estimates of $(\rho-K')^{-}$ and for the case of non-conservative noise, which completes the proof.

\end{proof}

We note that, as long as $\rho^{\varepsilon,\delta}$ takes values in the interval $\Big[{\rm{ess\inf}}u_0-\gamma,{\rm{ess\sup}}u_0+\gamma\Big]$, by the definition of $G_0$, $\rho^{\varepsilon,\delta}$ is a local in time solution to (\ref{SHE01}). We can therefore develop a local-in-time well-posedness theory for (\ref{SHE01}) by deriving estimates on the stopping time
\begin{equation}\label{randomtime}
\tau^{\varepsilon,\delta}_{\gamma}:=\inf\Big\{t\in[0,T];{\rm{ess\sup}}_{x\in\mathbb{T}^d}\rho^{\varepsilon,\delta}(t,x)>K+\gamma,\text{or }{\rm{ess\inf}}_{x\in\mathbb{T}^d}\rho^{\varepsilon,\delta}(t,x)<K'-\gamma\Big\},
\end{equation}
where $K:={\rm{ess}}\sup u_0$, $K':={\rm{ess}}\inf u_0$.

With the help of the $L^{\infty}$-estimate (\ref{Linfinity}), we will next show that $\tau^{\varepsilon,\delta}_{\gamma}$ is an $\{\mathcal{F}(t)\}_{t\in[0,T]}$-stopping time and it is $\mathbb{P}$-almost surely positive. Additionally, the following Lemma provides estimates on the asymptotic behavior of the stopping time in the small noise regime.
\begin{lemma}\label{stopproperty}
Under the same hypotheses as Lemma \ref{variationwellposed-2}. Let $\varepsilon,\delta>0$. In the conservative case, we further assume that $\varepsilon<\varepsilon_0$, where $\varepsilon_0$ is the constant that appears in Lemma \ref{variationwellposed-2}. Let $\rho^{\varepsilon,\delta}$ be the $H^{-1}(\mathbb{T}^d)$-variational solution (resp. mild solution) for  (\ref{approxeq}) (resp. (\ref{approxeq-2})). Let $\tau^{\varepsilon,\delta}_{\gamma}$ be defined by (\ref{randomtime}). Then $\tau^{\varepsilon,\delta}_{\gamma}$ is a $\mathbb{P}$-almost surely positive stopping time with respect to $\{\mathcal{F}(t)\}_{t\in[0,T]}$. Moreover, if
\begin{align}\label{localscale-1}
	\lim_{\varepsilon\rightarrow0}\Big(\varepsilon \delta(\varepsilon)^{-d}I_{\{i=1\}}+\varepsilon \delta(\varepsilon)^{-d-2}I_{\{i=2\}}\Big)=0,
\end{align}
then 
\begin{equation}\label{asymtrandom}
\lim_{\varepsilon\rightarrow0}\mathbb{P}(\tau^{\varepsilon,\delta(\varepsilon)}_{\gamma}>t)=1.
\end{equation}

\end{lemma}
\begin{proof}
{\bf{Step 1. $\tau^{\varepsilon,\delta}_{\gamma}$ is a stopping time.}}  We first show that  
\begin{align}\label{equirepresen}
\tau^{\varepsilon,\delta}_{\gamma}=\inf\Big\{t\in[0,T];\|\rho^{\varepsilon,\delta}(t)-\frac{K+K'}{2}\|_{L^{\infty}(\mathbb{T}^d)}>\gamma+\frac{K-K'}{2}\Big\}\wedge T.
\end{align}
Let $m_L$ be the Lebesgue measure on $\mathbb{T}^d$. By the definition of $L^{\infty}(\mathbb{T}^d)$-norm, we find that for any $f\in L^{\infty}(\mathbb{T}^d)$, 
\begin{align*}
&\|f-\frac{K+K'}{2}\|_{L^{\infty}(\mathbb{T}^d)}>\gamma+\frac{K-K'}{2},\\
\iff&m_L\Big(|f-\frac{K+K'}{2}|>\gamma+\frac{K-K'}{2}\Big)>0,\\
\iff&m_L\Big(\Big\{f-\frac{K+K'}{2}>\gamma+\frac{K-K'}{2},f>\frac{K+K'}{2}\Big\}\\
&\cup\Big\{\frac{K+K'}{2}-f>\gamma+\frac{K-K'}{2},f\leq\frac{K+K'}{2}\Big\}\Big)>0,\\
\iff&m_L(f>K+\gamma)>0,\ \text{or }m_L(f<K'-\gamma)>0,\\
\iff&{\rm{ess\sup}}_{x\in\mathbb{T}^d}f(x)>K+\gamma,\ \text{or }{\rm{ess\inf}}_{x\in\mathbb{T}^d}f(x)<K'-\gamma. 	
\end{align*}
This shows that (\ref{equirepresen}) holds.

Set 
\begin{align*}
A_{\gamma}=\{\rho\in L^{\infty}(\mathbb{T}^d):\|\rho-\frac{K+K'}{2}\|_{L^{\infty}(\mathbb{T}^d)}\leq\gamma+\frac{K-K'}{2}\}. 	
\end{align*}
We claim that $A_{\gamma}$ is a closed set in $H^{-1}(\mathbb{T}^d)$. Indeed, let $(\rho_n)_{n\geq1}\subset A_{\gamma}$ be a sequence such that $\rho_n\rightarrow\rho$ in $H^{-1}(\mathbb{T}^d)$ as $n\rightarrow\infty$. By the definition of $A_{\gamma}$, we have that 
\begin{align*}
\|\rho_n\|_{L^{\infty}(\mathbb{T}^d)}\leq K+\gamma. 
\end{align*}
As a consequence, there exists a subsequence $(\rho_{n_k})_{k\geq1}$, and $\tilde{\rho}\in L^{\infty}(\mathbb{T}^d)$ such that 
$\rho_{n_k}\rightarrow\tilde{\rho}$ weakly* in $L^{\infty}(\mathbb{T}^d)$, as $k\rightarrow\infty$. Thus for every $\psi\in C^{\infty}(\mathbb{T}^d)$, 
\begin{align*}
\lim_{k\rightarrow\infty}\langle\rho_{n_k},\psi\rangle=\langle\tilde{\rho},\psi\rangle=\langle\rho,\psi\rangle.  	
\end{align*}
Moreover, the fact $\tilde{\rho}\in L^{\infty}(\mathbb{T}^d)$ implies that $\rho,\tilde{\rho}$ can be extended as continuous functions on $L^2(\mathbb{T}^d)$, then the uniqueness in Riesz's representation theorem implies that $\rho=\tilde{\rho}$ almost everywhere. 
It follows that $\rho_{n_k}\rightarrow\rho$ weakly* in $L^{\infty}(\mathbb{T}^d)$, as $k\rightarrow\infty$. By the lower semi-continuity of the $L^{\infty}(\mathbb{T}^d)$-norm with respect to the weak* topology, we have that
\begin{align*}
\|\rho-\frac{K+K'}{2}\|_{L^{\infty}(\mathbb{T}^d)}\leq\liminf_{k\rightarrow\infty}\|\rho_{n_k}-\frac{K+K'}{2}\|_{L^{\infty}(\mathbb{T}^d)}\leq\gamma+\frac{K-K'}{2}. 	
\end{align*}

As a consequence, we find that $\rho\in A_{\gamma}$, which implies that $A_{\gamma}$ is closed in $H^{-1}(\mathbb{T}^d)$. 

It follows from Lemma \ref{variationwellposed} that $\rho^{\varepsilon,\delta}\in C([0,T];H^{-1}(\mathbb{T}^d))$, $\mathbb{P}$-almost surely. Due to the fact that $\{\mathcal{F}(t)\}_{t\in[0,T]}$ is right continuous, by the classical theory of stochastic analysis and stopping time, as seen, for example, in \cite[Proposition 2.3, Problem 2.6]{KS0}, $\tau^{\varepsilon,\delta}_{\gamma}$ is an $\{\mathcal{F}(t)\}_{t\in[0,T]}$-stopping time.  

{\bf{Step 2. The proof of (\ref{asymtrandom}).}} By the definition of $\tau^{\varepsilon,\delta}_{\gamma}$, for every $t\in(0,T)$, it follows that 
\begin{align*}
	\mathbb{P}(\tau^{\varepsilon,\delta}_{\gamma}> t)\geq&\mathbb{P}\Big(\sup_{s\in[0,t]}\|\rho^{\varepsilon,\delta}(s)-\frac{K+K'}{2}\|_{L^{\infty}(\mathbb{T}^d)}\leq\gamma+\frac{K-K'}{2}\Big).
\end{align*}
We claim that 
\begin{align}\label{esssup}
	&\Big\{\sup_{s\in[0,t]}\|\rho^{\varepsilon,\delta}(s)-\frac{K+K'}{2}\|_{L^{\infty}(\mathbb{T}^d)}\leq\gamma+\frac{K-K'}{2}\Big\}\notag\\
	=&\Big\{\|\rho^{\varepsilon,\delta}-\frac{K+K'}{2}\|_{L^{\infty}(\mathbb{T}^d\times[0,t])}\leq\gamma+\frac{K-K'}{2}\Big\}.
\end{align}
This can be proved by the fact that almost surely $\rho^{\varepsilon,\delta}\in C([0,T];H^{-1}(\mathbb{T}^d))$ and the lower semi-continuity of the $L^{\infty}(\mathbb{T}^d)$-norm. Consequently, we find that 
\begin{align*}
	\mathbb{P}(\tau^{\varepsilon,\delta}_{\gamma}{\color{blue}>}t)\geq&\mathbb{P}\Big(\|\rho^{\varepsilon,\delta}-\frac{K+K'}{2}\|_{L^{\infty}(\mathbb{T}^d\times[0,t])}\leq\gamma+\frac{K-K'}{2}\Big)\\
	=&1-\mathbb{P}\Big(\|\rho^{\varepsilon,\delta}-\frac{K+K'}{2}\|_{L^{\infty}(\mathbb{T}^d\times[0,t])}>\gamma+\frac{K-K'}{2}\Big).
\end{align*}
Using the definition of the $L^{\infty}(\mathbb{T}^d)$-norm,  
\begin{align*}
&\Big\{\|\rho^{\varepsilon,\delta}-\frac{K+K'}{2}\|_{L^{\infty}(\mathbb{T}^d\times[0,t])}>\gamma+\frac{K-K'}{2}\Big\}\\
\subset&\Big\{\|(\rho^{\varepsilon,\delta}-K)^{+}\|_{L^{\infty}(\mathbb{T}^d\times[0,t])}>\gamma\Big\}\cup\Big\{\|(\rho^{\varepsilon,\delta}-K')^{-}\|_{L^{\infty}(\mathbb{T}^d\times[0,t])}>\gamma\Big\}\\
\subset&\Big\{\|(\rho^{\varepsilon,\delta}-K)^{+}\|_{L^{\infty}(\mathbb{T}^d\times[0,t])}+\|(\rho^{\varepsilon,\delta}-K')^{-}\|_{L^{\infty}(\mathbb{T}^d\times[0,t])}>\gamma\Big\}. 
\end{align*}
In combination of Chebyshev's inequality, we get
\begin{align*}
\mathbb{P}(\tau^{\varepsilon,\delta}_{\gamma}{\color{blue}>}t)\geq&1-\mathbb{P}\Big((\|(\rho^{\varepsilon,\delta}-K)^+\|_{L^{\infty}(\mathbb{T}^d\times[0,t])}+\|(\rho^{\varepsilon,\delta}-K')^-\|_{L^{\infty}(\mathbb{T}^d\times[0,t])})>\gamma\Big)\\
\geq&1-\frac{1}{\gamma}\mathbb{E}\Big(\|(\rho^{\varepsilon,\delta}-K)^+\|_{L^{\infty}(\mathbb{T}^d\times[0,t])}+\|(\rho^{\varepsilon,\delta}-K')^-\|_{L^{\infty}(\mathbb{T}^d\times[0,t])}\Big).
\end{align*}
With the aid of (\ref{Linfinity}) in Lemma \ref{lem-moser-1} in the scaling regime $(\varepsilon,\delta(\varepsilon))$ satisfying $\lim_{\varepsilon\rightarrow0}\varepsilon \delta(\varepsilon)^{-d-2}=0$, it follows that
\begin{align*}
\lim_{\varepsilon\rightarrow0}\mathbb{P}(\tau^{\varepsilon,\delta(\varepsilon)}_{\gamma}{\color{blue}>}t)=1.
\end{align*}

{\bf {Step 3. The stopping time is $\mathbb{P}$-almost surely positive.}} For every $M>1$, 
\begin{align*}
\mathbb{P}(\tau^{\varepsilon,\delta}_{\gamma}\leq\frac{1}{M})\leq&\mathbb{P}\Big(\sup_{s\in[0,\frac{1}{M}]}\|\rho^{\varepsilon,\delta}(s)-\frac{K+K'}{2}\|_{L^{\infty}(\mathbb{T}^d)}>\gamma+\frac{K-K'}{2}\Big)\\
=&\mathbb{P}\Big(\|\rho^{\varepsilon,\delta}-\frac{K+K'}{2}\|_{L^{\infty}(\mathbb{T}^d\times[0,\frac{1}{M}])}>\gamma+\frac{K-K'}{2}\Big).	
\end{align*}
By using the definition of the $L^{\infty}(\mathbb{T}^d)$-norm again, for every $s\in[0,\frac{1}{M}]$, 
\begin{align*}
&\Big\{\|\rho^{\varepsilon,\delta}-\frac{K+K'}{2}\|_{L^{\infty}([0,\frac{1}{M}]\times\mathbb{T}^d)}>\gamma+\frac{K-K'}{2}\Big\}\\
\subset&\Big\{\|(\rho^{\varepsilon,\delta}-K)^{+}\|_{L^{\infty}([0,\frac{1}{M}]\times\mathbb{T}^d)}+\|(\rho^{\varepsilon,\delta}-K')^{-}\|_{L^{\infty}([0,\frac{1}{M}]\times\mathbb{T}^d)}>\gamma\Big\}. 
\end{align*}
Combining with Lemma \ref{lem-moser-1} and Chebyshev's inequality, there exists $\varepsilon_0=\varepsilon_0(\delta,G_0^{(1)})>0$, $\tilde{\gamma}=\tilde{\gamma}(d)>0$, $C=C(T)>0$ with $1\lesssim\lim_{T\rightarrow0}C(T)<\infty$, and $\gamma'>0$ independent on $\varepsilon,\delta,T,G_0$, such that 
\begin{align*}
	\mathbb{P}(\tau^{\varepsilon,\delta}_{\gamma}=0)=&\lim_{M\rightarrow\infty}\mathbb{P}(\tau^{\varepsilon,\delta}_{\gamma}\leq\frac{1}{M})\\
	\leq&\lim_{M\rightarrow\infty}\frac{1}{\gamma}\Big(\mathbb{E}\|(\rho^{\varepsilon,\delta}-K)^{+}\|_{L^{\infty}([0,\frac{1}{M}]\times\mathbb{T}^d)}+\mathbb{E}\|(\rho^{\varepsilon,\delta}-K')^{-}\|_{L^{\infty}([0,\frac{1}{M}]\times\mathbb{T}^d)}\Big)\\
	\leq&\frac{1}{\gamma}\lim_{M\rightarrow\infty}C(\varepsilon \delta^{-d-2}(\frac{1}{M})^{\tilde{\gamma}})^{\gamma'}=0.
\end{align*}
This shows that $\tau^{\varepsilon,\delta}_{\gamma}$ is $\mathbb{P}$-almost surely positive. The same argument can be carried out for the case of non-conservative noise, which completes the proof.

\end{proof}

The above Lemma \ref{stopproperty} implies that the local survival stopping times for  (\ref{approxeq}) and (\ref{approxeq-2}) are both
 almost surely positive. As a result, we are able to  get
the local in time well-posedness of  (\ref{SHE02}) and  (\ref{SHE01}) with irregular coefficient
$G$ satisfying Hypothesis H3. 
\begin{corollary}\label{localsolution}
Under the same hypotheses as Lemma \ref{variationwellposed-2}. Let $\varepsilon,\delta>0$, in the conservative case, we further assume that $\varepsilon<\varepsilon_0$, where $\varepsilon_0$ is the constant that appears in Lemma \ref{variationwellposed-2}. Let $\tau^{\varepsilon,\delta}_{\gamma}$ defined by (\ref{randomtime}). Then, there exists an $\varepsilon_0>0$ such that for every $\varepsilon\in(0,\varepsilon_0)$, there exists a unique local in time $H^{-1}(\mathbb{T}^d)$-variational solution of  (\ref{SHE01}) with initial data $u_0$, in the sense of Definition \ref{variation-local} and Definition \ref{localuniqueness}. Moreover, there exists a unique local in time mild solution of  (\ref{SHE02}) with initial data $u_0$, in the sense of Definition \ref{mild-local} and Definition \ref{localuniqueness}.
\end{corollary}
\begin{proof}
Let $\rho^{\varepsilon,\delta}$ be the $H^{-1}(\mathbb{T}^d)$-variational solution of (\ref{approxeq}) with initial data $u_0$. For $\gamma>0$, let $\tau^{\varepsilon,\delta}_{\gamma}$ be defined by (\ref{randomtime}). By Lemma \ref{variational-solu}, there exists an event $\tilde{\Omega}_1\in\mathcal{F}$ depends on $T, u_0$, with $\mathbb{P}(\tilde{\Omega}_1)=1$ such that for $\omega\in\tilde{\Omega}_1$, the solution $\rho^{\varepsilon,\delta}$ of  (\ref{approxeq}) satisfies that $\rho^{\varepsilon,\delta}(\omega)\in C([0,T];H^{-1}(\mathbb{T}^d))\cap L^2([0,T];L^2(\mathbb{T}^d))$. Therefore, it holds that 
\begin{align}\label{contH-1}
\rho^{\varepsilon,\delta}(\omega)\in C([0,\tau^{\varepsilon,\delta}_{\gamma}(\omega)];H^{-1}(\mathbb{T}^d))\cap L^2([0,\tau^{\varepsilon,\delta}_{\gamma}(\omega)];L^2(\mathbb{T}^d)),
\end{align}
for every $\omega\in\tilde{\Omega}_1$. By the definition of $\tau^{\varepsilon,\delta}_{\gamma}$, for any $n>0$, we have that $\mathbb{P}$-almost surely, 
\begin{equation*}
	\sup_{t\in[0,\tau^{\varepsilon,\delta}_{\gamma})}\|\rho^{\varepsilon,\delta}(t)-\frac{K+K'}{2}\|_{L^{\infty}(\mathbb{T}^d)}\leq\gamma+\frac{K-K'}{2}+\frac{1}{n}.  
\end{equation*}
Combining with (\ref{contH-1}), it follows that $\mathbb{P}$-almost surely, there exists a sequence $(\rho^{\varepsilon,\delta}(t_k))_{k\geq1}$, such that 
\begin{align*}
\rho^{\varepsilon,\delta}(t_k)\rightharpoonup\rho^{\varepsilon,\delta}(\tau^{\varepsilon,\delta}_{\gamma})	,
\end{align*}
weakly* in $L^{\infty}(\mathbb{T}^d)$, as $k\rightarrow\infty$. Therefore, the lower semi-continuity of the $L^{\infty}(\mathbb{T}^d)$-norm implies that $\mathbb{P}$-almost surely, $\|\rho^{\varepsilon,\delta}(\tau^{\varepsilon,\delta}_{\gamma})-\frac{K+K'}{2}\|_{L^{\infty}(\mathbb{T}^d)}\leq\gamma+\frac{K-K'}{2}+\frac{1}{n}$. Combining with Lemma \ref{stopproperty}, $(\rho^{\varepsilon,\delta},\tau^{\varepsilon,\delta}_{\gamma})$ is a local in time $H^{-1}(\mathbb{T}^d)$ variational solution for  (\ref{SHE01}). 

It remains to show the uniqueness up to $\tau^{\varepsilon,\delta}_{\gamma}$. Let $(u^1,\tau^1)$, $(u^2,\tau^2)$ be any two local in time $H^{-1}(\mathbb{T}^d)$-variational solutions of (\ref{SHE01}) with initial data $u_0$. We assume that $\tau^1,\tau^2\in(0,T]$, $
\mathbb{P}$-almost surely. The definition of the stopping time $\tau^{\varepsilon,\delta}_{\gamma}$ implies that $\mathbb{P}$-almost surely, 
\begin{align*}
\rho^{\varepsilon,\delta}(t,x)\in\Big[{\rm{ess\inf}}u_0-\gamma,{\rm{ess\sup}}u_0+\gamma\Big],   	
\end{align*}
for almost every $x\in\mathbb{T}^d$, $t<\tau^{\varepsilon,\delta}_{\gamma}$. Therefore, by the definition of $G_0$, we have that $\mathbb{P}$-almost surely, 
\begin{align*}
G(\rho^{\varepsilon,\delta}(t,x))=G_0(\rho^{\varepsilon,\delta}(t,x)),\ \ \text{for a.e. } x\in\mathbb{T}^d,\ \text{and for every }t<\tau^{\varepsilon,\delta}_{\gamma}	. 
\end{align*}
Recall that $K={\rm{ess}}\sup u_0$, $K'={\rm{ess}}\inf u_0$. Define the stopping time
\begin{equation*}
\tau_{1,\gamma}:=\inf\Big\{t\in[0,\tau^1);\|u^1(t)-\frac{K+K'}{2}\|_{L^{\infty}(\mathbb{T}^d)}>\gamma+\frac{K-K'}{2}\Big\}\wedge \tau^1. 
\end{equation*}
By subtracting $\rho^{\varepsilon,\delta}$ and $u^1$, we have that $\mathbb{P}$-almost surely, 
\begin{align*}
d(\rho^{\varepsilon,\delta}-u^1)=&\Delta(\rho^{\varepsilon,\delta}-u^1)+\varepsilon^{\frac{1}{2}}\nabla\cdot\Big((G_0(\rho^{\varepsilon,\delta})-G(u^1))dW^{\delta}(t)\Big)\\
=&\Delta(\rho^{\varepsilon,\delta}-u^1)+\varepsilon^{\frac{1}{2}}\nabla\cdot\Big((G_0(\rho^{\varepsilon,\delta})-G_0(u^1))dW^{\delta}(t)\Big),	
\end{align*}
holds in $(L^2(\mathbb{T}^d))^*$, for every $t<\tau^{\varepsilon,\delta}_{\gamma}\wedge\tau_{1,\gamma}$. By using It\^o's formula (see \cite[Theorem 4.2.5]{LR}), it follows from (\ref{H-1unique}) that there exists constant $C_{\delta}>0$ depending on the correlation structure of the noise, such that $\mathbb{P}$-almost surely, 
\begin{align*}
&\mathbb{E}\|\rho^{\varepsilon,\delta}(t,\omega)-u^1(t,\omega)\|_{H^{-1}(\mathbb{T}^d)}^2+\mathbb{E}\int_0^t\|\rho^{\varepsilon,\delta}(s,\omega)-u^1(s,\omega)\|_{L^2(\mathbb{T}^d)}^2ds\\
\leq&\Big(\frac{\varepsilon\|G_0^{(1)}(\cdot)\|_{L^{\infty}(\mathbb{R})}^2C_{\delta}}{2}\Big)\mathbb{E}\int_0^t\|\rho^{\varepsilon,\delta}(s,\omega)-u^1(s,\omega)\|_{L^2(\mathbb{T}^d)}^2ds,
\end{align*}
for every $t<\tau^{\varepsilon,\delta}_{\gamma}\wedge\tau_{1,\gamma}$. This implies that there exists $\tilde{\Omega}_1\in\mathcal{F}$ depending on $T, u_0$, with $\mathbb{P}(\tilde{\Omega}_1)=1$, such that for every $\varepsilon<\frac{1}{\|G_0^{(1)}(\cdot)\|_{L^{\infty}(\mathbb{R})}^2C_{\delta}}$, 
\begin{align}\label{locauniq-1}
	\rho^{\varepsilon,\delta}(t,x,\omega)=u^1(t,x,\omega),\ \ \text{for a.e. } x\in\mathbb{T}^d, 
\end{align}
for every $t<\tau^{\varepsilon,\delta}_{\gamma}(\omega)\wedge\tau_{1,\gamma}(\omega)$ and every $\omega\in\tilde{\Omega}_1$.

In the following, we claim that $\mathbb{P}$-almost surely, $\tau^{\varepsilon,\delta}_{\gamma}\wedge\tau^1=\tau_{1,\gamma}$. This can be proved by contradiction. Assume that $\tilde{\Omega}_1\cap\{\tau^{\varepsilon,\delta}_{\gamma}\wedge\tau^1>\tau_{1,\gamma}\}$ is not an empty set. For any $\omega\in\tilde{\Omega}_1\cap\{\tau^{\varepsilon,\delta}_{\gamma}\wedge\tau^1>\tau_{1,\gamma}\}$, 
\begin{align*}
\tau_{1,\gamma}(\omega)=&\inf\Big\{t\in[0,\tau^1(\omega));\|u^1(t,\omega)-\frac{K+K'}{2}\|_{L^{\infty}(\mathbb{T}^d)}>\gamma+\frac{K-K'}{2}\Big\}\wedge \tau^1(\omega)\\
\geq&\inf\Big\{t\in(0,\tau^{\varepsilon,\delta}_{\gamma}(\omega)\wedge\tau^1(\omega));\|u^1(t,\omega)-\frac{K+K'}{2}\|_{L^{\infty}(\mathbb{T}^d)}>\gamma+\frac{K-K'}{2}\Big\}\wedge \tau^1(\omega)\\
=&\inf\Big\{t\in(0,\tau^{\varepsilon,\delta}_{\gamma}(\omega)\wedge\tau^1(\omega));\|\rho^{\varepsilon,\delta}(t,\omega)-\frac{K+K'}{2}\|_{L^{\infty}(\mathbb{T}^d)}>\gamma+\frac{K-K'}{2}\Big\}\wedge \tau^1(\omega)\\
=&\tau^{\varepsilon,\delta}_{\gamma}(\omega)\wedge\tau^1(\omega),  
\end{align*}
which leads to a contradiction. Consequently, $\{\tau^{\varepsilon,\delta}_{\gamma}\wedge\tau^1>\tau_{1,\gamma}\}\subset\Omega/\tilde{\Omega}_1$. Similarly, we have $\{\tau^{\varepsilon,\delta}_{\gamma}\wedge\tau^1<\tau_{1,\gamma}\}\subset\Omega/\tilde{\Omega}_1$, which implies $\mathbb{P}(\tau^{\varepsilon,\delta}_{\gamma}\wedge\tau^1=\tau_{1,\gamma})=1$. 

The same argument can be applied to $u^2$ as well. As a consequence, we find that $\mathbb{P}$-almost surely,
\begin{equation*}
	u^1(t,x)=\rho^{\varepsilon,\delta}(t,x)=u^2(t,x),\ \ \text{for a.e. } x\in\mathbb{T}^d, 
\end{equation*}
for every $t<\tau^1\wedge\tau^2\wedge\tau^{\varepsilon,\delta}_{\gamma}$. Furthermore, since we have that $\mathbb{P}$-almost surely, 
	\begin{equation*}
\rho^{\varepsilon,\delta}\in C([0,T];H^{-1}(\mathbb{T}^d)).
\end{equation*}
Therefore, it follows that $\mathbb{P}$-almost surely, 
\begin{align}\label{conti-H-1}
\lim_{t\uparrow\tau^{\varepsilon,\delta}_{\gamma}}\|u^{1}(t)-\rho^{\varepsilon,\delta}(\tau^{\varepsilon,\delta}_{\gamma})\|_{H^{-1}(\mathbb{T}^d)}=&\lim_{t\uparrow\tau^{\varepsilon,\delta}_{\gamma}}\|u^{2}(t)-\rho^{\varepsilon,\delta}(\tau^{\varepsilon,\delta}_{\gamma})\|_{H^{-1}(\mathbb{T}^d)}\notag\\
=&\lim_{t\uparrow\tau^{\varepsilon,\delta}_{\gamma}}\|\rho^{\varepsilon,\delta}(t)-\rho^{\varepsilon,\delta}(\tau^{\varepsilon,\delta}_{\gamma})\|_{H^{-1}(\mathbb{T}^d)}=0. 	
\end{align}
This shows that $\mathbb{P}$-almost surely, 
\begin{equation*}
	u^1(t,x)=\rho^{\varepsilon,\delta}(t,x)=u^2(t,x),\ \ \text{for a.e. } x\in\mathbb{T}^d, 
\end{equation*}
for every $t\in[0,\tau^1\wedge\tau^2\wedge\tau^{\varepsilon,\delta}_{\gamma}]$. Then the local in time uniqueness for (\ref{SHE01}) up to $\tau^{\varepsilon,\delta}_{\gamma}$ holds in the sense of Definition \ref{localuniqueness}. Moreover, by Lemma \ref{lem-moser-1} and Lemma \ref{stopproperty}, there exists an $\varepsilon_0=\varepsilon_0(\delta,C,\gamma')>0$, such that for every $\varepsilon\in(0,\varepsilon_0)$, $\mathbb{P}(\tau^{\varepsilon,\delta}_{\gamma}>0)=1$, which implies the local in time well-posedness for (\ref{SHE01}).

For the case of non-conservative noise, the same argument shows that for every $\varepsilon>0$, there exists an unique local in time mild solution for (\ref{SHE02}). 

\end{proof}

With the help of the $L^p$-estimate, combining with Corollary \ref{localsolution}, the existence and the uniqueness of the local in time mild solution for (\ref{SHE01}) can be shown. 
\begin{lemma}\label{localmildboth}
Under the same hypotheses as Lemma \ref{variationwellposed-2}. Let $\delta>0$, $\varepsilon\in(0,\varepsilon_0)$, where $\varepsilon_0$ is the constant that appears in Lemma \ref{variationwellposed-2}. Let $(u^{\varepsilon,\delta},\tau^{\varepsilon,\delta}_{\gamma})$ be the local in time $H^{-1}(\mathbb{T}^d)$-variational solution of (\ref{SHE01}) with initial data $u_0$, then $u^{\varepsilon,\delta}$ is the local in time mild solution of (\ref{SHE01}) in the sense of Definition \ref{mild-local}. 	
\end{lemma}
\begin{proof}
By Definition \ref{variation-local}, we have $\mathbb{P}$-almost surely,  
\begin{align*}
u^{\varepsilon,\delta}(t,\omega)=&u_0+\int^t_0\Delta u^{\varepsilon,\delta}(s,\omega)ds+\varepsilon^{\frac{1}{2}}\int^t_0\nabla\cdot(G(u^{\varepsilon,\delta}(s))dW^{\delta}(s))(\omega)
\end{align*}
in $(L^2(\mathbb{T}^d))^*$ for every $t<\tau^{\varepsilon,\delta}_{\gamma}$. Using the assumption of $G$ implies $1_{ [0,\tau^{\varepsilon,\delta}]}G(u^{\varepsilon,\delta}(s,\cdot)) \in L^2([0,T]\times\Omega\times\mathbb{T}^d)$. Combined with \cite[Theorem 3.2]{GM}, Lemma \ref{gen-littlewoodpaley} and Lemma \ref{maximalLp-cons}, it follows that $(u^{\varepsilon,\delta},\tau^{\varepsilon,\delta}_{\gamma})$ is a local in time mild solution of (\ref{SHE01}). 
\end{proof}

\section{Speed of divergence for the expansion coefficients}\label{sec-5}
By \cite{walsh}, \cite{Daprato}, the mild solution of (\ref{coe}) (resp. (\ref{ccoe})) is defined by induction, via 
\begin{align}\label{mildcoe}
\bar{u}^{k,\delta}(t,x)=\int_0^t\Big\langle p(t-s,x-\cdot),\Big[\sum_{l=1}^{k-1}\frac{1}{l!}G^{(l)}(\bar{u}^{0,\delta}(s,\cdot))\mathcal{J}^{\delta}(k,l)(s,\cdot)\Big]dW^{\delta}(s)\Big\rangle,
\end{align}
and
\begin{align}\label{mildcoe-cons}
\bar{u}^{k,\delta}(t,x)=\int_0^t\Big\langle \nabla_x p(t-s,x-\cdot),\Big[\sum_{l=1}^{k-1}\frac{1}{l!}G^{(l)}(\bar{u}^{0,\delta}(s,\cdot))\mathcal{J}^{\delta}(k,l)(s,\cdot)\Big]dW^{\delta}(s)\Big\rangle,
\end{align}
for $k\geq1$, respectively. The following two theorems show that the stochastic integral in (\ref{mildcoe}) and (\ref{mildcoe-cons}) are well-defined and estimates the speed of divergence of (\ref{mildcoe}) and (\ref{mildcoe-cons}) respectively. 
\subsection{The case of nonconservative noise}

\begin{theorem}\label{div-speed}
Let $n\in\mathbb{N}_+$, $\delta>0$. Assume that the initial data $u_0$ satisfies Hypothesis H1, and $G$ satisfies Hypothesis H2. Let $\bar{u}^{n,\delta}$ be defined by (\ref{mildcoe}), and let $K_i(\delta,d)$ be defined by (\ref{Kdelta}). Then there exists a positive constant $C(G,p,n)<\infty$ such that for every $p\in[1,\infty)$, 
\begin{align}\label{e-13}
\sup_{t\in[0,T],x\in\mathbb{T}^d}\mathbb{E}|\bar{u}^{n,\delta}(t,x)|^p\leq C(G,p,n)K_1(\delta,d)^{\frac{pn}{2}}.
\end{align}

\end{theorem}
\begin{proof}
We argue by induction.

{\bf Step 1}. For $n=1$, $\bar{u}^{1,\delta}$ is written as
\begin{equation}\label{mildu1}
\bar{u}^{1,\delta}(t,x)=\int_0^t\langle p(t-s,x-\cdot),G(\bar{u}^{0,\delta}(s,\cdot))dW^{\delta}(s)\rangle.
\end{equation}
By the assumption on $G$ and $u_0$, by \cite{walsh}, the stochastic integral in (\ref{mildu1}) is well-defined, and it is the unique mild solution of (\ref{coe}). Combining with Lemma \ref{K2} and Lemma \ref{maximalLp-noncons}, the $L^p$-isometry of the stochastic integral (see \cite[Corollary 3.11]{NW}), the assumption {\color{blue}on} $G$ and $u_0$ implies that for every $p\in[2,\infty)$,  
\begin{align}\label{1-blowupspeed}
\mathbb{E}|\bar{u}^{1,\delta}(t,x)|^p\lesssim \sup_{t\in[0,T],x\in\mathbb{T}^d}\mathbb{E}|G(\bar{u}^{0,\delta}(t,x))|^pK_1(\delta,d)^{\frac{p}{2}}\leq C(G,p)K_1(\delta,d)^{\frac{p}{2}}.
\end{align}
Furthermore, using H\"older's inequality, we are able to see that (\ref{1-blowupspeed}) holds for every $p\in[1,2)$. 

{\bf Step 2}. For $k\in\mathbb{N}_+$, assume that (\ref{e-13}) holds for $n=1,\dots,k-1$. We aim to prove (\ref{e-13}) for $n=k$. The mild solution $\bar{u}^{k,\delta}$ satisfies
\begin{align}\label{milduk}
\bar{u}^{k,\delta}(t,x)=\int_0^t\Big\langle p(t-s,x-\cdot),\Big[\sum_{l=1}^{k-1}\frac{1}{l!}G^{(l)}(\bar{u}^{0,\delta}(s,\cdot))\mathcal{J}^{\delta}(k,l)(s,\cdot)\Big]dW^{\delta}(s)\Big\rangle,
\end{align}
where $\mathcal{J}^{\delta}(k,l)$ is defined by (\ref{Jkl}). By the Hypothesis H1, H2 on $G$ and $u_0$, and the induction hypothesis for $n=1,...,k-1$, by \cite{walsh}, the stochastic integral in (\ref{milduk}) is well-defined, and it is the unique mild solution of (\ref{coe}) (resp. (\ref{ccoe})). Combining with Lemma \ref{K2} and Lemma \ref{maximalLp-noncons},
\begin{align*}
\mathbb{E}|\bar{u}^{k,\delta}(t,x)|^p\leq&C(G,p)\sup_{s\in[0,T],y\in\mathbb{T}^d}\mathbb{E}\Big(\sum_{l=1}^{k-1}\frac{1}{l!}|\mathcal{J}^{\delta}(k,l)(s,y)|\Big)^pK_1(\delta,d)^{\frac{p}{2}}.
\end{align*}
For every $a_i\in\mathbb{R}$, $i=1,...,n$, there exists a constant depends on $p,n$ such that 
\begin{equation}\label{p-inequality}
(a_1+a_2+...+a_n)^p\leq C(p,n)(|a_1|^p+|a_2|^p+...+|a_n|^p).\end{equation}
With the help of the above inequality, it follows from the definition of $\mathcal{J}^{\delta}(k,l)$ in (\ref{Jkl}) and Lemma \ref{K2} that 
\begin{align*}
\mathbb{E}|\bar{u}^{k,\delta}(t,x)|^p\leq&K_1(\delta,d)^{\frac{p}{2}}C(G,p,k)\\
&\cdot\Big(\sum_{l=1}^{k-1}\frac{1}{|l!|^p}\sum_{(q_1,\dots,q_{k-l})\in\Lambda(k,l)}(\frac{l!}{q_1!\dots q_{k-l}!})^p\sup_{s\in[0,T],y\in\mathbb{T}^d}\mathbb{E}|\prod_{1\leq j\leq k-l}(\bar{u}^{j,\delta}(s,y))^{q_j}|^p\Big).
\end{align*}
For every $1\leq l\leq k-1$, let $(p_j)_{1\leq j\leq k-l}$ be a sequence of numbers with $p_j\in[1,\infty)$, $1\leq j\leq k-l$, and $\sum_{1\leq j\leq k-l}\frac{1}{p_j}=1$. By the induction hypothesis and H\"older's inequality, we have 
\begin{align*}
&\sum_{(q_1,\dots,q_{k-l})\in\Lambda(k,l)}(\frac{l!}{q_1!\dots q_{k-l}!})^p\sup_{s\in[0,T],y\in\mathbb{T}^d}\mathbb{E}|\prod_{1\leq j\leq k-l}(\bar{u}^{j,\delta}(s,y))^{q_j}|^p\\
\leq&\sum_{(q_1,\dots,q_{k-l})\in\Lambda(k,l)}(\frac{l!}{q_1!\dots q_{k-l}!})^p\sup_{s\in[0,T],y\in\mathbb{T}^d}\prod_{1\leq j\leq k-l}\Big(\mathbb{E}(\bar{u}^{j,\delta}(s,y))^{pq_jp_j}\Big)^{\frac{1}{p_j}}\\
\leq&\sum_{(q_1,\dots,q_{k-l})\in\Lambda(k,l)}(\frac{l!}{q_1!\dots q_{k-l}!})^p\prod_{1\leq j\leq k-l}\Big(\sup_{s\in[0,T],y\in\mathbb{T}^d}\mathbb{E}(\bar{u}^{j,\delta}(s,y))^{pq_jp_j}\Big)^{\frac{1}{p_j}}\\
\lesssim&\sum_{(q_1,\dots,q_{k-l})\in\Lambda(k,l)}(\frac{l!}{q_1!\dots q_{k-l}!})^p\prod_{1\leq j\leq k-l}\Big(K_1(\delta,d)^{\frac{jq_jp_jp}{2}}\Big)^{\frac{1}{p_j}}\\
=&\sum_{(q_1,\dots,q_{k-l})\in\Lambda(k,l)}(\frac{l!}{q_1!\dots q_{k-l}!})^pK_1(\delta,d)^{\sum_{1\leq j\leq k-l}\frac{jq_jp}{2}}.
\end{align*}
Further, the fact that $(q_1,\dots,q_{k-l})\in\Lambda(k,l)$ yields $\sum_{1\leq j\leq k-l}jq_j=k-1$. As a result, we conclude that
\begin{align*}
\sup_{t\in[0,T],x\in\mathbb{T}^d}\mathbb{E}|\bar{u}^{k,\delta}(t,x)|^p\leq& K_1(\delta,d)^{\frac{p}{2}}C(G,p,k)K_1(\delta,d)^{\frac{p(k-1)}{2}}\\
=&C(G,p,k)K_1(\delta,d)^{\frac{pk}{2}},
\end{align*}
which implies (\ref{e-13}) for $n=k$, $p\in[2,\infty)$. Induction completes the proof for $p\in[2,\infty)$. When $p\in[1,2)$, by H\"older's inequality, we have 
\begin{align*}
	\sup_{t\in[0,T],x\in\mathbb{T}^d}\mathbb{E}|\bar{u}^{n,\delta}(t,x)|^p\leq\Big(\sup_{t\in[0,T],x\in\mathbb{T}^d}\mathbb{E}|\bar{u}^{n,\delta}(t,x)|^2\Big)^{\frac{p}{2}}\leq C(G,p,n)K_1(\delta,d)^{\frac{np}{2}},
\end{align*}
which completes the proof. 
\end{proof}

\subsection{The case of conservative noise}

\begin{theorem}\label{div-speed-cons}
Let $n\in\mathbb{N}_+$, $\delta>0$. Assume that the initial data $u_0$ satisfies Hypothesis H1, and $G$ satisfies Hypothesis H2. Let $\bar{u}^{n,\delta}$ be defined by (\ref{mildcoe}). Then there exists a positive constant $C(G,p,n)<\infty$ such that for every $p\in[1,\infty)$, 
\begin{align}\label{e-13-cons}
\mathbb{E}\|\bar{u}^{n,\delta}\|_{L^p([0,T]\times\mathbb{T}^d)}^p\leq C(G,p,n)K_2(\delta,d)^{\frac{pn}{2}}. 
\end{align}
\end{theorem}
\begin{proof}
We argue by induction.

{\bf Step 1}. For $n=1$, and $\delta>0$, $\bar{u}^{1,\delta}$ is written as
\begin{equation*}
\bar{u}^{1,\delta}(t)=\int_0^t\nabla S(t-s)G(\bar{u}^{0,\delta}(s))dW^{\delta}(s). 
\end{equation*}
For every $p\in[1,\infty)$, thanks to Lemma \ref{gen-littlewoodpaley} and Lemma \ref{maximalLp-cons}, it follows from Hypothesis H2 that 
\begin{align*}
&\mathbb{E}\|\bar{u}^{1,\delta}\|_{L^p([0,T]\times\mathbb{T}^d)}^p\leq C(G,p)K_2(\delta,d)^{\frac{p}{2}}.
\end{align*}
This proves (\ref{e-13-cons}) for $n=1$, $p\in(2,\infty)$. For $p\in[1,2]$, H\"older's inequality implies (\ref{e-13-cons}) for $n=1$.

{\bf Step 2}. For $k\in\mathbb{N}_+$, by induction hypothesis, assume that (\ref{e-13-cons}) holds for $n=1,\dots,k-1$. We aim to prove (\ref{e-13-cons}) for $n=k$. The mild solution $\bar{u}^{k,\delta}$ satisfies (\ref{mildcoe-cons}). Using Lemma \ref{maximalLp-cons} again to see that for every $p\in(2,\infty)$,  
\begin{align*}
\mathbb{E}\|\bar{u}^{k,\delta}\|_{L^p([0,T]\times\mathbb{T}^d)}^p\leq&C(G,p)K_2(\delta,d)^{\frac{p}{2}}\int_{[0,T]\times\mathbb{T}^d}\mathbb{E}\Big(\sum_{l=1}^{k-1}\frac{1}{l!}|\mathcal{J}^{\delta}(k,l)(s,y)|\Big)^pdyds. 
\end{align*}
More precisely, applying H\"older's inequality with $\int_{\Omega}d\mathbb{P}$ replaced by $\int_0^T\int_{\mathbb{T}^d}\int_{\Omega}d\mathbb{P}dxdt$, we have 
\begin{align*}
	&\int_{[0,T]\times\mathbb{T}^d}\mathbb{E}\Big(\sum_{l=1}^{k-1}\frac{1}{l!}|\mathcal{J}^{\delta}(k,l)(s,y)|\Big)^pdyds\\
	\leq&C(p,n)\sum_{l=1}^{k-1}\frac{1}{|l!|^p}\sum_{(q_1,\dots,q_{k-l})\in\Lambda(k,l)}(\frac{l!}{q_1!\dots q_{k-l}!})^p\Big\|\prod_{1\leq j\leq k-l}(\bar{u}^{j,\delta})^{q_j}\Big\|_{L^p(\Omega\times[0,T]\times\mathbb{T}^d)}^p\\
\leq&C(p,n)\sum_{l=1}^{k-1}\frac{1}{|l!|^p}\sum_{(q_1,\dots,q_{k-l})\in\Lambda(k,l)}(\frac{l!}{q_1!\dots q_{k-l}!})^p\prod_{1\leq j\leq k-l}\Big\|\bar{u}^{j,\delta}\Big\|_{L^{pq_jp_j}(\Omega\times[0,T]\times\mathbb{T}^d)}^{pq_j}\\
\leq&C(G,p,n)\sum_{l=1}^{k-1}\frac{1}{|l!|^p}\sum_{(q_1,\dots,q_{k-l})\in\Lambda(k,l)}(\frac{l!}{q_1!\dots q_{k-l}!})^p\prod_{1\leq j\leq k-l}\Big(K_2(\delta,d)^{\frac{jq_jp_jp}{2}}\Big)^{\frac{1}{p_j}}\\
=&C(G,p,n)\sum_{l=1}^{k-1}\frac{1}{|l!|^p}\sum_{(q_1,\dots,q_{k-l})\in\Lambda(k,l)}(\frac{l!}{q_1!\dots q_{k-l}!})^pK_2(\delta,d)^{\sum_{1\leq j\leq k-l}\frac{jq_jp}{2}}.\end{align*}
Therefore we establish (\ref{e-13-cons}) for $n=k$ and $p\in[1,\infty)$. This completes the proof.  
\end{proof}

\section{Higher order fluctuations for smooth coefficients}\label{sec-6}
In this section, we prove Theorem \ref{main-1} and Theorem \ref{main-2}. In this section, we always assume that the assumptions of Theorem \ref{main-1} and Theorem \ref{main-2} are satisfied, and, in the conservative case,  
\begin{align}\label{varepsilonsmall}
	\varepsilon < \varepsilon_0 = \varepsilon_0(\delta,G_0),
\end{align}
where $\varepsilon_0$ is as in Lemma \ref{variationwellposed-2}. Let $u^{\varepsilon,\delta}$ be the mild solution of (\ref{SHE02}) and (\ref{SHE01}), respectively, and let $\bar{u}^{n,\delta}$, $n \in \mathbb{N}$, be the mild solution of (\ref{coe}) and (\ref{ccoe}), respectively. Set
\begin{align}\label{wnn}
w^{\varepsilon,\delta}_n=\varepsilon^{-\frac{n}{2}}\Big(u^{\varepsilon,\delta}-\sum_{i=0}^{n}\varepsilon^{\frac{i}{2}}\bar{u}^{i,\delta}\Big).
\end{align}

The aim is to prove that for every $n\in\mathbb{N}$, $w^{\varepsilon,\delta}_n$ converges to zero in a suitable space, provided a suitable relative scaling of $(\varepsilon,\delta(\varepsilon))$. 

\subsection{Expression of the remainder $w_n^{\varepsilon,\delta}$}

\begin{lemma}\label{sigman}
Let $\varepsilon, \delta>0$. For every $n\in\mathbb{N}$, let $w^{\varepsilon,\delta}_n$  be defined by (\ref{wnn}). Then 
\item{(i)} {\bf Non-conservative case:}  
$w^{\varepsilon,\delta}_n$ is a mild solution of 
 \begin{equation}\label{wn-1}
dw^{\varepsilon,\delta}_n=\Delta w^{\varepsilon,\delta}_ndt+\sigma^{\varepsilon,\delta}_n(t)dW^{\delta}(t),\ \ w^{\varepsilon,\delta}_n(0)=0.  
\end{equation}
\item{(ii)} {\bf Conservative case:}  
$w^{\varepsilon,\delta}_n$ is a mild solution of 
 \begin{equation}\label{wn-1-cons}
dw^{\varepsilon,\delta}_n=\Delta w^{\varepsilon,\delta}_ndt+\nabla\cdot(\sigma^{\varepsilon,\delta}_n(t)dW^{\delta}(t)),\ \ w^{\varepsilon,\delta}_n(0)=0.  
\end{equation}
Here, the diffusion coefficients $\sigma^{\varepsilon,\delta}_n, n\geq 0$ are given by
\begin{align}\label{eq-2}
\sigma^{\varepsilon,\delta}_n(\cdot)=&\varepsilon^{-\frac{n-1}{2}}\Big(G(u^{\varepsilon,\delta})-\Big[\sum_{m=1}^{n}\varepsilon^{\frac{m-1}{2}}\Big(\sum_{l=0}^{m-1}\frac{1}{l!}G^{(l)}(\bar{u}^{0,\delta})\mathcal{J}^{\delta}(m,l)\Big)\Big]\Big),\ n\geq1, 
\end{align}
where $\mathcal{J}^{\delta}(k,l)$ is given by (\ref{Jkl}). Here we make a convention where the summation $\sum_{m=1}^0$ is always zero, regardless of the objects being summed.
\end{lemma}
\begin{proof}
The result will be proved by induction. Since $\mathcal{J}^{\delta}(1,0)=1$, it is obvious that  
\begin{align*}
	\sigma^{\varepsilon,\delta}_1(\cdot)=G(u^{\varepsilon,\delta})-G(\bar{u}^{0,\delta}), 
\end{align*}
which satisfies (\ref{eq-2}) with $n=1$.

In the following, for any $n\geq2$, we assume that ({\ref{eq-2}}) holds for $n-1$, and aim to prove that ({\ref{eq-2}}) holds for $n$. By the definition of $w_n^{\varepsilon,\delta}$ in (\ref{wnn}), by a direct calculation, we find that
\begin{align*}
w^{\varepsilon,\delta}_n=&\varepsilon^{-\frac{n}{2}}\Big(u^{\varepsilon,\delta}-\sum_{i=0}^{n-1}\varepsilon^{\frac{i}{2}}\bar{u}^{i,\delta}-\varepsilon^{\frac{n}{2}}\bar{u}^{n,\delta}\Big)\\
=&\varepsilon^{-\frac{1}{2}}w^{\varepsilon,\delta}_{n-1}-\bar{u}^{n,\delta}. 
\end{align*}
As a consequence, the coefficients $\sigma_n^{\varepsilon,\delta}(\cdot)$ in  (\ref{wn-1}) can be derived recursively, 
\begin{equation*}
\sigma^{\varepsilon,\delta}_n(\cdot)=\frac{\sigma^{\varepsilon,\delta}_{n-1}(\cdot)}{\varepsilon^{\frac{1}{2}}}-\Big[\sum_{l=1}^{n-1}\frac{1}{l!}G^{(l)}(\bar{u}^{0,\delta})\mathcal{J}^{\delta}(n,l)\Big].
\end{equation*}
By the induction hypothesis that ({\ref{eq-2}}) holds for $n-1$, it follows that
\begin{align*}
\sigma^{\varepsilon,\delta}_n(\cdot)=&\frac{1}{\varepsilon^{\frac{1}{2}}}\Big\{\varepsilon^{-\frac{n-1-1}{2}}\Big(G(u^{\varepsilon,\delta})-\Big[\sum_{m=1}^{n-1}\varepsilon^{\frac{m-1}{2}}(\sum_{l=1}^{m-1}\frac{1}{l!}G^{(l)}(\bar{u}^{0,\delta})\mathcal{J}^{\delta}(m,l))\Big]\Big)\Big\}\\
&-\Big[\sum_{l=1}^{n-1}\frac{1}{l!}G^{(l)}(\bar{u}^{0,\delta})\mathcal{J}^{\delta}(n,l)\Big]\\
=&\varepsilon^{-\frac{n-1}{2}}\Big\{G(u^{\varepsilon,\delta})-\Big[\sum_{m=1}^{n-1}\varepsilon^{\frac{m-1}{2}}(\sum_{l=1}^{m-1}\frac{1}{l!}G^{(l)}(\bar{u}^{0,\delta})\mathcal{J}^{\delta}(m,l))\Big]\Big\}\\
&-\Big[\sum_{l=1}^{n-1}\frac{1}{l!}G^{(l)}(\bar{u}^{0,\delta})\mathcal{J}^{\delta}(n,l)\Big]\\
=&\varepsilon^{-\frac{n-1}{2}}\Big\{G(u^{\varepsilon,\delta})-\Big[\sum_{m=1}^{n}\varepsilon^{\frac{m-1}{2}}(\sum_{l=1}^{m-1}\frac{1}{l!}G^{(l)}(\bar{u}^{0,\delta})\mathcal{J}^{\delta}(m,l))\Big]\Big\}.
\end{align*}
Thus, ({\ref{eq-2}}) holds for $n$ and induction completes the proof.
\end{proof}

\subsection{Proofs of Theorem \ref{main-1} and \ref{main-2}}
In this section, we establish uniform bounds on the remainder terms $w^{\varepsilon,\delta}_n$ defined by (\ref{wnn}) and thereby prove Theorem \ref{main-1} and \ref{main-2}. In order to prove the convergence of the remainder terms, we require a priori estimates for the mild solutions of  (\ref{SHE02}). The following lemmas provide $L^p$-estimates for $u^{\varepsilon,\delta}$ in  (\ref{mild-1}) and (\ref{mild-1-cons}).
\begin{lemma}\label{Lp}
Assume that $G$ and $u_0$ satisfy Hypothesis H1 and H2. Let $\varepsilon,\delta>0$, $d\geq1$, and let $u^{\varepsilon,\delta}$ be the mild solution of (\ref{SHE02}). Then for every $p\in[1,+\infty)$, there exist constants $C_1=C_1(u_0,p)$ and $C_2=C_2(G,p)$, such that 
\begin{align}\label{eq-22}
\sup_{t\in[0,T],x\in\mathbb{T}^d}\mathbb{E}|u^{\varepsilon,\delta}(t,x)|^p\lesssim C_1+C_2\cdot(\varepsilon K_1(\delta,d))^{\frac{p}{2}}.
\end{align}
\end{lemma}
\begin{proof}
For any $p\in[2,\infty)$, by the mild form of $u^{\varepsilon,\delta}$, the $L^p$-isometry of the stochastic integral \cite[Corollary 3.11]{NW}, and the definition of (\ref{innerproduct}), it follows that for every $(t,x)\in[0,T]\times\mathbb{T}^d$, 
\begin{align*}
\mathbb{E}|u^{\varepsilon,\delta}(t,x)|^p\leq C(p)\mathbb{E}|(p(t)\ast u_0)(x)|^p+C(p)\varepsilon^{\frac{p}{2}}\mathbb{E}\Big[\Big|\int_0^{t}\|p(t-s,x-\cdot)G(u^{\varepsilon,\delta}(s,\cdot))\|_{\delta}^2ds\Big|^{\frac{p}{2}}\Big].
\end{align*}
Following the proof of Lemma \ref{maximalLp-noncons}, with the aid of Minkowski's inequality and Young's convolution inequality, and thanks to the boundedness of $G$, we find that 
\begin{align*}
&\mathbb{E}|u^{\varepsilon,\delta}(t,x)|^p\leq\|u_0\|_{L^{\infty}(\mathbb{T}^d)}^p\\
&+C(p)\varepsilon^{\frac{p}{2}}\Big|\int_0^{t}\int_{\mathbb{T}^{2d}}|p(t-s,x-y_1)p(t-s,x-y_2)|(\mathbb{E}|G(u^{\varepsilon,\delta}(s,y_1))G(u^{\varepsilon,\delta}(s,y_2))|^{\frac{p}{2}})^{\frac{2}{p}}\\
&\cdot R_{\delta}(y_1-y_2)dy_1dy_2ds\Big|^{\frac{p}{2}}\\
\leq&\|u_0\|_{L^{\infty}(\mathbb{T}^d)}^p+C(G,p)(\varepsilon K_{\delta}(T))^{\frac{p}{2}},
\end{align*}
where $K_{\delta}(T)$ is defined by (\ref{K2i}). Let $K_1(\delta,d)$ as in (\ref{Kdelta}), combining with (\ref{K2i}) and (\ref{divkernel}) we obtain  
\begin{equation*}
\sup_{t\in[0,T],x\in\mathbb{T}^d}\mathbb{E}|u^{\varepsilon,\delta}(t,x)|^p\lesssim\|u_0\|_{L^{\infty}(\mathbb{T}^d)}^p+C(G,p)(\varepsilon K_1(\delta,d))^{\frac{p}{2}},
\end{equation*}
which concludes the proof for $p\in[2,\infty)$. When $p\in[1,2)$, similar to the proof of Lemma \ref{maximalLp-noncons}, we complete the proof by using H\"older's inequality.  
\end{proof}

Regarding the case of conservative noise, we have the following analogue. 
\begin{lemma}
	Assume that $G$ and $u_0$ satisfy Hypothesis H1 and H2. Let  $d\geq1$, $\delta\in(0,1)$, and let $\varepsilon_0$ be as in Lemma \ref{variationwellposed-2}. For every $\varepsilon\in(0,\varepsilon_0)$, let $u^{\varepsilon,\delta}$ be the mild solution of (\ref{SHE01}). Then for every $p\in[1,+\infty)$, there exist constants $C_1=C_1(u_0,p)$ and $C_2=C_2(G,p)$, such that 
\begin{align}\label{eq-22-cons}
\mathbb{E}\|u^{\varepsilon,\delta}\|_{L^p([0,T]\times\mathbb{T}^d)}^p\lesssim C_1+C_2\cdot(\varepsilon K_2(\delta,d))^{\frac{p}{2}}.
\end{align}
\end{lemma}
\begin{proof}
	Let $p\in(2,\infty)$. For simplicity, we denote by $G(s,y):=G(u^{\varepsilon,\delta}(s,y))$, $G_{\delta}(s,y,\cdot):=G(s,y)\eta_{\delta}(y-\cdot)$. Thanks to Lemma \ref{lem-moser-1}, the $L^p(\Omega;L^p([0,T];L^p(\mathbb{T}^d)))$-norm of $u^{\varepsilon,\delta}$ is finite. With the help of Lemma \ref{gen-littlewoodpaley} and Lemma \ref{maximalLp-cons}, it follows from the hypothesis of $G$ that 
\begin{align}\label{Lpestimateforu-0}
\mathbb{E}\|u^{\varepsilon,\delta}\|_{L^p([0,T]\times\mathbb{T}^d)}^p\leq C(G,p)\varepsilon^{\frac{p}{2}}K_2(\delta,d)^{\frac{p}{2}}.
\end{align}
Let $c>0$, by the assumption $\varepsilon^{\frac{p}{2}}K_2(\delta,d)^{\frac{p}{2}}\leq c$, we have 
\begin{align}\label{Lpestimateforu}
	\mathbb{E}\|u^{\varepsilon,\delta}\|_{L^p([0,T]\times\mathbb{T}^d)}^p\leq C(u_0,p)+C(G,p)c^{\frac{p}{2}}. 
\end{align}
Applying H\"older's inequality, (\ref{eq-22-cons}) holds for all $p\in[1,\infty)$ as well. 

\end{proof}

The following lemma in number theory plays a key role in proving fluctuation expansions. For $k,l,m\in\mathbb{N}_+$, $k>l$, and $m\geq l+1$. Recall that $\Lambda(k,l)$ is defined by (\ref{inteq}). Let $\Lambda(k,l,m)$ be the set of all nonnegative integer solutions $(q_1, \dots, q_{k-l})$ satisfying 
\begin{eqnarray}
\left\{
  \begin{array}{ll}\label{inteq-1}
   &q_1+q_2+\dots+q_{k-l}=l,\\
   &q_1+2q_2+\dots+(k-l)q_{k-l}=m-1.
  \end{array}
\right.
\end{eqnarray}

Clearly, $\Lambda(k,l,k)=\Lambda(k,l)$. Moreover, when $m>(k-l)l+1$, $\Lambda(k,l,m)=\emptyset$. 

\begin{lemma}\label{number}
Let $k,l,m\in\mathbb{N}_+$, $k>l$, and $l+1\leq m\leq k$. 
\begin{description}
\item[(i)]For every $(q_1,q_2,...,q_{k-l})\in(\mathbb{N})^{k-l}$, let $P_{m-l}(q_1,q_2,...,q_{k-l})=(q_1,q_2,...,q_{m-l})$ be the projection onto first $m-l$ entries. Then for every $(q_1,...,q_{k-l})\in\Lambda(k,l,m)$, we have $P_{m-l}(q_1,...,q_{k-l})\in\Lambda(m,l)$, and when $l+1\leq m\leq k-1$, we have $q_j=0$, for $m-l+1\leq j\leq k-l$. 

\item[(ii)]Let $(q_1,...,q_{m-l})\in \Lambda(m,l)$, we define $(\bar{q}_1,...,\bar{q}_{m-l},...,\bar{q}_{k-l})$ as
 \begin{eqnarray}
\left\{
  \begin{array}{ll}\label{extension}
   \bar{q}_i=&q_i,\ \ 1\leq i\leq m-l\\
   \bar{q}_i=&0,\ \ m-l+1\leq i\leq k-l.
  \end{array}
\right.
\end{eqnarray}
Then we have $(\bar{q}_1,...,\bar{q}_{m-l},...,\bar{q}_{k-l})\in\Lambda(k,l,m)$. 
\end{description}
\end{lemma}

\begin{proof} 
Let $k,l\in\mathbb{N}_+$ be fixed. In the case of $m=k$, $\Lambda(k,l)=\Lambda(k,l,m)$, the result is obvious. In the case of $l+1\leq m\leq k-1$, it suffices to prove that for every element $(q_1,\dots,q_{k-l})\in\Lambda(k,l,m)$, we have 
\begin{equation}\label{qj0}
q_j=0,\ \ \ \text{for}\ \ m-l+1\leq j\leq k-l.
\end{equation}

Let $(q_1,\dots,q_{k-l})\in\Lambda(k,l,m)$, assume that  $q_j\geq1$,
 for some $j\in[m-l+1,k-l]$, then the unique choice to reach the minimum of $ q_1+2q_2+\dots+(k-l)q_{k-l}$ in (\ref{inteq-1}) is that
\begin{equation*}
q_1=l-q_j,\ \ \  q_i=0,\ \ \ i\neq j,1.
\end{equation*}
As a result, we get
\begin{align*}
q_1+jq_j=l+(j-1)q_j\geq l+(m-l+1-1)=m>m-1,
\end{align*}
which leads to a contradiction to the fact that $(q_1,\dots,q_{k-l})\in\Lambda(k,l,m)$, this completes the proof of (i). 

For every $(q_1,...,q_{m-l})\in\Lambda(m,l)$, let $(\bar{q}_1,...,\bar{q}_{m-l},...,\bar{q}_{k-l})$ be defined by (\ref{extension}). We have that 
\begin{align*}
\sum_{i=1}^{k-l}\bar{q}_i=\sum_{i=1}^{m-l}\bar{q}_i+0=l,	
\end{align*}
and 
\begin{align*}
\sum_{i=1}^{k-l}i\bar{q}	_i=\sum_{i=1}^{m-l}i\bar{q}	_i+0=m-1.
\end{align*}
These imply that $(\bar{q}_1,...,\bar{q}_{m-l},...,\bar{q}_{k-l})\in\Lambda(k,l,m)$, this completes the proof of (ii).
 
\end{proof}

When $i=1$ and $d=1$, we take $\delta=0$ and denote by $w^{\varepsilon}_n=w^{\varepsilon,0}_n$. As before, $w^{\varepsilon,\delta}_n$ is the remainder of the $n$-th order small noise expansion. The next result provides estimates for $w^{\varepsilon,\delta}_n$.

\begin{theorem}\label{n-CLT}
Assume that $G$ and $u_0$ satisfy Hypothesis H1 and H2. Let $n\in\mathbb{N}_+$, $\varepsilon, \delta>0$. Let $K_i(\delta,d)$ be defined by (\ref{Kdelta}) for $i=1,2$. 
\item{(i)} {\bf Non-conservative noise.} Let $w^{\varepsilon,\delta}_n$ be defined by (\ref{wnn}). Assume that the scaling regime $(\varepsilon,\delta(\varepsilon))$ satisfies
\begin{equation}\label{sca-k}
\lim_{\varepsilon\rightarrow0}\varepsilon K_1(\delta(\varepsilon),d)^{n+1}=0.
\end{equation}
Then there exists $\varepsilon_0>0$, such that for all $\varepsilon\in(0,\varepsilon_0)$, for every $p\in[1,+\infty)$, there exists a constant $C=C(G,p,n)$, such that 
\begin{equation}\label{lln-1}
\sup_{t\in[0,T],x\in\mathbb{T}^d}\mathbb{E}|w^{\varepsilon,\delta(\varepsilon)}_n(t,x)|^p\leq C(\varepsilon K_1(\delta(\varepsilon),d)^{n+1})^{\frac{p}{2}}.
\end{equation}
\item{(ii)} {\bf Conservative noise.} Assume that 
\begin{equation}\label{sca-k-cons}
\lim_{\varepsilon\rightarrow0}\varepsilon K_2(\delta(\varepsilon),d)^{n+1}=0
\end{equation}
holds for the conservative case. Then there exists $\varepsilon_0>0$, such that for all $\varepsilon\in(0,\varepsilon_0)$, for every $p\in[1,+\infty)$, there exists a constant $C=C(G,p,n)$, such that 
\begin{equation}\label{lln-1-cons}
\mathbb{E}\|w^{\varepsilon,\delta(\varepsilon)}_n\|_{L^p([0,T]\times\mathbb{T}^d)}^p\leq C(\varepsilon K_2(\delta(\varepsilon),d)^{n+1})^{\frac{p}{2}}. 
\end{equation} 
\end{theorem}
\begin{proof} Due to the structural similarity between two cases, the proof will proceed in a unified way. We will prove (\ref{lln-1}) and (\ref{lln-1-cons}) by induction. Let us first consider the cases $n=0$.

{\bf{Step 1. Induction for $n=0$.}}\quad Regarding the non-conservative case, the mild form of $w_0^{\varepsilon,\delta}$ can be written as
\begin{equation*}
w_0^{\varepsilon,\delta}(t,x)=\varepsilon^{\frac{1}{2}}\int_0^t\langle p(t-s,x-\cdot),G(u^{\varepsilon,\delta}(s,\cdot))dW^{\delta}(s)\rangle.
\end{equation*}
By analogous arguments as in the proof of Lemma \ref{maximalLp-noncons}, Lemma \ref{Lp}, we find that
\begin{align*}
&\sup_{t\in[0,T],x\in\mathbb{T}^d}\mathbb{E}|w^{\varepsilon,\delta}_0(t,x)|^p\\
\leq&C(G,p)\varepsilon^{\frac{p}{2}}\Big|\int_0^t\int_{\mathbb{T}^{2d}}p(t-s,x-y_1)p(t-s,x-y_2)R_{\delta}(y_1-y_2)dy_1dy_2ds\Big|^{\frac{p}{2}}\\
\leq&C(G,p)\Big(\varepsilon K_1(\delta,d)\Big)^{\frac{p}{2}},
\end{align*}
under the scaling regime (\ref{sca-k}), under the assumption $\varepsilon K_1(\delta(\varepsilon),d)\rightarrow0$, as $\varepsilon\rightarrow0$,  
\begin{align*}
\sup_{t\in[0,T],x\in\mathbb{T}^d}\mathbb{E}|w^{\varepsilon,\delta(\varepsilon)}_0(t,x)|^p\leq C(G,p)\Big(\varepsilon K_1(\delta(\varepsilon),d)\Big)^{\frac{p}{2}},
\end{align*}
which implies (\ref{lln-1}) holds for $n=0$. 

Regarding the conservative case, the mild form of $w_0^{\varepsilon,\delta}$ can be written as
\begin{equation*}
w_0^{\varepsilon,\delta}(t)=\int_0^t\nabla S(t-s)G(u^{\varepsilon,\delta}(s))dW^{\delta}(s).
\end{equation*}
Using the same argument of the $L^p$-isometry and the generalization of the Littlewood-Paley inequality, thanks to the $L^p$-estimate (\ref{Lpestimateforu}), with the help of the assumptions on $G$ and $u_0$, we have that for every $p\in(2,\infty)$, 
\begin{align*}
\mathbb{E}\|w^{\varepsilon,\delta(\varepsilon)}_0\|_{L^p([0,T]\times\mathbb{T}^d)}^p\leq&C(p)\varepsilon^{\frac{p}{2}}K_2(\delta(\varepsilon),d)^{\frac{p}{2}}\mathbb{E}\int^T_0\int_{\mathbb{T}^d}|G(u^{\varepsilon,\delta(\varepsilon)}(s,y))|^pdyds\\
\leq&C(G,p)\varepsilon^{\frac{p}{2}}K_2(\delta(\varepsilon),d)^{\frac{p}{2}}.   
\end{align*}
Then using H\"older's inequality to see that (\ref{lln-1}) and (\ref{lln-1-cons}) hold for $n=0$, $p\in[1,\infty)$.

{\bf{Step 2. Induction for $n=k$. }} For every $k\geq1$, and for every $p\in[1,\infty)$, regarding the non-conservative case, by induction hypothesis we have that under the scaling regime  (\ref{sca-k}), 
\begin{equation}\label{aim-noncons}
\sup_{t\in[0,T],x\in\mathbb{T}^d}\mathbb{E}|w^{\varepsilon,\delta(\varepsilon)}_n(t,x)|^p\leq C(G,p,n)(\varepsilon K_1(\delta(\varepsilon),d)^{n+1})^{\frac{p}{2}}
\end{equation}
holds for $n=0,\dots,k-1$.  For the conservative case, by induction hypothesis we have that under the scaling regime (\ref{sca-k-cons}), 
\begin{equation}\label{aim-cons}
\mathbb{E}\|w^{\varepsilon,\delta(\varepsilon)}_n\|_{L^p([0,T]\times\mathbb{T}^d)}^p\leq C(G,p,n)(\varepsilon K_2(\delta(\varepsilon),d)^{n+1})^{\frac{p}{2}}
\end{equation}
holds for $n=0,\dots,k-1$. We aim to deduce both (\ref{aim-noncons}) and (\ref{aim-cons}) hold for $n=k$ as well. By Lemma \ref{sigman}, we have that $w_k^{\varepsilon,\delta}$ satisfies the following mild form
\begin{eqnarray}
w^{\varepsilon,\delta(\varepsilon)}_k(t,x)=\int^t_0\langle p(t-s,x-\cdot), \sigma_k^{\varepsilon,\delta(\varepsilon)}(s,\cdot)dW^{\delta(\varepsilon)}(s)\rangle,
\end{eqnarray}
and 
\begin{eqnarray}
w^{\varepsilon,\delta(\varepsilon)}_k(t)=\int^t_0\nabla S(t-s) \sigma_k^{\varepsilon,\delta(\varepsilon)}(s)dW^{\delta(\varepsilon)}(s),
\end{eqnarray}
respectively, where
\begin{align*}
&\sigma_k^{\varepsilon,\delta(\varepsilon)}(\cdot)\\
=&\varepsilon^{-\frac{k-1}{2}}\Big(G(u^{\varepsilon,\delta(\varepsilon)})-G(\bar{u}^{0,\delta(\varepsilon)})-\Big[\sum_{m=2}^{k-1}\varepsilon^{\frac{m-1}{2}}\Big(\sum_{l=1}^{m-1}\frac{1}{l!}G^{(l)}(\bar{u}^{0,\delta(\varepsilon)})\mathcal{J}^{\delta(\varepsilon)}(m,l)\Big)\Big]\Big)\\
&-\Big[\sum_{l=1}^{k-1}\frac{1}{l!}G^{(l)}(\bar{u}^{0,\delta(\varepsilon)})\mathcal{J}^{\delta(\varepsilon)}(k,l)\Big]\\
=&\varepsilon^{-\frac{k-1}{2}}\Big(G(u^{\varepsilon,\delta(\varepsilon)})-\sum_{l=0}^{k-1}\frac{1}{l!}G^{(l)}(\bar{u}^{0,\delta(\varepsilon)})(u^{\varepsilon,\delta(\varepsilon)}-\bar{u}^{0,\delta(\varepsilon)})^{l}\Big)\\
&+\varepsilon^{-\frac{k-1}{2}}\Big(\sum_{l=1}^{k-1}\frac{1}{l!}G^{(l)}(\bar{u}^{0,\delta(\varepsilon)})(u^{\varepsilon,\delta(\varepsilon)}-\bar{u}^{0,\delta(\varepsilon)})^{l}-\Big[\sum_{m=2}^{k-1}\varepsilon^{\frac{m-1}{2}}\Big(\sum_{l=1}^{m-1}\frac{1}{l!}G^{(l)}(\bar{u}^{0,\delta(\varepsilon)})\mathcal{J}^{\delta(\varepsilon)}(m,l)\Big)\Big]\Big)\\
&-\Big[\sum_{l=1}^{k-1}\frac{1}{l!}G^{(l)}(\bar{u}^{0,\delta(\varepsilon)})\mathcal{J}^{\delta(\varepsilon)}(k,l)\Big]\\
=&I_1^{\varepsilon}+I_2^{\varepsilon}+I_3^{\varepsilon}+I_4^{\varepsilon},
\end{align*}
and $I_i^{\varepsilon}$, $i=1,2,3,4$, are defined by
\begin{align}\label{I1-4}
I_1^{\varepsilon}:=&\varepsilon^{-\frac{k-1}{2}}\Big(G(u^{\varepsilon,\delta(\varepsilon)})-\sum_{l=0}^{k-1}\frac{1}{l!}G^{(l)}(\bar{u}^{0,\delta(\varepsilon)})(u^{\varepsilon,\delta(\varepsilon)}-\bar{u}^{0,\delta(\varepsilon)})^{l}\Big),\notag\\
I_2^{\varepsilon}:=&G^{(1)}(\bar{u}^{0,\delta(\varepsilon)})\Big[\varepsilon^{-\frac{k-1}{2}}\Big(u^{\varepsilon,\delta(\varepsilon)}-\bar{u}^{0,\delta(\varepsilon)}-\sum^{k}_{m=2}\varepsilon^{\frac{m-1}{2}}\mathcal{J}^{\delta(\varepsilon)}(m,1)\Big)\Big],\notag\\
I_3^{\varepsilon}:=&\sum^{k-2}_{l=2}\frac{1}{l!}G^{(l)}(\bar{u}^{0,\delta(\varepsilon)})\Big[\varepsilon^{-\frac{k-1}{2}}\Big((u^{\varepsilon,\delta(\varepsilon)}-\bar{u}^{0,\delta(\varepsilon)})^l-\sum^{k}_{m=l+1}\varepsilon^{\frac{m-1}{2}}\mathcal{J}^{\delta(\varepsilon)}(m,l)\Big)\Big],\notag\\
I_4^{\varepsilon}:=&\frac{1}{(k-1)!}G^{(k-1)}(\bar{u}^{0,\delta(\varepsilon)})\Big[\varepsilon^{-\frac{k-1}{2}}(u^{\varepsilon,\delta(\varepsilon)}-\bar{u}^{0,\delta(\varepsilon)})^{k-1}-\mathcal{J}^{\delta(\varepsilon)}(k,k-1)\Big].
\end{align}

For $k=1$, we set $I_2^{\varepsilon},I_3^{\varepsilon},I_4^{\varepsilon}=0$. For $k=2$, we set $I_2^{\varepsilon},I_3^{\varepsilon}=0$. For $k=3$, we set $I_3^{\varepsilon}=0$.  Otherwise, $I_1^{\varepsilon}, I_2^{\varepsilon}, I_3^{\varepsilon}, I_4^{\varepsilon}$ are well-defined. 

We first proceed with estimating $I_1^{\varepsilon}$. By Taylor expansion, we get $|I_1^{\varepsilon}|\leq C(G)|w^{\varepsilon,\delta(\varepsilon)}_0|^{k}/\varepsilon^{\frac{k-1}{2}}$. Since $\mathcal{J}^{\delta(\varepsilon)}(k,1)=\bar{u}^{k-1,\delta(\varepsilon)}$, for any $k\geq1$, and due to the hypothesis on $G$, we have $|I_2^{\delta(\varepsilon)}|\leq C(G)|w^{\varepsilon,\delta(\varepsilon)}_{k-1}|$. For $I_4^{\delta(\varepsilon)}$, since $\mathcal{J}^{\delta(\varepsilon)}(k,k-1)=(\bar{u}^{1,\delta(\varepsilon)})^{k-1}$, by the hypothesis on $G$ and the Binomial theorem, we have
\begin{align*}
I_4^{\varepsilon}=&\frac{1}{(k-1)!}G^{(k-1)}(\bar{u}^{0,\delta(\varepsilon)})\Big[(w_1^{\varepsilon,\delta(\varepsilon)}+\bar{u}^{1,\delta(\varepsilon)})^{k-1}-(\bar{u}^{1,\delta(\varepsilon)})^{k-1}\Big]\\
=&\frac{1}{(k-1)!}G^{(k-1)}(\bar{u}^{0,\delta(\varepsilon)})\Big[\sum_{l=0}^{k-2}C^l_{k-1}(w_1^{\varepsilon,\delta(\varepsilon)})^{k-1-l}(\bar{u}^{1,\delta(\varepsilon)})^l+(\bar{u}^{1,\delta(\varepsilon)})^{k-1}-(\bar{u}^{1,\delta(\varepsilon)})^{k-1}\Big]\\
=&\frac{1}{(k-1)!}G^{(k-1)}(\bar{u}^{0,\delta(\varepsilon)})\Big[\sum_{l=0}^{k-2}C^l_{k-1}(w_1^{\varepsilon,\delta(\varepsilon)})^{k-1-l}(\bar{u}^{1,\delta(\varepsilon)})^l\Big],
\end{align*}
where $C^l_{k-1}:=\frac{(k-1)!}{l!(k-1-l)!}$. By (\ref{wnn}), we find that 
\begin{align*}
&(u^{\varepsilon,\delta(\varepsilon)}-\bar{u}^{0,\delta(\varepsilon)})^l\\
=&\varepsilon^{\frac{(k-l)l}{2}}\Big(\frac{u^{\varepsilon,\delta(\varepsilon)}-\bar{u}^{0,\delta(\varepsilon)}}{\varepsilon^{\frac{k-l}{2}}}\Big)^l=\varepsilon^{\frac{(k-l)l}{2}}\Big(w^{\varepsilon,\delta(\varepsilon)}_{k-l}+\sum_{r=1}^{k-l}\varepsilon^{\frac{r-k+l}{2}}\bar{u}^{r,\delta(\varepsilon)}\Big)^l\\
=&\varepsilon^{\frac{(k-l)l}{2}}\Big[(w_{k-l}^{\varepsilon,\delta(\varepsilon)})^l+\sum_{m=1}^{l-1}C_m^l\Big[\sum_{r=1}^{k-l}\varepsilon^{\frac{r-k+l}{2}}\bar{u}^{r,\delta(\varepsilon)}\Big]^m\big[w_{k-l}^{\varepsilon,\delta(\varepsilon)}\Big]^{l-m}+\Big(\sum_{r=1}^{k-l}\varepsilon^{\frac{r-k+l}{2}}\bar{u}^{r,\delta(\varepsilon)}\Big)^l\Big]\\
=&\varepsilon^{\frac{(k-l)l}{2}}(w_{k-l}^{\varepsilon,\delta(\varepsilon)})^l
+\sum_{m=1}^{l-1}C_m^l\Big[\sum_{r=1}^{k-l}\varepsilon^{\frac{r}{2}}\bar{u}^{r,\delta(\varepsilon)}\Big]^m\Big[\varepsilon^{\frac{k-l}{2}}w_{k-l}^{\varepsilon,\delta(\varepsilon)}\Big]^{l-m}
+\Big(\sum_{r=1}^{k-l}\varepsilon^{\frac{r}{2}}\bar{u}^{r,\delta(\varepsilon)}\Big)^l.
\end{align*}
Furthermore, using Lemma \ref{number}, we deduce that
\begin{align*}
&\Big(\sum_{r=1}^{k-l}\varepsilon^{\frac{r}{2}}\bar{u}^{r,\delta(\varepsilon)}\Big)^l\\
=&\sum_{q_1+...+q_{k-l}=l}\Big(\frac{l!}{q_1!...q_{k-l}!}\Big)\prod_{1\leq r\leq k-l}(\varepsilon^{\frac{r}{2}}\bar{u}^{r,\delta(\varepsilon)})^{q_r}\\
=&\sum_{m=l+1}^{(k-l)l+1}\sum_{(q_1,...,q_{k-l})\in\Lambda(k,l,m)}\Big(\frac{l!}{q_1!...q_{k-l}!}\Big)\prod_{1\leq r\leq k-l}(\varepsilon^{\frac{r}{2}}\bar{u}^{r,\delta(\varepsilon)})^{q_r}\\
=&\sum_{m=l+1}^{k}\sum_{(q_1,...,q_{m-l})\in\Lambda(m,l)}\Big(\frac{l!}{q_1!...q_{m-l}!}\Big)\prod_{1\leq r\leq m-l}(\varepsilon^{\frac{r}{2}}\bar{u}^{r,\delta(\varepsilon)})^{q_r}\\
&+\sum_{m=k+1}^{(k-l)l+1}\sum_{(q_1,...,q_{k-l})\in\Lambda(k,l,m)}\Big(\frac{l!}{q_1!...q_{k-l}!}\Big)\prod_{1\leq r\leq k-l}(\varepsilon^{\frac{r}{2}}\bar{u}^{r,\delta(\varepsilon)})^{q_r}\\
=&\sum_{m=l+1}^{k}\varepsilon^{\frac{m-1}{2}}\mathcal{J}^{\delta(\varepsilon)}(m,l)+\sum_{m=k+1}^{(k-l)l+1}\varepsilon^{\frac{m-1}{2}}\sum_{(q_1,...,q_{k-l})\in\Lambda(k,l,m)}\Big(\frac{l!}{q_1!...q_{k-l}!}\Big)\prod_{1\leq r\leq k-l}(\bar{u}^{r,\delta(\varepsilon)})^{q_r},
\end{align*}
where we have used the property of $(q_1,\dots,q_{m-l})\in\Lambda(m,l)$, in particular, $\sum_{1\leq r\leq m-l}\frac{rq_r}{2}=\frac{m-1}{2}$. Based on the above two identities, $I_3^{\varepsilon}$ can be rewritten as
\begin{align*}
I_3^{\varepsilon}=&\sum^{k-2}_{l=2}\frac{1}{l!}G^{(l)}(\bar{u}^{0,\delta(\varepsilon)})\Big\{ \varepsilon^{\frac{(k-l)l-(k-1)}{2}}(w_{k-l}^{\varepsilon,\delta(\varepsilon)})^l\\
&+\varepsilon^{-\frac{k-1}{2}}\sum_{m=1}^{l-1}C_m^l\Big[\sum_{r=1}^{k-l}\varepsilon^{\frac{r}{2}}\bar{u}^{r,\delta(\varepsilon)}\Big]^m\Big[\varepsilon^{\frac{k-l}{2}}w_{k-l}^{\varepsilon,\delta(\varepsilon)}\Big]^{l-m}\\
&+\sum_{m=k+1}^{(k-l)l+1}\varepsilon^{\frac{m-k}{2}}\sum_{(q_1,...,q_{k-l})\in\Lambda(k,l,m)}\Big(\frac{l!}{q_1!...q_{k-l}!}\Big)\prod_{1\leq r\leq k-l}(\bar{u}^{r,\delta(\varepsilon)})^{q_r}\Big\}.
\end{align*}
For every $p\in[2,\infty)$, by the estimates of $I_i^{\varepsilon}$, $i=1,2,3,4$, the assumptions on $G$, the $L^p$-estimates for $u^{\varepsilon,\delta(\varepsilon)}$ in Lemma \ref{Lp} and the divergence speed of $\bar{u}^{k,\delta(\varepsilon)}$ in Theorem \ref{div-speed}, we have that 
\begin{align*}
	\mathbb{E}\Big[\Big|\int_0^t\|p(t-s,x-\cdot)(I_1^{\varepsilon}(s,\cdot)+I_2^{\varepsilon}(s,\cdot)+I_3^{\varepsilon}(s,\cdot)+I_4^{\varepsilon}(s,\cdot))\|_{\delta(\varepsilon)}^2ds\Big|^{\frac{p}{2}}\Big]<\infty, 
\end{align*}
and 
\begin{align*}
	\mathbb{E}\Big[\Big|\int_0^t\|\nabla S(t-s)(I_1^{\varepsilon}(s,\cdot)+I_2^{\varepsilon}(s,\cdot)+I_3^{\varepsilon}(s,\cdot)+I_4^{\varepsilon}(s,\cdot))\|_{\delta(\varepsilon)}^2ds\Big|^{\frac{p}{2}}\Big]<\infty, 
\end{align*}
with respect to the non-conservative case and the conservative case, respectively. In the following, we will discuss the non-conservative case and the conservative case separately. 

{\bf Non-conservative case.} Using Lemma \ref{K2} and Lemma \ref{maximalLp-noncons} to see that  
\begin{align*}
&\sup_{t\in[0,T],x\in\mathbb{T}^d}\mathbb{E}|w_{k}^{\varepsilon,\delta(\varepsilon)}(t,x)|^p\\
\leq&K_1(\delta(\varepsilon),d)^{\frac{p}{2}}\Big(\sup_{t\in[0,T],y\in\mathbb{T}^d}\mathbb{E}|I_1^{\delta(\varepsilon)}(t,y)+I_2^{\delta(\varepsilon)}(t,y)+I_3^{\delta(\varepsilon)}(t,y)+I_4^{\delta(\varepsilon)}(t,y)|^p\Big).
\end{align*}
Combining all the previous estimates, taking the supremum over $t$ and $x$, and using (\ref{divkernel}), we reach
\begin{align}\label{1to6}
&\sup_{t\in[0,T],x\in\mathbb{T}^d}\mathbb{E}|w^{\varepsilon,\delta(\varepsilon)}_{k}(t,x)|^p
\leq C(G)K_1(\delta(\varepsilon),d)^{\frac{p}{2}}\Big\{\mathcal{K}_1^{\varepsilon}+\mathcal{K}_2^{\varepsilon}+\mathcal{K}_3^{\varepsilon}+\mathcal{K}_4^{\varepsilon}+\mathcal{K}_5^{\varepsilon}+\mathcal{K}_6^{\varepsilon}\Big\},
\end{align}
where
\begin{align*}
\mathcal{K}_1^{\varepsilon}:=&\varepsilon^{-\frac{(k-1)p}{2}}\sup_{t\in[0,T],y\in\mathbb{T}^d}\mathbb{E}|w^{\varepsilon,\delta(\varepsilon)}_0(t,y)|^{kp},\\
\mathcal{K}_2^{\varepsilon}:=&\sup_{t\in[0,T],y\in\mathbb{T}^d}\mathbb{E}\Big|\sum_{l=0}^{k-2}C^l_{k-1}(w_1^{\varepsilon,\delta(\varepsilon)}(t,y))^{k-1-l}(\bar{u}^{1,\delta(\varepsilon)}(t,y))^l\Big|^{p},\\
\mathcal{K}_3^{\varepsilon}:=&\sup_{t\in[0,T],y\in\mathbb{T}^d}\mathbb{E}|w^{\varepsilon,\delta(\varepsilon)}_{k-1}(t,y)|^{p},\\
\mathcal{K}_4^{\varepsilon}:=&\sum^{k-2}_{l=2}\varepsilon^{\frac{(k-l)l-(k-1)}{2}p}\sup_{t\in[0,T],y\in\mathbb{T}^d}\mathbb{E}|w^{\varepsilon,\delta(\varepsilon)}_{k-l}(t,y)|^{lp},\\
\mathcal{K}_5^{\varepsilon}:=&\sum^{k-2}_{l=2}\varepsilon^{-\frac{p(k-1)}{2}}\Big(\sup_{t\in[0,T],y\in\mathbb{T}^d}\mathbb{E}\Big|\sum_{m=1}^{l-1}C_m^l\Big[\sum_{r=1}^{k-l}\varepsilon^{\frac{r}{2}}\bar{u}^{r,\delta(\varepsilon)}(t,y)\Big]^m\Big[\varepsilon^{\frac{k-l}{2}}w_{k-l}^{\varepsilon,\delta(\varepsilon)}(t,y)\Big]^{l-m}\Big|^{p}\Big),\\
\mathcal{K}_6^{\varepsilon}:=&\sup_{t\in[0,T],y\in\mathbb{T}^d}\mathbb{E}\Big|\sum^{k-2}_{l=2}\sum_{m=k+1}^{(k-l)l+1}\varepsilon^{\frac{m-k}{2}}\sum_{(q_1,...,q_{k-l})\in\Lambda(k,l,m)}\Big(\frac{l!}{q_1!...q_{k-l}!}\Big)\prod_{1\leq r\leq k-l}(\bar{u}^{r,\delta(\varepsilon)}(t,y))^{q_r}\Big|^p.
\end{align*}
In the following, $\mathcal{K}_1^{\varepsilon},\dots, \mathcal{K}_6^{\varepsilon}$ will be estimated one by one.

By the induction hypothesis (\ref{lln-1}) for $n=0,\dots,k-1$, there exists an $\varepsilon_0>0$ such that for every $\varepsilon\in(0,\varepsilon_0)$, we have
\begin{align*}
\mathcal{K}_1^{\varepsilon}&\leq C(G,p,k)(\varepsilon K_1(\delta(\varepsilon),d)^k)^{\frac{p}{2}},\quad 
\mathcal{K}_3^{\varepsilon}\leq C(G,p,k)(\varepsilon K_1(\delta(\varepsilon),d)^k)^{\frac{p}{2}}.
\end{align*}
For $\mathcal{K}_2^{\varepsilon}$, with the help of H\"older's inequality and (\ref{p-inequality}), it follows that
\begin{align*}
\mathcal{K}_2^{\varepsilon}\leq&C(p,k)\sup_{t\in[0,T],y\in\mathbb{T}^d}\sum_{l=0}^{k-2}\mathbb{E}|(w_1^{\varepsilon,\delta(\varepsilon)}(t,y))^{k-1-l}(\bar{u}^{1,\delta(\varepsilon)}(t,y))^l|^p\\
\leq&C(p,k)\sup_{t\in[0,T],y\in\mathbb{T}^d}\sum_{l=0}^{k-2}\Big(\mathbb{E}|w_1^{\varepsilon,\delta(\varepsilon)}(t,y)|^{2p(k-1-l)}\Big)^{\frac{1}{2}}\Big(\mathbb{E}|\bar{u}^{1,\delta(\varepsilon)}(t,y)|^{2pl}\Big)^{\frac{1}{2}}.
\end{align*}
By the induction hypothesis, Theorem \ref{div-speed}, and  (\ref{sca-k}), there exists an $\varepsilon_0>0$ such that for every $\varepsilon\in(0,\varepsilon_0)$, we have that
\begin{align*}
\mathcal{K}_2^{\varepsilon}\leq&C(G,p,k)\sum_{l=0}^{k-2}(\varepsilon K_1(\delta(\varepsilon),d)^2)^{\frac{p(k-1-l)}{2}}K_1(\delta(\varepsilon),d)^{\frac{pl}{2}}\\
\leq&(k-1)(\varepsilon K_1(\delta(\varepsilon),d)^2)^{\frac{p}{2}}K_1(\delta(\varepsilon),d)^{\frac{p(k-2)}{2}}\leq C(G,p,k)(\varepsilon K_1(\delta(\varepsilon),d)^k)^{\frac{p}{2}}.
\end{align*}
For $\mathcal{K}_4^{\varepsilon}$, by the induction hypothesis, we find that
\begin{align*}
\mathcal{K}_4^{\varepsilon}\leq&C(G,p,k)\sum_{l=2}^{k-2}\Big(\varepsilon^{(k-l)l-(k-1)+l}K_1(\delta(\varepsilon),d)^{l(k-l+1)}\Big)^{\frac{p}{2}}\\
\leq&C(G,p,k)\sum_{l=2}^{k-2}\Big(\varepsilon^{(k-l+1)l-k}K_1(\delta(\varepsilon),d)^{(k-l+1)l-k}\Big)^{\frac{p}{2}}\Big(\varepsilon K_1(\delta(\varepsilon),d)^{k}\Big)^{\frac{p}{2}}\\
\leq&C(G,p,k)\Big(\varepsilon K_1(\delta(\varepsilon),d)^{k}\Big)^{\frac{p}{2}}.
\end{align*}
We next focus on $\mathcal{K}_5^{\varepsilon}$. Due to H\"older's inequality and the inequality (\ref{p-inequality}), there exists $\varepsilon_0>0$ such that for every $\varepsilon\in(0,\varepsilon_0)$,
\begin{align*}
\mathcal{K}_5^{\varepsilon}=&\sum^{k-2}_{l=2}\varepsilon^{-\frac{p(k-1)}{2}}\Big(\sup_{t\in[0,T],y\in\mathbb{T}^d}\mathbb{E}\Big|\sum_{m=1}^{l-1}C_m^l\Big[\sum_{r=1}^{k-l}\varepsilon^{\frac{r}{2}}\bar{u}^{r,\delta(\varepsilon)}(t,y)\Big]^m\Big[\varepsilon^{\frac{k-l}{2}}w_{k-l}^{\varepsilon,\delta(\varepsilon)}(t,y)\Big]^{l-m}\Big|^{p}\Big)\\
\leq&C(p,k)\sum^{k-2}_{l=2}\varepsilon^{-\frac{p(k-1)}{2}}\\
&\cdot\Big(\sup_{t\in[0,T],y\in\mathbb{T}^d}\Big|\sum_{m=1}^{l-1}\Big(\mathbb{E}\Big|\sum_{r=1}^{k-l}\varepsilon^{\frac{r}{2}}\bar{u}^{r,\delta(\varepsilon)}(t,y)\Big|^{2mp}\Big)^{\frac{1}{2}}\Big(\mathbb{E}\Big|\varepsilon^{\frac{k-l}{2}}w_{k-l}^{\varepsilon,\delta(\varepsilon)}(t,y)\Big|^{2(l-m)p}\Big)^{\frac{1}{2}}\Big|\Big)\\
\leq&C(p,k)\sum^{k-2}_{l=2}\varepsilon^{-\frac{p(k-1)}{2}}\\
&\cdot\Big(\sup_{t\in[0,T],y\in\mathbb{T}^d}\Big|\sum_{m=1}^{l-1}\Big(\sum_{r=1}^{k-l}\varepsilon^{mpr}\mathbb{E}|\bar{u}^{r,\delta(\varepsilon)}(t,y)|^{2mp}\Big)^{\frac{1}{2}}\Big(\mathbb{E}\Big|\varepsilon^{\frac{k-l}{2}}w_{k-l}^{\varepsilon,\delta(\varepsilon)}(t,y)\Big|^{2(l-m)p}\Big)^{\frac{1}{2}}\Big|\Big).
\end{align*}
Combining with the induction hypothesis and Theorem~\ref{div-speed}, it follows that
\begin{align}\label{K5}
\mathcal{K}_5^{\varepsilon}\leq&C(G,p,k)\sum_{l=2}^{k-2}\varepsilon^{-\frac{p(k-1)}{2}}\sum_{m=1}^{l-1}\Big(\sum_{r=1}^{k-l}\varepsilon^{mpr}K_1(\delta(\varepsilon),d)^{mpr}\Big)^{\frac{1}{2}}\Big(\varepsilon^{\frac{k-l+1}{2}}K_1(\delta(\varepsilon),d)^{\frac{k-l+1}{2}}\Big)^{(l-m)p}\notag\\
=&C(G,p,k)\varepsilon^{-\frac{p(k-1)}{2}}\sum_{l=2}^{k-2}\varepsilon^{\frac{(k-l+1)lp}{2}}K_1(\delta(\varepsilon),d)^{\frac{(k-l+1)lp}{2}}\notag\\
&\sum_{m=1}^{l-1}\Big(\sum_{r=1}^{k-l}\varepsilon^{mpr}K_1(\delta(\varepsilon),d)^{mpr}\Big)^{\frac{1}{2}}\Big(\varepsilon^{\frac{k-l+1}{2}}K_1(\delta(\varepsilon),d)^{\frac{k-l+1}{2}}\Big)^{-mp}.
\end{align}
Due to the scaling regime $(\varepsilon,\delta(\varepsilon))$ we chose in (\ref{sca-k}), there exists $\varepsilon_0>0$, for every $\varepsilon\in(0,\varepsilon_0)$, we have
\begin{align*}
\varepsilon K_1(\delta(\varepsilon),d)<1,\ \ \ \ \ \varepsilon^{\frac{k-l}{2}}K_1(\delta(\varepsilon),d)^{\frac{k-l}{2}}<1.
\end{align*}
It follows that
\begin{align*}
&\sum_{m=1}^{l-1}\Big(\sum_{r=1}^{k-l}\varepsilon^{mpr}K_1(\delta(\varepsilon),d)^{mpr}\Big)^{\frac{1}{2}}\Big(\varepsilon^{\frac{k-l+1}{2}}K_1(\delta(\varepsilon),d)^{\frac{k-l+1}{2}}\Big)^{-mp}\\
\lesssim&\sum_{m=1}^{l-1}\varepsilon^{\frac{mp}{2}}K_1(\delta(\varepsilon),d)^{\frac{mp}{2}}\Big(\varepsilon^{\frac{k-l+1}{2}}K_1(\delta(\varepsilon),d)^{\frac{k-l+1}{2}}\Big)^{-mp}\\
=&\sum_{m=1}^{l-1}\Big(\varepsilon^{\frac{k-l}{2}}K_1(\delta(\varepsilon),d)^{\frac{k-l}{2}}\Big)^{-mp}\\
\lesssim&\Big(\varepsilon^{\frac{k-l}{2}}K_1(\delta(\varepsilon),d)^{\frac{k-l}{2}}\Big)^{-(l-1)p}.
\end{align*}
Combining with the estimate (\ref{K5}), we get
\begin{align*}
\mathcal{K}_5^{\varepsilon}\leq&C(G,p,k)\varepsilon^{-\frac{p(k-1)}{2}}\sum_{l=2}^{k-2}\varepsilon^{\frac{(k-l+1)lp}{2}}K_1(\delta(\varepsilon),d)^{\frac{(k-l+1)lp}{2}}\Big(\varepsilon^{\frac{k-l}{2}}K_1(\delta(\varepsilon),d)^{\frac{k-l}{2}}\Big)^{-(l-1)p}\\
=&C(G,p,k)\varepsilon^{-\frac{p(k-1)}{2}}\sum_{l=2}^{k-2}\varepsilon^{\frac{kp}{2}}K_1(\delta(\varepsilon),d)^{\frac{kp}{2}}\\
\leq&C(G,p,k)\Big(\varepsilon K_1(\delta(\varepsilon),d)^k\Big)^{\frac{p}{2}}.
\end{align*}
Finally, we consider the term $\mathcal{K}_6^{\varepsilon}$. For $k> l\geq0$,
with the help of H\"older's inequality for indices $(p,...,p_{k-l})$ with $\sum_{j=1}^{k-l}\frac{1}{p_j}=1$, the inequality (\ref{p-inequality}), the definitions of $\mathcal{J}^{\delta}(k,l)$, $\Lambda(k,l,m)$ and the induction hypothesis, we conclude that there exists an $\varepsilon_0>0$ such that for every $\varepsilon\in(0,\varepsilon_0)$,
\begin{align*}
\mathcal{K}_6^{\varepsilon}=&\sup_{t\in[0,T],y\in\mathbb{T}^d}\mathbb{E}\Big|\sum^{k-2}_{l=2}\sum_{m=k+1}^{(k-l)l+1}\varepsilon^{\frac{(m-k)}{2}}\sum_{(q_1,...,q_{k-l})\in\Lambda(k,l,m)}(\frac{l!}{q_1!...q_{k-l}!})\prod_{1\leq j\leq k-l}(\bar{u}^{j,\delta(\varepsilon)}(t,y))^{q_j}\Big|^p\\
\leq&C(p,k)\sup_{t\in[0,T],y\in\mathbb{T}^d}\sum^{k-2}_{l=2}\sum_{m=k+1}^{(k-l)l+1}\varepsilon^{\frac{(m-k)p}{2}}\\
&\cdot\Big(\sum_{(q_1,...,q_{k-l})\in\Lambda(k,l,m)}(\frac{l!}{q_1!...q_{k-l}!})^p\prod_{1\leq j\leq k-l}(\mathbb{E}|\bar{u}^{j,\delta(\varepsilon)}(t,y)|^{q_jp_jp})^{\frac{1}{p_j}}\Big)\\
\leq&C(G,p,k)\sum^{k-2}_{l=2}\sum_{m=k+1}^{(k-l)l+1}\varepsilon^{\frac{(m-k)p}{2}}\\
&\cdot\Big(\sum_{(q_1,...,q_{k-l})\in\Lambda(k,l,m)}(\frac{l!}{q_1!...q_{k-l}!})^p\prod_{1\leq j\leq k-l}K_1(\delta(\varepsilon),d)^{\frac{jq_jp}{2}}\Big)\\
\leq&C(G,p,k)\sum^{k-2}_{l=2}\sum_{m=k+1}^{(k-l)l+1}\varepsilon^{\frac{(m-k)p}{2}}K_1(\delta(\varepsilon),d)^{\frac{(m-1)p}{2}}\\
=&C(G,p,k)\sum^{k-2}_{l=2}\sum_{m=k+1}^{(k-l)l+1}\varepsilon^{\frac{(m-k-1)p}{2}}K_1(\delta(\varepsilon),d)^{\frac{(m-1-k)p}{2}}(\varepsilon K_1(\delta(\varepsilon),d)^k)^{\frac{p}{2}}\\
\leq&C(G,p,k)(\varepsilon K_1(\delta(\varepsilon),d)^{k})^{\frac{p}{2}}.
\end{align*}
Combining all the above estimates, by (\ref{1to6}), using the scaling regime (\ref{sca-k}), there exists $\varepsilon_0>0$, such that for every $\varepsilon\in(0,\varepsilon_0)$,
\begin{equation*}
\sup_{t\in[0,T],x\in\mathbb{T}^d}\mathbb{E}|w_k^{\varepsilon,\delta(\varepsilon)}(t,x)|^p\leq C(G,p,k)\Big(\varepsilon K_1(\delta(\varepsilon),d)^{k+1}\Big)^{\frac{p}{2}}.
\end{equation*}
Thus, (\ref{lln-1}) holds for $n=k$, $p\in[2,\infty)$. Moreover, H\"older's inequality implies that (\ref{lln-1}) holds for $n=k$, $p\in[1,2)$. This  completes the proof for the non-conservative case. 

{\bf Conservative case.} Using Lemma \ref{gen-littlewoodpaley} and Lemma \ref{maximalLp-cons} to see that  
\begin{align*}
\mathbb{E}\|w_{k}^{\varepsilon,\delta(\varepsilon)}\|_{L^p([0,T]\times\mathbb{T}^d)}^p\leq&C(p)K_2(\delta(\varepsilon),d)^{\frac{p}{2}}\mathbb{E}\Big\|I_1^{\varepsilon}+I_2^{\varepsilon}+I_3^{\varepsilon}+I_4^{\varepsilon}\Big\|_{L^p([0,T]\times\mathbb{T}^d)}^p,
\end{align*}
where $I_i^{\delta}$, $i=1,2,3,4$, are defined by (\ref{I1-4}) in the same formulation but with conservative coefficients $\bar{u}^{k,\delta}$ and solution $u^{\varepsilon,\delta}$ instead. Combining all the previous estimates, with the same procedure in the nonconservative case, we are able to see that 
\begin{align}\label{1to6-cons}
&\mathbb{E}\|w^{\varepsilon,\delta(\varepsilon)}_{k}\|_{L^p([0,T]\times\mathbb{T}^d)}^p
\leq C(G)K_2(\delta(\varepsilon),d)^{\frac{p}{2}}\Big\{\mathcal{K}_1^{\varepsilon}+\mathcal{K}_2^{\varepsilon}+\mathcal{K}_3^{\varepsilon}+\mathcal{K}_4^{\varepsilon}+\mathcal{K}_5^{\varepsilon}+\mathcal{K}_6^{\varepsilon}\Big\},
\end{align}
where
\begin{align*}
\mathcal{K}_1^{\varepsilon}:=&\varepsilon^{-\frac{(k-1)p}{2}}\mathbb{E}\|(w^{\varepsilon,\delta(\varepsilon)}_0)^k\|_
{L^p([0,T]\times\mathbb{T}^d)}^{p},\\
\mathcal{K}_2^{\varepsilon}:=&\mathbb{E}\Big\|\sum_{l=0}^{k-2}C^l_{k-1}(w_1^{\varepsilon,\delta(\varepsilon)})^{k-1-l}(\bar{u}^{1,\delta(\varepsilon)})^l\Big\|_{L^p([0,T]\times\mathbb{T}^d)}^{p},\\
\mathcal{K}_3^{\varepsilon}:=&\mathbb{E}\|w^{\varepsilon,\delta(\varepsilon)}_{k-1}\|_{L^p([0,T]\times\mathbb{T}^d)}^{p},\\
\mathcal{K}_4^{\varepsilon}:=&\sum^{k-2}_{l=2}\varepsilon^{\frac{(k-l)l-(k-1)}{2}p}\mathbb{E}\|(w^{\varepsilon,\delta(\varepsilon)}_{k-l})^l\|_{L^p([0,T]\times\mathbb{T}^d)}^{p},\\
\mathcal{K}_5^{\varepsilon}:=&\sum^{k-2}_{l=2}\varepsilon^{-\frac{p(k-1)}{2}}\Big(\mathbb{E}\Big\|\sum_{m=1}^{l-1}C_m^l\Big[\sum_{r=1}^{k-l}\varepsilon^{\frac{r}{2}}\bar{u}^{r,\delta(\varepsilon)}\Big]^m\Big[\varepsilon^{\frac{k-l}{2}}w_{k-l}^{\varepsilon,\delta(\varepsilon)}\Big]^{l-m}\Big\|_{L^p([0,T]\times\mathbb{T}^d)}^{p}\Big),\\
\mathcal{K}_6^{\varepsilon}:=&\mathbb{E}\Big\|\sum^{k-2}_{l=2}\sum_{m=k+1}^{(k-l)l+1}\varepsilon^{\frac{m-k}{2}}\sum_{(q_1,...,q_{k-l})\in\Lambda(k,l,m)}\Big(\frac{l!}{q_1!...q_{k-l}!}\Big)\prod_{1\leq j\leq k-l}(\bar{u}^{j,\delta(\varepsilon)})^{q_j}\Big\|_{L^p([0,T]\times\mathbb{T}^d)}^p.
\end{align*}
Then we employ the same argument as in the proof of the non-conservative case, replacing the supremum $\sup_{s\in [0,T],x\in\mathbb{T}^d}$ by the integration over $[0,T]\times\mathbb{T}^d$. Furthermore, compare to the proof of the non-conservative case, applications of H\"older's inequality with respect to integration $\int_{\Omega}d\mathbb{P}$ are  replaced by applications with respect to $\int_0^T\int_{\mathbb{T}^d}\int_{\Omega}d\mathbb{P}dxdt$. This implies that there exists $\varepsilon_0>0$, such that for every $\varepsilon\in(0,\varepsilon_0)$, we have  
\begin{equation*}
\mathcal{K}^{\varepsilon}_i\leq C(G,p,k)\Big(\varepsilon K_2(\delta(\varepsilon),d)\Big)^{\frac{p}{2}}	
\end{equation*}
holds for $i=1,...,6$. Consequently, it follows that 
\begin{equation*}
\mathbb{E}\|w_k^{\varepsilon,\delta(\varepsilon)}\|_{L^p([0,T]\times\mathbb{T}^d)}^p\leq C(G,p,k)\Big(\varepsilon K_2(\delta(\varepsilon),d)^{k+1}\Big)^{\frac{p}{2}}.
\end{equation*}
This completes the proof of the conservative case. 
\end{proof}

\section{Higher order fluctuations for non-smooth coefficients}\label{sec-7}

In this section, we show that the higher order fluctuation expansions hold for the local in time solutions of (\ref{SHE02}) and (\ref{SHE01}) with a non-smooth coefficient $G$. As prototypes, this section considers the coefficients $G(\zeta)=\sqrt{\zeta}$ and $G(\zeta)=\sqrt{\zeta(1-\zeta)}$. Consequently, the corresponding results can be applied to the Dean-Kawasaki equation, SSEP, Dawson-Watanabe equation, and Fleming-Viot equation.

\begin{theorem}\label{sbm2}
Assume that Hypothesis H3 holds for the initial data $u_0$ and the diffusion coefficient $G$ , let $\gamma$ be a fixed suitable constant that appears in Hypothesis H3. Let $\varepsilon, \delta>0$, $n\in\mathbb{N}_+$. In the conservative case, assume that $\varepsilon,\delta$ satisfy (\ref{varepsilonsmall}).  Let $(u^{\varepsilon,\delta},\tau^{\varepsilon,\delta}_{\gamma})$ be the local in time mild solution of  (\ref{SHE02}) (resp.  (\ref{SHE01})), where $\tau^{\varepsilon,\delta}_{\gamma}$ is the $\{\mathcal{F}(t)\}_{t\in[0,T]}$-stopping time defined by 
\begin{equation}\label{stoppingtime-new}
\tau^{\varepsilon,\delta}_{\gamma}:=\inf\Big\{t\in[0,T];{\rm{ess\sup}}_{x\in\mathbb{T}^d}u^{\varepsilon,\delta}(t,x)>K+\gamma,\text{or }{\rm{ess\inf}}_{x\in\mathbb{T}^d}u^{\varepsilon,\delta}(t,x)<K'-\gamma\Big\}\wedge T. 
\end{equation}
Let $K_i(\delta,d), i=1,2$ be defined by (\ref{Kdelta}). \begin{description}
\item[(i)] ({\rm{Non-conservative\ noise}})\quad Assume that 
\begin{equation}\label{localscale-3}
\lim_{\varepsilon\rightarrow0}\varepsilon(\delta(\varepsilon)^{-d}+K_1(\delta(\varepsilon),d)^{n+1})=0.
\end{equation}
Then for almost every $(t,x)\in[0,T]\times\mathbb{T}^d$,
\begin{align*}
\Big|\varepsilon^{-\frac{n}{2}}\Big(u^{\varepsilon,\delta(\varepsilon)}-\sum_{i=0}^{n}\varepsilon^{\frac{i}{2}}\bar{u}^{i,\delta(\varepsilon)}\Big)(t,x)\Big|\rightarrow0,
\end{align*}
in probability, as $\varepsilon\rightarrow0$.

  \item[(ii)] ({\rm{Conservative\ noise}})\quad Assume that 
\begin{equation}\label{localscale-4}
\lim_{\varepsilon\rightarrow0}\varepsilon(\delta(\varepsilon)^{-d-2}+K_2(\delta(\varepsilon),d)^{n+1})=0.
\end{equation}
Then for every $p\in[1,+\infty)$, 
\begin{align*}
\Big\|\varepsilon^{-\frac{n}{2}}\Big(u^{\varepsilon,\delta(\varepsilon)}-\sum_{i=0}^{n}\varepsilon^{\frac{i}{2}}\bar{u}^{i,\delta(\varepsilon)}\Big)\Big\|_{L^p([0,T]\times\mathbb{T}^d)}\rightarrow0,
\end{align*}
in probability, as $\varepsilon\rightarrow0$.
\end{description}
\end{theorem}
\begin{proof}
Let $\rho^{\varepsilon,\delta}$ be the global in time variational solution (resp.\ mild solution) of the approximating equation (\ref{approxeq}) (resp.  (\ref{approxeq-2})) on $[0,T]$, where the smooth diffusion coefficient $G_0$ is defined by (\ref{G0}).
 From the local in time uniqueness of  (\ref{SHE01}) and  (\ref{SHE02}) in Corollary \ref{localsolution}, we have that for every $\delta>0$, let $\varepsilon_0=2\|G_0\|_{L^{\infty}(\mathbb{R})}^{-2}\delta^{-d}$, then for every $\varepsilon\in(0,\varepsilon_0)$, we have $\mathbb{P}$-almost surely, 
\begin{align}\label{coincide}
\rho^{\varepsilon,\delta}(t,x)=u^{\varepsilon,\delta}(t,x)\ 	{\text{for }}a.e. \ x\in\mathbb{T}^d,\ \text{for every }t<\tau^{\varepsilon,\delta}_{\gamma}.
\end{align}
We define 
\begin{align}
\bar{w}^{\varepsilon,\delta}_n:=\varepsilon^{-\frac{n}{2}}\Big(\rho^{\varepsilon,\delta}-\sum_{i=0}^{n}\varepsilon^{\frac{i}{2}}\bar{\rho}^{i,\delta}\Big),\label{wnn-2}
\end{align}
where in the non-conservative case, $\bar{\rho}^{i,\delta}$ is the mild solution of (\ref{coe}) with $G$ replaced by $G_0$, and in the  conservative noise, $\bar{\rho}^{i,\delta}$ is the mild solution of (\ref{ccoe}) with $G$ replaced by $G_0$. Due to the fact that 
\begin{equation*}
	\bar{u}^{0,\delta}(t,x)=p(t)\ast u_0\in\Big[{\rm{ess\inf}}u_0-\gamma,{\rm{ess\sup}}u_0+\gamma\Big],
\end{equation*}
together with the fact that $G_0(\zeta)=G(\zeta)$ for $\zeta\in\Big[{\rm{ess\inf}}u_0-\gamma,{\rm{ess\sup}}u_0+\gamma\Big]$, it follows that $\bar{\rho}^{i,\delta}=\bar{u}^{i,\delta}$, for $i=1,...,n-1$. 
Let $\delta>0$, and let $w^{\varepsilon,\delta}_n$ be defined by (\ref{wnn}). Thanks to  (\ref{coincide}),  we are able to see that there exists $\varepsilon_0>0$ such that for every $\varepsilon\in(0,\varepsilon_0)$, $n\in\mathbb{N}$, we have $\mathbb{P}$-almost surely, 
\begin{align}\label{coincide-2}
\bar{w}^{\varepsilon,\delta}_n(t,x)=w^{\varepsilon,\delta}_n(t,x)\ 	{\text{for }}a.e.\ x\in\mathbb{T}^d,\ {\text{for every }}t<\tau^{\varepsilon,\delta}_{\gamma}.
\end{align}
Regarding the nonconservative case, by the definition of $\bar{w}^{\varepsilon,\delta}_n$ by (\ref{wnn-2}), for almost every $(t,x)\in[0,T]\times\mathbb{T}^d$, and for $\gamma'>0$, it follows that 
\begin{align*}
&\mathbb{P}\Big(|w_n^{\varepsilon,\delta}(t,x)|>\gamma'\Big)\\
=&\mathbb{P}\Big(|w_n^{\varepsilon,\delta}(t,x)|>\gamma',\tau^{\varepsilon,\delta}_{\gamma}> t\Big)+\mathbb{P}\Big(|w_n^{\varepsilon,\delta}(t,x)|>\gamma',\tau^{\varepsilon,\delta}_{\gamma}\leq t\Big)\\
\leq&\mathbb{P}\Big(|\bar{w}_n^{\varepsilon,\delta}(t,x)|>\gamma'\Big)+\mathbb{P}\Big(\tau^{\varepsilon,\delta}_{\gamma}\leq t\Big).
\end{align*}
Regarding the conservative case, it follows that 
\begin{align*}
&\mathbb{P}\Big(\|w_n^{\varepsilon,\delta}\|_{L^p([0,T]\times\mathbb{T}^d)}>\gamma'\Big)\\
=&\mathbb{P}\Big(\|w_n^{\varepsilon,\delta}\|_{L^p([0,T]\times\mathbb{T}^d)}>\gamma',\tau^{\varepsilon,\delta}_{\gamma}> T'\Big)+\mathbb{P}\Big(\|w_n^{\varepsilon,\delta}\|_{L^p([0,T]\times\mathbb{T}^d)}>\gamma',\tau^{\varepsilon,\delta}_{\gamma}\leq T'\Big)\\
\leq&\mathbb{P}\Big(\|\bar{w}_n^{\varepsilon,\delta}\|_{L^p([0,T]\times\mathbb{T}^d)}>\gamma'\Big)+\mathbb{P}\Big(\tau^{\varepsilon,\delta}_{\gamma}\leq T'\Big).
\end{align*}
Since $\bar{w}^{\varepsilon,\delta}_n$ is the remainder for the fluctuation expansion of $\rho^{\varepsilon,\delta}$, it follows from Theorem \ref{n-CLT} and Chebyshev's inequality that
\begin{equation*}
\lim_{\varepsilon\rightarrow0}\mathbb{P}(|\bar{w}_n^{\varepsilon,\delta(\varepsilon)}(t,x)|>\gamma')=0,\ \text{for almost every } (t,x)\in[0,T]\times\mathbb{T}^d, 
\end{equation*}
for the nonconservative case, and 
\begin{equation*}
\lim_{\varepsilon\rightarrow0}\mathbb{P}(\|\bar{w}_n^{\varepsilon,\delta(\varepsilon)}\|_{L^p([0,T]\times\mathbb{T}^d)}>\gamma')=0, 
\end{equation*}
for the conservative case. Furthermore, combining with Lemma \ref{stopproperty}, we have that
\begin{equation*}
\lim_{\varepsilon\rightarrow0}\Big[\mathbb{P}(\tau^{\varepsilon,\delta(\varepsilon)}_{\gamma}\leq T)+\mathbb{P}(\tau^{\varepsilon,\delta(\varepsilon)}_{\gamma}\leq t)\Big]=0, 
\end{equation*}
for every $t\in[0,T]$. This completes the proof. 
\end{proof}

\section{Applications}\label{sec-8}

In this section, we present several applications of Theorem \ref{sbm2} to interacting particle systems. The result of Theorem \ref{sbm2} for non-conservative noise applied to Dawson-Watanabe equation and the Fleming-Viot process yields the following results.

\begin{example}[Correlated Dawson-Watanabe equation] 
Let $d\geq1$, $n\in\mathbb{N}_+$. Assume that there is an $a>0$ so that $u_0\in L^{\infty}(\mathbb{T}^d;[a,\infty))$. Let $\varepsilon,\delta>0$, and $W^{\delta}$ be defined by (\ref{smooth-1}). Consider the correlated Dawson-Watanabe equation
\begin{equation}\label{SBM-2}
du^{\varepsilon,\delta}=\Delta u^{\varepsilon,\delta}dt+\varepsilon^{\frac{1}{2}}\sqrt{u^{\varepsilon,\delta}}dW^{\delta}(t),\quad u^{\varepsilon,\delta}(0)=u_0, 
\end{equation}
where $W^{\delta}$ is an infinite dimensional Brownian motion with spatial correlation length $\delta$. Then for every $\gamma\in(0,a)$, there exists a unique local in time mild solution $(u^{\varepsilon,\delta}(t))_{t\in[0,\tau^{\varepsilon,\delta}_{\gamma}]}$ with $\tau^{\varepsilon,\delta}_{\gamma}$ defined by (\ref{stoppingtime-new}). Let $K_i(\delta,d), i=1,2$ be defined by (\ref{Kdelta}), and assume that 
\begin{align}\label{scale-new-1}
\lim_{\varepsilon\rightarrow0}\varepsilon(\delta(\varepsilon)^{-d}+K_1(\delta(\varepsilon),d)^{n+1})=0. 
\end{align}
Then for every $(t,x)\in[0,T]\times\mathbb{T}^d$,
 $u^{\varepsilon,\delta}$ satisfies
\begin{align}\label{res-2}
\Big|\varepsilon^{-\frac{n}{2}}\Big(u^{\varepsilon,\delta(\varepsilon)}-\sum_{i=0}^{n}\varepsilon^{\frac{i}{2}}\bar{u}^{i,\delta(\varepsilon)}\Big)(t,x)\Big|\rightarrow0,
\end{align}
in probability, as $\varepsilon\rightarrow0$. 
\end{example}

\begin{example}[Correlated Fleming-Viot equation]
Let $d\geq1$, $n\in\mathbb{N}_+$. Assume that there is $a>0$ so that $u_0\in L^{\infty}(\mathbb{T}^d;[a,\infty))$.  Let $\varepsilon,\delta>0$, and $W^{\delta}$ be defined by (\ref{smooth-1}). Consider the correlated Fleming-Viot SPDE 
\begin{equation*}
du^{\varepsilon,\delta}=\Delta u^{\varepsilon,\delta}dt+\varepsilon^{\frac{1}{2}}\sqrt{u^{\varepsilon,\delta}(1-u^{\varepsilon,\delta})}dW^{\delta}(t),\quad
u^{\varepsilon,\delta}(0)=u_0.
\end{equation*}
Then for every $\gamma\in(0,a)$, there exists a unique local in time mild solution $(u^{\varepsilon,\delta},\tau^{\varepsilon,\delta}_{\gamma})$ with $\tau^{\varepsilon,\delta}_{\gamma}$ defined by (\ref{stoppingtime-new}). Assume that $(\varepsilon,\delta(\varepsilon))$ satisfies (\ref{scale-new-1}). Then $u^{\varepsilon,\delta}$ satisfies the expansion formula (\ref{res-2}).
\end{example}

The result of Theorem \ref{sbm2} for conservative noise applied to the symmetric simple exclusion process and Dean-Kawasaki equation yields the following results. 

\begin{example}[Symmetric simple exclusion process]
Let $d\geq1$, $n\in\mathbb{N}_+$. Assume that there is $a\in(0,\frac{1}{2})$ so that $u_0\in L^{\infty}(\mathbb{T}^d;[a,1-a])$. Let $\varepsilon,\delta>0$, and let $W^{\delta}$ be defined by (\ref{smooth-1}). Consider the solution to the fluctuating continuum model for the symmetric simple exclusion process 
\begin{equation}\label{SSEP-2}
du^{\varepsilon,\delta}=\Delta u^{\varepsilon,\delta}dt+\varepsilon^{\frac{1}{2}}\nabla\cdot\Big(\sqrt{u^{\varepsilon,\delta}(1-u^{\varepsilon,\delta})}dW^{\delta}(t)\Big),\quad u^{\varepsilon,\delta}(0)=u_0. 
\end{equation}
Then for every $\gamma\in(0,a)$, let $\varepsilon,\delta$ satisfy (\ref{varepsilonsmall}), there exists a unique local in time mild solution $(u^{\varepsilon,\delta},\tau^{\varepsilon,\delta}_{\gamma})$ with $\tau^{\varepsilon,\delta}_{\gamma}$ defined by (\ref{stoppingtime-new}). Assume that 
\begin{align}\label{scale-new-2}
\lim_{\varepsilon\rightarrow0}\varepsilon (\delta(\varepsilon)^{-d-2}+\delta(\varepsilon)^{-(n+1)d})=0.
\end{align}
Then $u^{\varepsilon,\delta}$ satisfies 
\begin{align}\label{res-2-cons}
\Big\|\varepsilon^{-\frac{n}{2}}\Big(u^{\varepsilon,\delta(\varepsilon)}-\sum_{i=0}^{n}\varepsilon^{\frac{i}{2}}\bar{u}^{i,\delta(\varepsilon)}\Big)\Big\|_{L^p([0,T]\times\mathbb{T}^d)}\rightarrow0,
\end{align}
in probability, as $\varepsilon\rightarrow0$. 
\end{example}

\begin{example}[Correlated Dean-Kawasaki equation]
Let $d\geq1$, $n\in\mathbb{N}_+$. Assume that there is $a\in(0,\frac{1}{2})$ so that $u_0\in L^{\infty}(\mathbb{T}^d;[a,1-a])$.  Let $\varepsilon,\delta>0$, and $W^{\delta}$ be defined by (\ref{smooth-1}). Consider the correlated Dean-Kawasaki equation
\begin{equation*}
du^{\varepsilon,\delta}=\Delta u^{\varepsilon,\delta}dt+\varepsilon^{\frac{1}{2}}\nabla\cdot(\sqrt{u^{\varepsilon,\delta}}dW^{\delta}(t)),\quad u^{\varepsilon,\delta}(0)=u_0. 
\end{equation*}
Then for every $\gamma\in(0,a)$, let $\varepsilon,\delta$ satisfy (\ref{varepsilonsmall}), there exists a unique local in time mild solution $(u^{\varepsilon,\delta},\tau^{\varepsilon,\delta}_{\gamma})$ with $\tau^{\varepsilon,\delta}_{\gamma}$ defined by (\ref{stoppingtime-new}). Assume that $(\varepsilon,\delta(\varepsilon))$ satisfies (\ref{scale-new-2}). Then $u^{\varepsilon,\delta}$ satisfies the expansion formula (\ref{res-2-cons}).
\end{example}

\noindent{\bf  Acknowledgements}\quad Benjamin Gess acknowledges support by the Max Planck Society through the Research Group ``Stochastic Analysis in the Sciences". This work was funded by the Deutsche Forschungsgemeinschaft (DFG, German Research Foundation) via IRTG 2235 - Project Number 282638148, and cofunded by the European Union (ERC, FluCo, grant agreement No. 101088488). Views and opinions expressed are however those of the author(s)only and do not necessarily reflect those of the European Union or of the European Research Council. Neither the European Union nor the granting authority can be held responsible for them. Rangrang Zhang acknowledges support by Beijing Natural Science Foundation (No. 1212008), National Natural Science Foundation of China (No. 12171032), Beijing Institute of Technology Research Fund Program for Young Scholars and MIIT Key Laboratory of Mathematical Theory, Computation in Information Security.

We thank Vitalii Konarovskyi for discussion and for carefully reading a preliminary version of the manuscript.

\bibliographystyle{alpha}
\bibliography{expansionlib}

\begin{flushleft}
\small \normalfont
\textsc{Benjamin Gess\\
Faculty of Mathematics, University of Bielefeld\\
33615 Bielefeld, Germany \\
and\\
Max--Planck--Institut f\"ur Mathematik in den Naturwissenschaften\\
04103 Leipzig, Germany.}\\
\texttt{\textbf{bgess@math.uni-bielefeld.de}}
\end{flushleft}

\begin{flushleft}
\small \normalfont
\textsc{Zhengyan Wu\\
Department of Mathematics, University of Bielefeld, D-33615 Bielefeld, Germany.}\\
\texttt{\textbf{zwu@math.uni-bielefeld.de}}
\end{flushleft}

\begin{flushleft}
\small \normalfont
\textsc{Rangrang Zhang\\
School of Mathematics and Statistics,
Beijing Institute of Technology, Beijing, 100081, China.}\\
\texttt{\textbf{rrzhang@amss.ac.cn}}
\end{flushleft}

\end{document}